\documentclass[10pt]{amsart}
\usepackage{amsmath, amscd, amsthm, amssymb, graphicx, mathrsfs, setspace,fancyhdr,geometry,tabularx,shapepar, hyperref, xcolor, tikz, bbold, mathtools}
 \usepackage[stable]{footmisc}
 \usepackage{tikz-cd}
 \usepackage{quiver}

\usepackage{tikz}
\usepackage{pgfplots}
 
\pgfplotsset{
    compat=newest
}
 
 \usepackage{enumitem, xcolor, tikz, wrapfig,lipsum}
 \usetikzlibrary{decorations.pathreplacing,calc}
\usetikzlibrary{decorations.pathmorphing,backgrounds,positioning,fit,matrix}
\usepackage[all]{xy}
\hypersetup{colorlinks=false}
\usepackage{csquotes}
\usepackage{colortbl}
\usetikzlibrary{positioning}

\theoremstyle{plain}
\newtheorem{thm}{Theorem}[section]

\newtheorem{lem}[thm]{Lemma}
\newtheorem{cor}[thm]{Corollary}
\newtheorem{prop}[thm]{Proposition}

\newtheorem*{thm*}{Theorem}
\newtheorem*{prop*}{Proposition}
\newtheorem*{problem*}{Problem}

\theoremstyle{definition}
\newtheorem{defn}[thm]{Definition}

\newtheorem{rem}[thm]{Remark}

\newtheorem{ex}[thm]{Example}

\numberwithin{equation}{section}

\newcommand{\R}{{\mathbb R}}

\newcommand{\deloop}{\mathsf{B}}
\newcommand{\set}[1]{\mathcal{#1}}

\newcommand{\cat}[1]{\mathsf{#1}}
\newcommand{\LPC}[1]{\mathfrak{#1}}

\newcommand{\weight}{\mathrm{W}}

\newcommand{\DeltaYY}{\Delta^{\cat{Y}}}

\begin{document}

\title{Interleaving Distances, Monoidal Actions and 2-Categories} 

\author{Patrick K. McFaddin and Tom Needham}

\begin{abstract}
    Interleaving distances are used widely in Topological Data Analysis (TDA) as a tool for comparing topological signatures of datasets. The theory of interleaving distances has been extended through various category-theoretic constructions, enabling its usage beyond standard constructions of TDA, while clarifying certain observed stability phenomena by unifying them under a common framework. Inspired by metrics used in the field of statistical shape analysis, which are based on minimizing energy functions over group actions, we define three new types of increasingly general interleaving distances. Our constructions use ideas from the theories of monoidal actions and 2-categories. We show that these distances naturally extend the category with a flow framework of de Silva, Munch and Stefanou and the locally persistent category framework of Scoccola, and we provide a general stability result. Along the way, we give examples of distances that fit into our framework which connect to ideas from differential geometry, geometric shape analysis, statistical TDA and multiparameter persistent homology.
\end{abstract}

\maketitle

\setcounter{tocdepth}{1}
\tableofcontents

\section{Introduction}

Interleaving distances are used in the field of Topological Data Analysis (TDA) to compare certain topological invariants extracted from datasets~\cite{chazal2009proximity}. Most commonly, these invariants take the form of persistence modules. A \emph{persistence module} $M$ associates to each real number $t \in \R$ a vector space $M(t)$ and to each pair of numbers $s,t \in \R$ with $s \leq t$, a linear map $M(s \leq t): M(s) \to M(t)$, and this data satisfies certain natural axioms (see Section \ref{sec:interleaving_group_actions} for a precise definition); in practice, a persistence module typically encodes the multiscale clustering (or higher-order topological) structure of a dataset, so that the interleaving distance between two such modules gives a quantitative comparison of the topological structures of the underlying datasets. We say that persistence modules $M$ and $N$ are \emph{$t$-interleaved} if there are linear maps $\phi_s:M(s) \to N(s+t)$ and $\psi_s:N(s) \to M(s+t)$ for all $s \in \R$ such that the following diagrams commute:
\begin{equation}\label{eqn:basic_interleaving_distance_1}\begin{tikzcd}
	{M(r)} && {M(s)} \\
	{N(r+t)} && {N(s+t)}
	\arrow["{M(r \leq s)}", from=1-1, to=1-3]
	\arrow["{\phi_r}"', from=1-1, to=2-1]
	\arrow["{\phi_s}", from=1-3, to=2-3]
	\arrow["{N(r+t \leq s+t)}"', from=2-1, to=2-3]
\end{tikzcd}
\begin{tikzcd}
	{N(r)} && {N(s)} \\
	{M(r+t)} && {M(s+t)}
	\arrow["{N(r \leq s)}", from=1-1, to=1-3]
	\arrow["{\psi_r}"', from=1-1, to=2-1]
	\arrow["{\psi_s}", from=1-3, to=2-3]
	\arrow["{M(r+t \leq s+t)}"', from=2-1, to=2-3]
\end{tikzcd}\end{equation}
\begin{equation}\label{eqn:basic_interleaving_distance_2}\begin{tikzcd}
	{M(s)} && {M(s+2t)} \\
	& {N(s+t)}
	\arrow["{M(s \leq s+2t)}", from=1-1, to=1-3]
	\arrow["{\phi_s}"', from=1-1, to=2-2]
	\arrow["{\psi_{s+t}}"', from=2-2, to=1-3]
\end{tikzcd}
\begin{tikzcd}
	{N(s)} && {N(s+2t)} \\
	& {M(s+t)}
	\arrow["{M(s \leq s+2t)}", from=1-1, to=1-3]
	\arrow["{\psi_s}"', from=1-1, to=2-2]
	\arrow["{\phi_{s+t}}"', from=2-2, to=1-3]
\end{tikzcd}\end{equation}
The \emph{interleaving distance} between $M$ and $N$ is then defined to be the infimum over $t \geq 0$ such that they are $t$-interleaved. This definition has a natural category-theoretic interpretation, and this interpretation has led to several generalizations of the concept, allowing for comparison of a wide variety of persistence module-like objects~\cite{bubenik2014categorification,bubenik2015metrics,bubenik2017interleaving,de2018theory,stefanou2018dynamics,scoccola2020locally} (many of these generalized constructions are described in detail throughout the paper). The category-theoretic viewpoint also allows for conceptually simple proofs of important theoretical properties of TDA such as the stability of invariants---roughly, if two datasets are similar with respect to, say, Hausdorff distance, then their topological signatures are guaranteed to be similar with respect to interleaving distance.

This paper constructs further generalizations of the notion of interleaving distance, inspired by ideas rooted in differential geometry and statistical shape analysis. Our main motivation was an analogy between interleaving distances and metrics which arise through group actions: namely, suppose that a group $\set{G}$ endowed with a left-invariant metric $d_\set{G}$ acts transitively from the left on a topological space $\set{X}$; we can then define the distance between points $x$ and $y$ in $\set{X}$ as 
\begin{equation}\label{eqn:group_distance}
\inf_{g \in \set{G}} \{d_\set{G}(e,g) \mid gx = y\},
\end{equation}
where $e$ denotes the identity element of $\set{G}$. It is straightforward to show that this defines a pseudometric, and that it defines a true metric under further mild assumptions on the action. Such metrics arise in the context of homogeneous spaces, where they are geodesic distances with respect to natural Riemannian metrics, and in medical imaging and shape analysis, where the space $\set{X}$ represents a collection of medical images and the group $\set{G}$ is a group of diffeomorphisms of the image domain which acts by deforming the images~\cite{grenander1996elements,bauer2014overview} (see Section \ref{subsect:examples_of_interleaving_distances} for more details). 

Consider the following interpretation of the distance defined in \eqref{eqn:group_distance}: if the quantity $d_\set{G}(e,g)$ is taken to represent the ``energy" of the group element $g \in \set{G}$, then \eqref{eqn:group_distance} measures the distance between points $x$ and $y$ by determining the minimum energy required to ``align" $x$ to $y$ via the group action. It is natural to consider a relaxation of this idea which doesn't require a perfect alignment---for example, such a relaxation is necessary to obtain a meaningful distance if the group action is not transitive. The interleaving distance between persistence modules can be viewed as such a relaxation. Indeed, the group in question is $\R$ (under addition), and a real number $t$ acts on a persistence module $M$ to produce the persistence module $t \cdot M$ with $(t \cdot M)(s) = M(s + t)$. The computation of interleaving distance between modules $M$ and $N$ can be viewed as finding a ``rough alignment" between $M$ and $t \cdot N$, and between $N$ and $t \cdot M$, consisting of the collections of linear maps $(\phi_s)_{s \in \R}$ and $(\psi_s)_{s \in \R}$, respectively. The diagrams \eqref{eqn:basic_interleaving_distance_1} can be understood to say that these collections of linear maps are consistent with the persistence structures of $M$ and $N$ (so that the collections are instances of an appropriate notion of morphism in the category of persistence modules), while the diagrams \eqref{eqn:basic_interleaving_distance_2} tell us that $M$ and $t \cdot N$ are roughly isomorphic, in the sense that $\phi_s$ and $\psi_{s + t}$ are close to being inverses (likewise for $N$ and $t \cdot M$). The energy of a group element in this setting is simply $|t|$, so that we may finally interpret the interleaving distance between $M$ and $N$ to be the minimum energy required to find a rough alignment between them.

The goal of this paper is to identify a framework for interleaving distances which captures the notion of the energy of a rough alignment, as described above, and which generalizes the notions of interleaving that have appeared previously in the literature. To achieve this, we use ideas from the theories of \emph{actegories} (actions of monoidal categories) and \emph{2-categories} (categories endowed with ``morphisms between morphisms"). 

\subsection{Main Results} Let us now summarize our main results. They are stated here somewhat informally, but are given with pointers to the precise results in the main body of the paper.

\smallskip
\paragraph{{\bf New Constructions of Interleaving Distances.}} For an arbitrary small category $\cat{X}$, let $\cat{X}_0$ denote its set of objects. We define three increasingly general notions of interleaving distance on $\cat{X}_0$, with each notion defined in terms of different forms of auxiliary data. We sketch the ideas here, with proper definitions provided in the body of the paper; in particular, we give an overview of basic 2-category theory in Section \ref{sec:background} for readers who are unfamiliar with this terminology.
\begin{itemize}
    \item When $\cat{X}$ is endowed with an action of a monoidal category $\cat{G}$ (i.e., a monoidal functor from $\cat{G}$ to the monoidal category of endomorphisms of $\cat{X}$), and when the objects of $\cat{G}$ are endowed with a ``weight" or ``energy" function satisfying certain simple axioms (called a \emph{monoidal weight}---see Definition \ref{def:monoidal_weight}), one obtains an associated as \emph{$\cat{G}$-interleaving distance} on $\cat{X}_0$ (Definition \ref{def:interleaving_distance_actegory}).
    \item Given a 2-functor from a 2-category $\cat{C}$ to the 2-category of small categories $\cat{Cat}$ such that $\cat{X}$ is the join of the categories in its image, together with a Lawvere weight on $\cat{C}$ (Definition \ref{def:lawvere_weight}; see also~\cite{lawvere1973metric,bubenik2017interleaving}), one obtains an associated \emph{$2$-functor interleaving distance} on $\cat{X}_0$ (Definition \ref{def:interleaving_distance}).  
    \item Suppose that $\cat{X}$ is a 2-category. When endowed with a \emph{Lawvere 2-weight} (Definition \ref{def:lawvere_2_weight}), one obtains an associated \emph{2-weighted 2-category interleaving distance} on $\cat{X}_0$ (Definition \ref{def:2-weighted-2-category-interleaving}). 
\end{itemize}

We show that these interleaving distances are true distances in the following sense.

\begin{thm*}[Theorems \ref{thm:actegory_interleaving}, \ref{thm:interleaving_distance} and \ref{thm:interleaving_weighted_2_category}]
    All three versions of interleaving distance described above define extended pseudometrics.
\end{thm*}

We also show that these notions generalize constructions of interleaving distance which have appeared previously in the literature.

\begin{thm*}[Propositions \ref{prop:equivalence_to_flow} and  \ref{prop:bubenik_de_silva_scott_interleaving} and Theorems \ref{thm:2-Functors_to_2-categories} and \ref{thm:LPC_to_2_functor}]
    The interleaving distances described above include the notions of generalized interleaving distance~\cite{bubenik2015metrics}, interleaving distance on a category with a flow~\cite{de2018theory} and interleaving distance on a locally persistent category~\cite{scoccola2020locally}.
\end{thm*}

In particular, the $\cat{G}$-interleaving distance construction is directly inspired by the construction of interleaving distance on a category with a flow~\cite{de2018theory,stefanou2018dynamics}---when the underlying group $\set{G}$ is $\R$, one immediately recovers the latter notion of interleaving. The $\cat{G}$-interleaving distances are arguably the most intuitive construction in this paper; most of the examples we provide are of this type, and this is the most natural way to connect interleaving distance with metrics defined in terms of group actions, as in \eqref{eqn:group_distance}. On the other hand, the 2-functor interleaving distances and 2-weighted 2-category interleaving distances offer a level of generality that is sometimes simpler to work in. We note that there is a brief discussion in the conclusion of \cite{de2018theory} on the interpretation of a category with a flow as a 2-functor, but details are not developed therein. The 2-weighted 2-category formalism was inspired by the locally persistent category framework of \cite{scoccola2020locally}; the difference is that our framework is based on 2-categories (informally, categories enriched in categories) rather than categories enriched in persistent sets.

Let us remark here that the category with a flow framework~\cite{de2018theory} is developed in the more general setting of lax monoidal functors. We have opted to work in the strict setting throughout the paper. This choice was made to keep the exposition simpler, and because the examples we are interested in fit into the strict functor formalism. We expect that most of our constructions and results can be extended to the lax setting, given motivation from a compelling example.

\smallskip
\paragraph{{\bf Connections to Shape Analysis and Persistent Homology.}} We give several examples of specific interleaving distances which are interesting from the perspective of TDA or other areas of mathematics. Firstly, we accomplish our goal of showing that the group action distance \eqref{eqn:group_distance} can be viewed as an interleaving distance---more specifically, we prove the following.

\begin{prop*}[Proposition \ref{prop:group_action_on_set}]
    Let $\set{G}$ be a group with left-invariant metric $d_\set{G}$ and left action on a set $\set{X}$. There is a monoidal category $\cat{G}$, a category $\cat{X}$ with $\cat{X}_0 = \set{X}$ and an action of $\cat{G}$ on $\cat{X}$ such that the associated $\cat{G}$-interleaving distance is equal to \eqref{eqn:group_distance}.
\end{prop*}

We use this result to realize other metrics of interest as interleaving distances; namely, geodesic distances on homogeneous spaces (Corollary \ref{cor:homogeneous_spaces}) and a metric on configuration spaces which is used frequently in the statistical shape analysis literature (Corollary \ref{cor:configuration_space}), and which is related to the Large Deformation Diffeomorphic Metric Matching framework used in medical imaging applications~\cite{beg2005computing}. The latter gives a new connection between concepts from the fields of TDA and pattern theory, in the sense of Grenander~\cite{grenander1996elements}.

In Section \ref{sec:interleaving_group_actions}, we focus on interleaving distances in the context of persistence modules. Extending the definition given above, we consider \emph{generalized persistence modules}, which are functors $M:\cat{P} \to \cat{C}$ from some poset category $\cat{P}$ to another category $\cat{C}$ (when the poset is the real numbers and $\cat{C}$ is the category of vector spaces, this recovers the definition above)~\cite{bubenik2015metrics}---we denote the collection of generalized persistence modules as $\cat{Fun}(\cat{P},\cat{C})$. Using a Yoneda-like construction (Proposition \ref{prop:yoneda_construction}), we construct interleaving distances on the space of generalized persistence modules from group (or, more generally, monoid) actions on $\cat{P}$

\begin{thm*}[Theorem \ref{thm:generalized_persistence_module_interleaving}]
    Let $\cat{P}$ be a poset category, let $\cat{C}$ be an arbitrary category and let $\set{G}$ be a monoid endowed with a monoidal weight. From an action of $\set{G}$ on (the underlying poset of) $\cat{P}$,  one obtains an associated interleaving distance on the space of generalized persistence modules $\cat{Fun}(\cat{P},\cat{C})$. 
\end{thm*}

We provide several examples of this construction in Section \ref{subsec:examples_of_2_functor_interleaving}, including when the group $\set{G}$ is:
\begin{itemize}
    \item the group of positive real numbers under multiplication---this gives a connection to recently observed statistical properties of random persistence diagrams~\cite{bobrowski2023universal} (see Remark \ref{rem:interpretation_of_multiplication_group});
    \item a group of diffeomorphisms---this provides another connection to ideas from statistical shape analysis;
    \item the additive monoid of vectors in $\R^n$ with nonnegative entries, when the underlying poset is $\R^n$ with its product poset structure---i.e., the usual setting of \emph{multiparameter persistent homology}. 
\end{itemize}

\smallskip
\paragraph{{\bf Stability.}} A general theme in TDA is the notion of stability---in this context, this means that the maps which transform datasets into topological descriptors should be Lipschitz with respect to appropriate metrics. The categorical perspective lends itself well to this sort of question, and one can derive quite general and conceptually simple stability results in this setting. We prove the following general result. 

\begin{thm*}[Theorem \ref{thm:stability}]
    A 2-functor between 2-weighted 2-categories which is weight non-increasing yields a Lipschitz map with respect to the associated interleaving distances.
\end{thm*}

Due to the very general structure of 2-weighted 2-category interleaving distances, the proof of this result is almost obvious. However, as we specialize to more specific settings, we obtain corollaries which are less so. We give the details of a specialization to the setting of $\cat{G}$-interleaving distances (Proposition \ref{prop:stability_actegory}) and apply this to give a generalization of the famous sublevel set persistent homology stability result of \cite{chazal2009proximity} to the setting of group actions and generalized persistence modules (Proposition \ref{prop:sublevel_stability}). 

\subsection{Outline of the Paper} Section \ref{sec:background} provides some background on 2-categories and actegories, and serves to set general notation. Section \ref{sec:interleaving distances from actegories} treats the $\cat{G}$-interleaving distance and 2-functor interleaving distance constructions, includes proofs that they define extended pseudometrics, and provides examples. Section \ref{sec:interleaving_group_actions} contains constructions, results and examples related to generalized persistence modules. Section \ref{sec:weighted_2_category_interleaving}  defines 2-weighted 2-functor interleaving distances and shows that this concept generalizes the previously introduced notions of interleaving distance. Section \ref{sec:stability} concerns stability of interleaving distances, in the sense described above. The paper concludes with a short discussion on future directions, in Section \ref{sec:discussion}.

\section{Background on 2-Categories and Monoidal Actions}\label{sec:background}

This section introduces the categorical terminology and notation that will be used for the rest of the paper. We assume that the reader is familiar with the basic notions of category theory, such as categories, functors and monoidal categories.

\subsection{Basics of 2-Categories}

The goal of this section is to introduce a family of metrics determined by a categorical action of a monoidal category (an \emph{actegory}---see Section \ref{sec:monoidal_actions}). The construction is closely related to the theory of 2-categories, so we begin with a very brief introduction to this theory, which will also allow us to fix notation.

\subsubsection{2-Categories} We begin with a formal definition of a 2-category.

\begin{defn}[2-Category]\label{defn:2cat}
A \emph{2-category} $\cat{C}$ consists of:
\begin{enumerate}
    \item A class of \emph{objects}, $\cat{C}_0$. We generally denote objects of $\cat{C}$ as $A,B,C,\ldots$, except when we are dealing with specialized examples.
    \item For each pair $A,B \in \cat{C}_0$, a class $\cat{C}_1(A,B)$ of \emph{morphisms} (sometimes we refer to them as \emph{$1$-morphisms}, if clarity is required). We generally denote these as $f:A \to B$ or $A \xrightarrow{f} B$. The $1$-morphisms are endowed with a composition, which we generally denote with multiplicative notation: for $A \xrightarrow{f} B$ and $B \xrightarrow{g} C$, the composition is denoted $A \xrightarrow{gf} C$. When more notationally convenient, we will instead use traditional composition notation, $g \circ f$. The composition must satisfy the usual axioms of a category: it is associative and for each object $A$ there is a unique identity map, denoted $1_A$, satisfying $1_B f = f = f 1_A$ for all $f \in \cat{C}_1(A,B)$. We use $\cat{C}_1$ to indicate the collection of all 1-morphism classes, $\{\cat{C}_1(A,B)\}_{A,B \in \cat{C}_0}$.

    \item For $f,g \in \cat{C}_1(A,B)$, a class of \emph{$2$-morphisms} $\cat{C}_2(f,g)$. A 2-morphism will typically be denoted $\alpha: f \Rightarrow g$, $f \xRightarrow{\alpha} g$ or 
    \[
    \begin{tikzcd}
      A \arrow[r, bend left, "f", ""{name=U,inner sep=1pt,below}]
      \arrow[r, bend right, "g"{below}, ""{name=D,inner sep=1pt}]
      & B.
      \arrow[Rightarrow, from=U, to=D, "\alpha"]
    \end{tikzcd}
    \]
    We have two flavors of composition for 2-morphisms:
    \begin{enumerate}
        \item For $f,g,h \in \cat{C}_1(A,B)$ and $f \xRightarrow{\alpha} g$, $g \xRightarrow{\beta} h$, \emph{vertical composition} is denoted using multiplicative notation as $f \xRightarrow{\beta \alpha} h$. When required for clarity, we use $\beta \circ \alpha$ to denote vertical composition. We assume that vertical composition is unital and associative. Denote the identity of $f$ as $1_f$. 
        \item For $f,g \in \cat{C}_1(A,B)$, $h,k \in \cat{C}_1(B,C)$ and $f \xRightarrow{\alpha} g$, $h \xRightarrow{\gamma} k$, \emph{horizontal composition} is denoted $hf \xRightarrow{\gamma \bullet \alpha} kg$. We assume that horizontal composition is also unital and associative, with unit $1_{1_A}:1_A \Rightarrow 1_A$.
    \end{enumerate}
    These two compositions are required to satisfy the \emph{interchange law}. Given objects $A, B, C \in \cat{C}_0$; 1-morphisms $f, g, h: A \to B$ and $f', g', h': B \to C$; and 2-morphisms $\delta:f \Rightarrow g$, $\beta: g \Rightarrow h$, $\gamma: f' \Rightarrow g'$, and $\alpha: g' \Rightarrow h'$, we have:
    \[
    (\alpha \bullet \beta) (\gamma \bullet \delta) = (\alpha \gamma) \bullet (\beta \delta). 
    \]
    Diagrammatically: 
\[
\begin{tikzcd}
	& {} \\
	A && B && C
	\arrow[""{name=0, anchor=center, inner sep=0}, "g"{pos=0.7}, from=2-1, to=2-3]
	\arrow[""{name=1, anchor=center, inner sep=0}, "{g'}"{pos=0.7}, from=2-3, to=2-5]
	\arrow[""{name=2, anchor=center, inner sep=0}, "h"', curve={height=30pt}, from=2-1, to=2-3]
	\arrow[""{name=3, anchor=center, inner sep=0}, "f", curve={height=-30pt}, from=2-1, to=2-3]
	\arrow[""{name=4, anchor=center, inner sep=0}, "{f'}", curve={height=-30pt}, from=2-3, to=2-5]
	\arrow[""{name=5, anchor=center, inner sep=0}, "{h'}"', curve={height=30pt}, from=2-3, to=2-5]
	\arrow["\delta", shorten <=4pt, shorten >=4pt, Rightarrow, from=3, to=0]
	\arrow["\beta", shorten <=4pt, shorten >=4pt, Rightarrow, from=0, to=2]
	\arrow["\gamma", shorten <=4pt, shorten >=4pt, Rightarrow, from=4, to=1]
	\arrow["\alpha", shorten <=4pt, shorten >=4pt, Rightarrow, from=1, to=5]
\end{tikzcd}
    \]
    
    We use $\cat{C}_2$ to indicate the collection of all 2-morphism classes, $\{\cat{C}_2(f,g)\}_{f,g \in \cat{C}_1(A,B); A,B \in \cat{C}_0}$.
\end{enumerate}
\end{defn}

\begin{rem}
We will use similar notation for categories (which we sometimes refer to as \emph{1-categories}, when clarity is required). That is, a category $\cat{D}$ consists of a class of objects $\cat{D}_0$ and, for each $A,B \in \cat{D}_0$, a class of morphisms $\cat{D}_1(A,B)$, together with a unital and associative composition operation.
\end{rem}

\begin{rem}[Whiskering]\label{rem:whiskering}
    The axioms of a 2-category allow for compositions between 1-morphisms and 2-morphisms, called \emph{whiskerings}. Given objects $A, B, C \in \cat{C}_0$, 1-morphisms $f, g: A \to B$, $k: B \to C$, and a 2-morphism $\alpha: f \to g$, the whiskering $k \bullet \alpha$ is simply the horizontal composition $1_k \bullet \alpha$. This is given by either of the diagrams:
\[\begin{tikzcd}
	A && B & C
	\arrow[""{name=0, anchor=center, inner sep=0}, "f", curve={height=-18pt}, from=1-1, to=1-3]
	\arrow[""{name=1, anchor=center, inner sep=0}, "g"', curve={height=18pt}, from=1-1, to=1-3]
	\arrow["k", from=1-3, to=1-4]
	\arrow["\alpha", shorten <=5pt, shorten >=5pt, Rightarrow, from=0, to=1]
\end{tikzcd}
\text{ or }
\begin{tikzcd}
	A && B && C.
	\arrow[""{name=0, anchor=center, inner sep=0}, "f", curve={height=-18pt}, from=1-1, to=1-3]
	\arrow[""{name=1, anchor=center, inner sep=0}, "g"', curve={height=18pt}, from=1-1, to=1-3]
	\arrow[""{name=2, anchor=center, inner sep=0}, "k", curve={height=-18pt}, from=1-3, to=1-5]
	\arrow[""{name=3, anchor=center, inner sep=0}, "k"', curve={height=18pt}, from=1-3, to=1-5]
	\arrow["\alpha", shorten <=5pt, shorten >=5pt, Rightarrow, from=0, to=1]
	\arrow["{1_k}", shorten <=5pt, shorten >=5pt, Rightarrow, from=2, to=3]
\end{tikzcd}\]
Given another 1-morphism $h: A \to B$ and 2-morphism $\beta:g \Rightarrow h$, the composition $(\beta \circ \alpha) \bullet k$ may be expressed by the diagram
\[
\begin{tikzcd}
	& {} \\
	A && B && C
	\arrow[""{name=0, anchor=center, inner sep=0}, "g"{pos=0.7}, from=2-1, to=2-3]
	\arrow[""{name=1, anchor=center, inner sep=0}, "{k}"{pos=0.7}, from=2-3, to=2-5]
	\arrow[""{name=2, anchor=center, inner sep=0}, "h"', curve={height=30pt}, from=2-1, to=2-3]
	\arrow[""{name=3, anchor=center, inner sep=0}, "f", curve={height=-30pt}, from=2-1, to=2-3]
	\arrow[""{name=4, anchor=center, inner sep=0}, "{k}", curve={height=-30pt}, from=2-3, to=2-5]
	\arrow[""{name=5, anchor=center, inner sep=0}, "{k}"', curve={height=30pt}, from=2-3, to=2-5]
	\arrow["\alpha", shorten <=4pt, shorten >=4pt, Rightarrow, from=3, to=0]
	\arrow["\beta", shorten <=4pt, shorten >=4pt, Rightarrow, from=0, to=2]
	\arrow["1_k", shorten <=4pt, shorten >=4pt, Rightarrow, from=4, to=1]
	\arrow["1_k", shorten <=4pt, shorten >=4pt, Rightarrow, from=1, to=5]
\end{tikzcd}
    \]
    In this case, the interchange law states that $$(1_k \bullet \beta) \circ (1_k \bullet \alpha) = (1_k \circ 1_k) \bullet (\beta \circ \alpha),$$ which simplifies to $k \bullet (\beta \circ \alpha) = (k \bullet \beta) \circ (k \bullet \alpha).$ Thus, in the case of whiskerings, horizontal composition (left) distributes over vertical composition. Applying the interchange law to a whiskering in the reverse order establishes right distributivity: $(\beta \circ \alpha) \bullet k = (\beta \bullet k) \circ (\alpha \bullet k).$
\end{rem}

\begin{ex}[The 2-Category of Small Categories]\label{ex:2_category_small_categories}
    The motivating example for the above definition is $\cat{Cat}$, the 2-category whose objects are small categories, whose 1-morphisms are functors and whose 2-morphisms are natural transformations. One can verify that this data satisfies the axioms of a 2-category (see, e.g., \cite[Example 2.3.14]{johnson20212}). In particular, vertical composition of composable natural transformations is defined componentwise by $(\beta\alpha)_A = \beta_A\alpha_A$ and horizontal composition works as follows. For functors $f,g: \cat{C} \rightarrow \cat{D}$ and $h,k:\cat{D} \rightarrow \cat{E}$ and natural transformations $f \xRightarrow{\alpha} g$  and $h \xRightarrow{\beta} k$, the horizontal composition $\beta \bullet \alpha$ is the natural transformation $hf \xRightarrow{\beta \bullet \alpha} kg$ with component at $A \in \cat{C}_0$ given by
\begin{equation}\label{eqn:horizontal_composition_nat_trans}
(\beta \bullet \alpha)_A = \beta_{g(A)} h(\alpha_A) = k(\alpha_A)\beta_{f(A)}.
\end{equation}
\end{ex}

\begin{rem}[Bicategories]\label{rem:2-categories_and_bicategories}
    A more concise definition of a 2-category is that it is a 1-category such that each set of (1-)morphisms is itself a small category, subject to certain compatibility axioms. That is, a 2-category is a category enriched in $\cat{Cat}$. There are more general versions of this construction, such as the notion of a \emph{bicategory}, or a 1-category for which the 1-morphism sets are small categories, but where the enriched category axioms are only required to hold up to natural isomorphism---see \cite{johnson20212}. We work in the setting of (strict) 2-categories in this paper, but expect that many of our results can be generalized.
\end{rem}

\begin{ex}[2-Categories from Monoidal 1-Categories]\label{ex:delooping}
Let $\cat{G}$ be a (strict) monoidal category with tensor product $\otimes$. In anticipation of a 2-categorical interpretation, we take the following notational convention for monoidal categories: objects of $\cat{G}$ will generically be denoted as $f$, $g$, $h$, $\ldots$ (notation typically reserved for 1-morphisms in an arbitrary 2-category) and morphisms of $\cat{G}$ will generically be denoted as $\alpha$, $\beta$, $\gamma$, $\ldots$ (notation typically reserved for 2-morphisms in an arbitrary 2-category). The identity object of $\cat{G}$ will generically be denoted as $e$. 

The \emph{delooping} of $\cat{G}$ is the 2-category $\deloop\cat{G}$ with: 
\begin{enumerate}
    \item $\deloop \cat{G}_0$ consisting of a single abstract element, which we denote $\star$,
    \item $\deloop \cat{G}_1(\star,\star) = \cat{G}_0$, with composition of $f$ and $g$ given by $f \otimes g$, and
    \item $\deloop \cat{G}_2(f,g) = \cat{G}_1(f,g)$, with vertical composition of $\alpha \in \cat{G}_1(f,g)$ and $\beta \in \cat{G}_1(g,h)$ given by composition in $\cat{G}$, $\beta\alpha$, and with horizontal composition of $\alpha \in \cat{G}_1(f,g)$ and $\beta \in \cat{G}_1(h,k)$ given by $\beta \otimes \alpha \in \cat{G}_1(h \otimes f,k \otimes g)$ (i.e., $\beta \bullet \alpha = \beta \otimes \alpha)$.
\end{enumerate}
In fact, 2-categories with a single object correspond exactly to deloopings of strict monoidal categories (see \cite{johnson20212}, Examples 2.1.19 and 2.3.11).
\end{ex}

\begin{rem}\label{rem:delooping}
    The notation $\deloop \cat{G}$ is somewhat standard, intended to evoke notation used in the context of classifying spaces.

    If $\set{G}$ is a group, we also use $\deloop \set{G}$ to denote the 1-category obtained by \emph{delooping} the group. That is, $\deloop \set{G}_0$ contains a single abstract object $\star$, $\deloop \set{G}_1(\star,\star) = \set{G}$ and compositions are given by the group law. 
\end{rem}

We now describe two useful 1-categories determined by any 2-category.

\begin{defn}[Associated 1-Categories]\label{def:associated_1_categories}
    For a 2-category $\cat{C}$, the \emph{underlying 1-category of $\cat{C}$} is the 1-category determined by the data $\cat{C}_0$ and $\cat{C}_1$. By abuse of notation, we may continue to denote the underlying 1-category as $\cat{C}$. Given $A,B \in \cat{C}_0$, the data $\cat{C}_1(A,B)$ and $\{\cat{C}_2(f,g)\}_{f,g \in \cat{C}_1(A,B)}$ determines a 1-category (using vertical compositions) called the \emph{associated 1-category determined by $(A,B)$}. This is denoted $\cat{C}(A,B)$.
\end{defn}

\begin{ex}
    Let $\cat{C}$ be a 2-category with a single object $\star$. Then $\cat{C} = \deloop\cat{G}$ for some monoidal category $\cat{G}$, and $\cat{C}(\star,\star) = \cat{G}$.
\end{ex}

\subsubsection{2-Functors}\label{sec:2_functors}

The notion of a functor between categories generalizes to the 2-category setting.

\begin{defn}[2-Functor]\label{def:2-functor}
Let $\cat{C}$ and $\cat{D}$ be 2-categories. A \emph{2-functor}  $\Delta:\cat{C} \to \cat{D}$ consists of:
\begin{enumerate}
    \item For all $A \in \cat{C}_0$, an object $\Delta(A) \in \cat{D}_0$. 
    \item For each 1-morphism $A \xrightarrow{f} B$ in $\cat{C}_1$, a 1-morphism $\Delta(f) =: \Delta_f:\Delta(A) \to \Delta(B)$ in $\cat{D}_1$. Composition and identities must be preserved; i.e., $\Delta_{gf} = \Delta_g \Delta_f$ and $\Delta_{1_A} = 1_{\Delta(A)}$.
    \item For each 2-morphism $f \xRightarrow{\alpha} g$ in $\cat{C}_2$, a 2-morphism $\Delta(\alpha):\Delta_f \Rightarrow \Delta_g$ in $\cat{D}_2$. Identities and (horizontal and vertical) compositions must be preserved; that is, $\Delta(1_f) = 1_{\Delta_f}$, $\Delta(1_{1_A}) = 1_{1_{\Delta(A)}}$,  $\Delta(\beta \alpha) = \Delta(\beta)\Delta(\alpha)$ and $\Delta(\delta \bullet \gamma) = \Delta(\delta) \bullet \Delta(\gamma)$. 
\end{enumerate}
Points (1) and (2) say that $\Delta$ determines a functor from the underlying 1-category of $\cat{C}$ to the underlying 1-category of $\cat{D}$. Points (2) and (3) say that for all $A,B \in \cat{C}_0$, $\Delta$ determines a functor $\cat{C}(A,B) \to \cat{C}(\Delta(A),\Delta(B))$ (see Definition \ref{def:associated_1_categories}), such that horizontal compositions are also preserved. 
\end{defn}

\begin{rem}[Lax 2-Functors]\label{rem:laxness}
    It is frequently convenient to relax the definition of a 2-functor by allowing diagrams to commute only up to natural transformations. This results in the definition of a \emph{lax 2-functor}. Most of our constructions and results can be extended to the lax setting. In order to simplify the exposition, we work with \emph{strict} 2-functors (i.e., 2-functors as described in Definition \ref{def:2-functor}) in this paper.
\end{rem}

\subsection{Basics of Actegories}\label{sec:monoidal_actions} Let $\cat{X}$ be a small (1-)category. The category $\cat{End}(\cat{X})$ of \emph{endomorphisms of $\cat{X}$} has objects given by functors from $\cat{X}$ to itself, and morphisms given by natural transformations. This category has a monoidal structure: using the notation of Example \ref{ex:2_category_small_categories} and Definition \ref{def:associated_1_categories}, $\cat{End}(\cat{X}) = \cat{Cat}(\cat{X},\cat{X})$, and its monoidal structure comes from composition laws of 1- and 2-morphisms in $\cat{Cat}$. Let $1_\cat{X}$ denote the identity functor on $\cat{X}$, which serves as the identity object for the monoidal structure.

Let $\cat{G}$ be a (strict) monoidal category with bifunctor denoted $\otimes$ and identity object $e \in \cat{G}_0$. A \emph{(strict) monoidal functor} $T:\cat{G} \to \cat{End}(\cat{X})$ is a functor which also preserves the monoidal structure; i.e., $T(e) = 1_\cat{X}$, for objects $g,h \in \cat{G}_0$, $T(g \otimes h) = T(g)T(h)$ and for morphisms $g \xrightarrow{\alpha} h$ and $k \xrightarrow{\beta} \ell$, $T(\beta \otimes \alpha) = T(\beta) \bullet T(\alpha)$. We will typically write $T(g) = T_g$, so that the first condition becomes $T(g \otimes h) = T_g T_h$. 

\begin{rem}[Lax Monoidal Functors]\label{rem:lax_monoidal}
    Similar to the 2-category setting described in Remark \ref{rem:laxness}, it can be useful to relax the notion of a monoidal functor to that of a \emph{lax monoidal functor}, where diagrams commute up to natural transformation. For the sake of simplicity, we focus on the strict case, although most of what follows can be extended to the lax setting. 
\end{rem}

A monoidal functor of the form $T:\cat{G} \to \cat{End}(\cat{X})$ is referred to as an \emph{action of $\cat{G}$ on $\cat{X}$}, and $\cat{X}$ is sometimes referred to as a \emph{$\cat{G}$-actegory}; these structures (as well as lax versions) have been widely studied, going back at least to the 1970s~\cite{benabou1967introduction,pareigis1977non,kelly2006doctrinal,janelidze2001note,mccrudden2000categories}.

A monoidal functor $T:\cat{G} \to \cat{End}(\cat{X})$ determines a 2-functor $\deloop T:\deloop \cat{G} \to \cat{Cat}$ defined on objects by $\deloop T(\star) = \cat{X}$, and on 1- and 2-morphisms by observing that $\deloop \cat{G}(\star, \star) = \cat{G}$ and $\cat{Cat}(\cat{X},\cat{X}) = \cat{End}(\cat{X})$ and setting $\deloop T = T$ on these associated 1-categories. By \cite[Examples 2.1.19 and 2.3.11]{johnson20212}, this association is a bijection onto the collection of 2-functors $\Delta:\cat{C} \to \cat{Cat}$ whose domain 2-category has a single object. That is, we have the following.

\begin{prop}\label{prop:monoidal_functors_and_2_functors}
    Any monoidal functor $T:\cat{G} \to \cat{End}(\cat{X})$ gives rise to a 2-functor $\Delta = \deloop T: \deloop \cat{G} \to \cat{Cat}$. Conversely, if $\cat{C}$ is a 2-category with a single object and $\Delta:\cat{C} \to \cat{Cat}$ is a 2-functor, then $\cat{C} = \deloop \cat{G}$ for some strict monoidal category $\cat{G}$ and $\Delta = \deloop T$ for a  monoidal functor $T:\cat{G} \to \cat{End}(\Delta(\star))$.
\end{prop}

\subsubsection{Category with a Flow}\label{sec:category_with_flow} A particular type of actegory was studied in \cite{de2018theory,stefanou2018dynamics}: a \emph{category with a flow} is a category $\cat{X}$, together with a (strict or lax) monoidal functor $T:\mathbb{R}_{\geq 0} \to \cat{End}(\cat{X})$ (i.e., an action of $\R_{\geq 0}$ on $\cat{X}$). For the sake of simplicity, let us focus on the case of a strict monoidal functor; the lax case can be easily adapted. Here, $\mathbb{R}_{\geq 0}$ is considered as a poset category (with respect to $\leq$), with monoidal structure given by addition.

In \cite{de2018theory,stefanou2018dynamics}, a category with a flow $\cat{X}$ is endowed with an extended pseudometric on $\cat{X}_0$. The distance is called \emph{interleaving distance} in the original paper (as it generalizes interleaving distances used previously in TDA~\cite{chazal2009proximity,bubenik2014categorification,bubenik2015metrics}), but we will refer to it as \emph{flow interleaving distance} to distinguish it from the generalized interleaving distances which are introduced in the following subsections. For $t \in \R_{\geq 0}$, we write $T_t = T(t) \in \cat{End}(\cat{X})$. A \emph{$t$-flow interleaving of $X,Y \in \cat{X}_0$} is a pair of morphisms $\phi:X \to T_t(Y)$ and $\psi:Y \to T_t(X)$ such that the diagrams
    \[
    \begin{tikzcd}
    X \ar[rr, "T(0 \leq 2t)_X"] \ar[rd, swap, "\phi"] & & T_{2t}(X) \\ 
    &  T_t(Y) \ar[ru, swap, "T_t(\psi)" ] & 
    \end{tikzcd}
    \qquad 
    \begin{tikzcd}
    Y \ar[rr, "T(0 \leq 2t)_Y"] \ar[rd, swap, "\psi"] & &T_{2t}(Y) \\ 
    &  T_t(X) \ar[ru, swap, "T_t(\phi)" ] & 
    \end{tikzcd}
    \]
commute. The \emph{flow interleaving distance} is given by
\begin{equation}\label{eqn:flow_interleaving}
d^\mathrm{Fl}_T(X,Y) = \inf \{t \in \R_{\geq 0} \mid \mbox{there exists a $t$-flow interleaving of $X$ and $Y$}\}.
\end{equation}
It is shown in \cite[Theorem 2.7]{de2018theory} that $d^\mathrm{Fl}_T$ defines an extended pseudometric on $\cat{X}_0$.

\section{Interleaving Distances from Actegories and 2-Functors}\label{sec:interleaving distances from actegories}

We now give generalizations of the flow interleaving distances described in Section \ref{sec:category_with_flow}. Our constructions require the data of an actegory (Section \ref{sec:actegory_interleaving}), or more generally a 2-functor from some 2-category into $\cat{Cat}$ (Section \ref{sec:2-functor_interleaving}). 

\subsection{Interleaving Distance Associated to an Actegory}\label{sec:actegory_interleaving} Let $\cat{G}$ be an arbitrary monoidal category. We generally omit the tensor product notation for objects and write $gh = g \otimes h$ for $g,h \in \cat{G}_0$ from now on. Let $\cat{X}$ be a $\cat{G}$-actegory, with monoidal functor $T:\cat{G} \to \cat{End}(\cat{X})$. We generalize the flow interleaving distance to an interleaving distance on $\cat{X}_0$ depending on a certain metric-like structure on $\cat{G}$.

\begin{defn}[Monoidal Weight]\label{def:monoidal_weight}
    A \emph{monoidal weight} on $\cat{G}$ is a function $\weight:\cat{G}_0 \to \R \cup \{\infty\}$ satisfying the following axioms:
    \begin{enumerate}
        \item for all $g \in \cat{G}_0$, $\weight(g) \geq 0$ and $\weight(e) = 0$;
        \item for all $f$ and $g$ in $\cat{G}_0$, 
        \[
        \weight(gf) \leq \weight(g) + \weight(f).
        \]
    \end{enumerate}
\end{defn}

\begin{rem}
    This structure is closely related to a well-known categorical construction due to Lawvere \cite{lawvere1973metric}. This connection is explained in detail below in Remark \ref{rem:lawvere_weight}.
\end{rem}

\begin{ex}
    For $\cat{G} = \R_{\geq 0}$ (see Section \ref{sec:category_with_flow}), $\weight(t) = t$ is a monoidal weight.
\end{ex}

\begin{rem}\label{rem:monoidal_weight_monoid}
    Let $\set{G}$ be a monoid. A function $\weight:\set{G} \to \R \cup \{\infty\}$ satisfying axioms analogous to those of Definition \ref{def:monoidal_weight} will, by abuse of terminology, also be referred to as a \emph{monoidal weight}.
\end{rem}

We now define interleavings and interleaving distances induced by a monoidal functor. 

\begin{defn}[Interleaving]\label{defn:interleaving_actegory}
    Let $T:\cat{G} \to \cat{End}(\cat{X})$ be a monoidal functor and let $X,Y \in \cat{X}_0$. For $g,h \in \cat{G}_0$, we say that $X$ and $Y$ are \emph{$(g, h)$-interleaved} if there exist a pair of morphisms $\phi: X \to T_g(Y)$, $\psi: Y \to T_h(X)$ in $\cat{X}$ and a pair of morphisms $\alpha:e \to gh$ and $\beta: e \Rightarrow hg$ in $\cat{G}$ such that the following diagrams commute: 
    \[
    \begin{tikzcd}
    X \ar[rr, "T(\alpha)_X"] \ar[rd, swap, "\phi"] & & T_{gh}(X) \\ 
    &  T_g(Y) \ar[ru, swap, "T_g(\psi)" ] & 
    \end{tikzcd}
    \qquad 
    \begin{tikzcd}
    Y \ar[rr, "T(\beta)_Y"] \ar[rd, swap, "\psi"] & &T_{hg}(Y) \\ 
    &  T_h(X) \ar[ru, swap, "T_h(\phi)" ] & 
    \end{tikzcd}
    \]
    In this case, we say that the pair $(\phi,\psi)$ is a \emph{$(g,h)$-interleaving of $X$ and $Y$}.
\end{defn}

\begin{defn}[$\cat{G}$-Interleaving Distance]\label{def:interleaving_distance_actegory}
    Let $T:\cat{G} \to \cat{End}(\cat{X})$ be a monoidal functor and $\weight$ a monoidal weight on $\cat{G}$.
    The \emph{$\cat{G}$-interleaving distance} determined by this data is the function
    \[
    d_T = d_{T,\weight}:\cat{X}_0 \times \cat{X}_0 \to \R \cup \{\infty\}
    \]
    defined by 
    \[
    d_T(X,Y) := \inf \{\max\{\weight(g),\weight(h)\} \mid \mbox{$X$ and $Y$ are $(g,h)$-interleaved}\}.
    \]
\end{defn}

\begin{rem}[Intuition Behind the Definition]\label{rem:intuition}
    We can intuit the definition of $\cat{G}$-interleaving distance by considering the monoidal functor $T$ as a framework for measuring the ``energy" that is required for the ``registration" of two objects $X,Y \in \cat{X}_0$. To elaborate on this:
    \begin{itemize}
        \item We can think of $T_g(Y)$ as using $g$ to register $Y$ to $X$, and likewise, $T_h(X)$ uses $h$ to register $X$ to $Y$---preservation of compositions by $T$ means that registrations roughly behave like an action of a transformation group.
        \item The natural transformations $T(\alpha)$ and $T(\beta)$ can be understood to control which registrations are admissible; e.g., the existence of $\alpha:e \Rightarrow gh$ tells us that $gh$ is sufficiently comparable to an identity transformation.
        \item The weight $\mathrm{W}$ allows us to the measure the ``cost" of a registration as $\max\{\weight(g),\weight(h)\}$.
    \end{itemize} 
    Thus the $\cat{G}$-interleaving distance $d_T(X,Y)$, at an intuitive level, measures the minimum weight that it takes for the objects $X$ and $Y$ to be admissibly registered to one another.
\end{rem}

\begin{thm}\label{thm:actegory_interleaving}
    Let $T:\cat{G} \to \cat{End}(\cat{X})$ be a monoidal functor and $\weight$ a monoidal weight on $\cat{G}$. The associated $\cat{G}$-interleaving distance $d_T$ is an extended pseudometric on $\cat{X}_0$. 
\end{thm}

We will prove Theorem \ref{thm:actegory_interleaving} below in Section \ref{sec:2-functor_interleaving} by deriving it as a corollary of a more general theorem built around an interleaving distance constructed from the data of a 2-functor. First, we prove formally the intuitively obvious fact that the notion of $\cat{G}$-interleaving distance generalizes that of flow interleaving distance.

\begin{prop}\label{prop:equivalence_to_flow}
Let $\cat{G} = \R_{\geq 0}$, considered as a monoidal category as in Section \ref{sec:category_with_flow} and let $T:\cat{G} \to \cat{End}(\cat{X})$ be a monoidal functor. Then the associated flow interleaving distance $d^\mathrm{Fl}_T$ is equal to the associated $\cat{G}$-interleaving distance $d_{T,\weight}$, where $\weight$ is the monoidal weight $\weight(t) = t$.
\end{prop}

\begin{proof}
We first observe that the flow interleaving distance can be written as 
\[
d^{\mathrm{Fl}}_T(X,Y) = \inf \{t \geq 0 \mid \mbox{$X$ and $Y$ are $(t,t)$-interleaved}\},
\]
where $(t,t)$-interleaved is meant in the sense of Definition \ref{defn:interleaving_actegory}. Now suppose that $X,Y \in \mathsf{X}_0$ are $(s,t)$-interleaved and, without loss of generality, that $s \leq t$. It then suffices to show that $X$ and $Y$ are $(t,t)$-interleaved. Write the diagrams for the $(s,t)$-interleaving as
    \[
    \begin{tikzcd}
    X \ar[rr, "T(0 \leq s+t)_X"] \ar[rd, swap, "\phi"] & & T_{s+t}(X) \\ 
    &  T_s(Y) \ar[ru, swap, "T_s(\psi)" ] & 
    \end{tikzcd}
    \qquad 
    \begin{tikzcd}
    Y \ar[rr, "T(0 \leq t +s)_Y"] \ar[rd, swap, "\psi"] & &T_{t+s}(Y) \\ 
    &  T_t(X) \ar[ru, swap, "T_t(\phi)" ] & 
    \end{tikzcd}
    \]
We then have the following commutative diagrams
\[\begin{tikzcd}
	X && {T_{s+t}(X)} && {T_{t+t}(X)} \\
	& {T_s(Y)} \\
	&& {T_t(Y)}
	\arrow["{T(0 \leq s + t)_X}"', from=1-1, to=1-3]
	\arrow["\phi"', from=1-1, to=2-2]
	\arrow["{T_s(\psi)}"', from=2-2, to=1-3]
	\arrow["{T(s \leq t)_Y}"', from=2-2, to=3-3]
	\arrow["{T(s+t \leq t+t)_X}"', from=1-3, to=1-5]
	\arrow["{T_t(\psi)}"', from=3-3, to=1-5]
	\arrow["{T(0 \leq t + t)_X}", curve={height=-24pt}, from=1-1, to=1-5]
\end{tikzcd}\]
\[\begin{tikzcd}
	Y &&&& {T_{t+t}(Y)} \\
	&&& {T_{t+s}(Y)} \\
	&& {T_t(X)}
	\arrow["{T_t(\phi)}"', from=3-3, to=2-4]
	\arrow["{T_t(T(s \leq t)_Y)}"', from=2-4, to=1-5]
	\arrow["\psi"', from=1-1, to=3-3]
	\arrow["{T(0 \leq t + s)_Y}"{description}, from=1-1, to=2-4]
	\arrow["{T(0 \leq t + t)_Y}", from=1-1, to=1-5]
\end{tikzcd}\]
and the perimeters of these diagrams give a $(t,t)$-interleaving of $X$ and $Y$. 
\end{proof}

It is shown in \cite{de2018theory} that many familiar metrics, such as Hausdorff distance and $\ell_\infty$ distance between scalar functions defined on a common space, can be realized as flow interleaving distances for appropriately chosen (strict or lax) monoidal functors. These results show that the examples of \cite{de2018theory} are, in turn, realized as 2-functor interleaving distances, at least in the strict case. The same result in the lax setting can also be proved by a similar argument, which we omit for the sake of simplicity.

\subsection{Interleaving Distance Associated to a 2-Functor into $\cat{Cat}$}\label{sec:2-functor_interleaving}
Although the focus of this paper is the $\cat{G}$-interleaving distance construction, one can easily generalize this construction for input data consisting of a 2-functor of the form $\Delta:\cat{C} \to \cat{Cat}$, where $\cat{C}$ is an arbitrary 2-category. We believe that it is conceptually simpler to show that these more general interleaving distances are extended pseudometrics, and from this we derive Theorem \ref{thm:actegory_interleaving} as an immediate corollary.

Suppose that $\Delta:\cat{C} \to \cat{Cat}$ is a 2-functor. To make the 2-functor conditions in this setting more explicit, we have:
\begin{enumerate}
    \item For all $A \in \cat{C}_0$, $\Delta(A)$ is a small category.
    \item For all $A \xrightarrow{f} B$, there is a functor $\Delta_f:\Delta(A) \to \Delta(B)$. 
    \item For each $f,g \in \cat{C}_1(A,B)$ and $\alpha \in \cat{C}_2(f,g)$, we have a natural transformation $\Delta(\alpha)$. It is assumed that, for all $A,B \in \cat{C}_0$, $\Delta$ defines a functor $\cat{C}(A,B) \to \cat{C}(\Delta(A),\Delta(B))$. For $f,g \in \cat{C}_1(A,B)$, naturality of $\Delta(\alpha)$ means that for each $X \in \Delta(A)_0$, we have a morphism given by the \emph{component of $\Delta(\alpha)$ at $X$},
    \[
    \Delta(\alpha)_X: \Delta_f(X) \to \Delta_g(X)
    \]
    such that, for every morphism $X \xrightarrow{w} Y$ in $\Delta(A)_1$, the following diagram commutes:
    \begin{equation}\label{eqn:natural_transformation_diagram_cat}
    \begin{tikzcd}[row sep=scriptsize, column sep=scriptsize]
    \Delta_f(X) \ar[rr, "\Delta(\alpha)_X"] \ar[d,swap,"\Delta_f(w)"] &  & \Delta_g(X)\ar[d, "\Delta_g(w)"] \\ 
    \Delta_f(Y) \ar[rr, "\Delta(\alpha)_Y"] &&  \Delta_g(X)
    \end{tikzcd}
    \end{equation}
    
\end{enumerate} 
Points (1) and (2) are assumed to give a functor of underlying 1-categories and the functor of (3) is assumed to preserve horizontal compositions and units. Preservation of horizontal compositions, in this context, means that for every $f,g:A \to B$ and $h,k:B \to C$ in $\cat{C}_1$ and $f \xRightarrow{\alpha} g$ and $h \xRightarrow{\beta} k$ in $\cat{C}_2$, and for all $X \in \Delta(A)_0$, we have 
\begin{equation}\label{eqn:horizontal_composition_2_functor}
\Delta(\beta \bullet \alpha)_X = \Delta(\beta)_{\Delta_g(X)}\Delta_h(\Delta(\alpha)_X) = \Delta_k(\Delta(\alpha)_X)\Delta(\beta)_{\Delta_f(X)},
\end{equation}
as this is just  \eqref{eqn:horizontal_composition_nat_trans} in different notation. 

We now use the data of a 2-functor $\Delta: \cat{C} \to \cat{Cat}$ to generalize the notion of $\cat{G}$-interleaving distance (Definition \ref{def:interleaving_distance_actegory}). We begin with some preliminary definitions.

\begin{defn}\label{defn:object_image}
        Let $\Delta:\cat{C} \to \cat{Cat}$ be a 2-functor. Define the \emph{object-image of $\Delta$} to be
    \[
    \mathrm{Im}(\Delta) := \bigcup_{A \in \cat{C}_0} \Delta(A)_0.
    \]
This coincides with the collection of objects of the \emph{join} of the categories $\Delta(A)$ \cite[1.2.8]{Lurie2009}, i.e., $$\mathrm{Im}(\Delta) = \left(\underset{A \in \cat{C}_0}{\bigstar} \Delta(A)\right)_0.$$
\end{defn}

\begin{defn}[Interleaving]\label{defn:interleaving}
    Let $\Delta: \cat{C} \to \cat{Cat}$ be a 2-functor and let $X,Y \in \mathrm{Im}(\Delta)$, so that $X \in \Delta(A)_0$ and $Y \in \Delta(B)_0$ for some $A, B \in \cat{C}_0.$ For morphisms $g \in \cat{C}_1(B, A)$  and $h \in \cat{C}_1(A,B)$, we say that $X$ and $Y$ are $(g, h)$\emph{-interleaved} if there exist a pair of morphisms $ \phi: X \to \Delta_g(Y) \in \Delta(A)_1$, $ \psi: Y \to \Delta_h(X) \in \Delta(B)_1$ and a pair of 2-morphisms $\alpha: 1_A \Rightarrow gh$, $\beta: 1_B \Rightarrow hg \in \cat{C}_2$ such that the following diagrams commute: 
    \[
    \begin{tikzcd}
    X \ar[rr, "\Delta(\alpha)_X"] \ar[rd, swap, "\phi"] & & \Delta_{gh}(X) \\ 
    &  \Delta_g(Y) \ar[ru, swap, "\Delta_g(\psi)" ] & 
    \end{tikzcd}
    \qquad 
    \begin{tikzcd}
    Y \ar[rr, "\Delta(\beta)_Y"] \ar[rd, swap, "\psi"] & &\Delta_{hg}(Y) \\ 
    &  \Delta_h(X) \ar[ru, swap, "\Delta_h(\phi)" ] & 
    \end{tikzcd}
    \]
    In this case, we say that the pair $(\phi,\psi)$ is a \emph{$(g,h)$-interleaving of $X$ and $Y$}.
\end{defn}

\begin{defn}[Lawvere Weight]\label{def:lawvere_weight}
    Let $\cat{C}$ be a (1 or 2)-category. A \emph{Lawvere weight} on $\cat{C}$ is a function $\weight:\cat{C}_1 \to \R \cup \{\infty\}$ satisfying the following axioms:
    \begin{enumerate}
        \item for all $g \in \cat{C}_1(A,B)$, $\weight(g) \geq 0$ and $\weight(g) = 0$ if $A = B$ and $g = 1_A$;
        \item for all composable $f$ and $g$ in $\cat{C}_1$, 
        \[
        \weight(gf) \leq \weight(g) + \weight(f).
        \]
    \end{enumerate}
    A (1 or 2)-category endowed with a weight will be called a \emph{weighted (1 or 2)-category}.
\end{defn}

\begin{rem}\label{rem:induced_lawvere_weight}
    Let $\cat{G}$ be a monoidal category and $\weight$ a monoidal weight (Definition \ref{def:monoidal_weight}). Let $\cat{B}\cat{G}$ be the 2-category obtained by delooping (Example \ref{ex:delooping}). There is a natural Lawvere weight on $\cat{B}\cat{G}$ induced via $\cat{B}\cat{G}_1 = \cat{G}_0$.
\end{rem}

\begin{rem}\label{rem:lawvere_weight}
    Said differently, a Lawvere weight on $\cat{C}$ is an enrichment of (the underlying 1-category of) $\cat{C}$ over the monoidal category whose objects are elements of $[0,\infty]$, such that there exists a morphism $s \rightarrow t$ if and only if $s \leq t$, and with monoidal structure given by addition. Lawvere first considered structures of this form in the influential paper \cite{lawvere1973metric}, where he showed that the notion of a category endowed with a Lawvere weight is equivalent to the notion of an extended quasi-pseudo-metric space (a class $\set{X}$ endowed with a kernel $d:\set{X} \times \set{X} \to \R \cup \{\infty\}$ satisfying $d(x,x) = 0$ and the triangle inequality), also known as a \emph{Lawvere metric space}.
    
    The terminology \emph{weight} is not entirely standard; the structure described in Definition \ref{def:lawvere_weight} has been alternatively referred to as a \emph{metric} or \emph{norm} \cite{lawvere1973metric} or \emph{length} \cite{neeman2020metrics} on the category. The \emph{weight} terminology is used in, for example,  \cite{bubenik2017interleaving}.
\end{rem}

\begin{defn}[2-Functor Interleaving Distance]\label{def:interleaving_distance}
    Let $\Delta: \cat{C} \to \cat{Cat}$ be a 2-functor and $\weight$ a Lawvere weight on $\cat{C}$. 
    The \emph{2-functor interleaving distance} determined by this data is the function
    \[
    d_\Delta = d_{\Delta,\weight}:\mathrm{Im}(\Delta) \times \mathrm{Im}(\Delta) \to \R \cup \{\infty\}
    \]
    defined by 
    \[
    d_\Delta(X,Y) := \inf \{\max\{\weight(g),\weight(h)\} \mid \mbox{$X$ and $Y$ are $(g,h)$-interleaved}\}.
    \]
\end{defn}

\begin{thm}\label{thm:interleaving_distance}
    Let $\Delta:\cat{C} \to \cat{Cat}$ be a 2-functor and $\weight$ a Lawvere weight on $\cat{C}$. The 2-functor interleaving distance $d_\Delta$ defines an extended pseudometric on $\mathrm{Im}(\Delta)$. 
\end{thm}

The proof will use a technical lemma.

\begin{lem}\label{lem:interleavability}
    For any 2-functor $\Delta: \cat{C} \to \cat{Cat}$, the following condition holds. Suppose $h,k:A \to B$ and $g,\ell: B \to A$ are 1-morphisms in $\cat{C}_1$ and that there exist 2-morphisms $1_A \xRightarrow{\alpha} gh$ and $1_B \xRightarrow{\beta} k\ell$ in $\cat{C}_2$; i.e., we have the diagram:
        \begin{equation}\label{eqn:interleavability_diagram}
        \begin{tikzcd}
	A && A \\
	B && B
	\arrow[""{name=0, anchor=center, inner sep=0}, "{1_A}"{pos=0.6}, curve={height=-12pt}, from=1-1, to=1-3]
	\arrow[""{name=1, anchor=center, inner sep=0}, "gh"', curve={height=6pt}, from=1-1, to=1-3]
	\arrow[""{name=2, anchor=center, inner sep=0}, "k\ell", curve={height=-6pt}, from=2-1, to=2-3]
	\arrow[""{name=3, anchor=center, inner sep=0}, "{1_B}"', curve={height=12pt}, from=2-1, to=2-3]
	\arrow["\ell"', curve={height=6pt}, from=2-1, to=1-1]
	\arrow["h"', curve={height=6pt}, from=1-1, to=2-1]
	\arrow["k", curve={height=-6pt}, from=1-3, to=2-3]
	\arrow["g", curve={height=-6pt}, from=2-3, to=1-3]
	\arrow["\alpha", shorten <=2pt, shorten >=2pt, Rightarrow, from=0, to=1]
	\arrow["\beta"', shorten <=2pt, shorten >=2pt, Rightarrow, from=3, to=2]
\end{tikzcd}
    \end{equation}
    Then there must be a 2-morphism $\gamma:1_A \Rightarrow gk\ell h$
\[\begin{tikzcd}
	A && A
	\arrow[""{name=0, anchor=center, inner sep=0}, "{1_A}"{description, pos=0.3}, from=1-1, to=1-3]
	\arrow[""{name=1, anchor=center, inner sep=0}, "gh"', curve={height=18pt}, from=1-1, to=1-3]
	\arrow[""{name=2, anchor=center, inner sep=0}, "{gk\ell h}", curve={height=-18pt}, from=1-1, to=1-3]
	\arrow["{\;\exists \; \gamma}"'{pos=0.4}, shorten <=2pt, shorten >=2pt, Rightarrow, from=0, to=2]
	\arrow["{\;\alpha}", shorten <=2pt, shorten >=2pt, Rightarrow, from=0, to=1]
\end{tikzcd}\]
such that, for any object $X \in \Delta(A)_0$,
    \begin{equation}\label{eqn:interleavability}
\Delta_g\big(\Delta(\beta)_{\Delta_h(X)}\big)  \Delta(\alpha)_X = \Delta(\gamma)_X;
    \end{equation}
    that is, the following diagram commutes:
\[\begin{tikzcd}
	{X=\Delta_{1_A}(X)} && {\Delta_{gk\ell h}(X) = \Delta_g \Delta_{k\ell} \Delta_h(X)} \\
	& {\Delta_{gh}(X) = \Delta_g \Delta_h(X)}
	\arrow["{\Delta(\alpha)_X}"', from=1-1, to=2-2]
	\arrow["{\Delta(\gamma)_X}", from=1-1, to=1-3]
	\arrow["{\Delta_g\big(\Delta(\beta)_{\Delta_h(X)}\big)}"', from=2-2, to=1-3]
\end{tikzcd}\]
\end{lem}

\begin{proof}
    Suppose that we have the diagram \eqref{eqn:interleavability}. Consider the diagram
    \[\begin{tikzcd}
	A && B && B && A
	\arrow["g"{description}, from=1-5, to=1-7]
	\arrow["h"{description}, from=1-1, to=1-3]
	\arrow[""{name=0, anchor=center, inner sep=0}, "{k\ell}"{description, pos=0.3}, curve={height=-12pt}, from=1-3, to=1-5]
	\arrow[""{name=1, anchor=center, inner sep=0}, "{1_B}"{description, pos=0.3}, curve={height=12pt}, from=1-3, to=1-5]
	\arrow["{gk\ell h}"{description}, curve={height=-30pt}, from=1-1, to=1-7]
	\arrow[""{name=2, anchor=center, inner sep=0}, "gh"{description, pos=0.7}, curve={height=25pt}, from=1-1, to=1-7]
	\arrow[""{name=3, anchor=center, inner sep=0}, "{1_A}"{description, pos=0.7}, curve={height=50pt}, from=1-1, to=1-7]
	\arrow["\beta"', shorten <=2pt, shorten >=2pt, Rightarrow, from=1, to=0]
	\arrow["\alpha"', shorten <=2pt, shorten >=2pt, Rightarrow, from=3, to=2]
\end{tikzcd}\]
    Then we construct $\gamma:1_A \Rightarrow gk\ell h$ by whiskering (Remark \ref{rem:whiskering}) as 
    \[
    \gamma = (g \bullet \beta \bullet h) \circ \alpha.
    \]
    Equation \eqref{eqn:interleavability} follows by the assumption that $\Delta$ is a 2-functor and by unpacking equation \eqref{eqn:horizontal_composition_2_functor}.
\end{proof}

\begin{proof}[Proof of Theorem \ref{thm:interleaving_distance}]
Symmetry of $d_\Delta$ is obvious. To see that $X=Y$ implies $d_\Delta(X,Y) = 0$, suppose that $X=Y \in \Delta(A)$ for $A \in \cat{C}_0$. Then $X$ and $Y$ are $(1_A,1_A)$-interleaved, since
\[
\begin{tikzcd}
X \ar[rr, "\Delta(1_{1_A})_X = 1_X"] \ar[rd, swap, "1_X"] & & \Delta_{1_A}\Delta_{1_A}X = X \\ 
&  \Delta_{1_A}(Y) = X \ar[ru, swap, "\Delta_{1_A}(1_X) = 1_X" ] & 
\end{tikzcd}
\]
commutes. Therefore
\[
d_\Delta(X,Y) \leq \weight(1_A) = 0.
\]

It only remains to check the triangle inequality. Let $X,Y,Z \in \mathrm{Im}(\Delta)$, say $X \in \Delta(A)_0$, $Y \in \Delta(B)_0$ and $Z \in \Delta(C)_0$. Suppose that $X$ and $Y$ are $(g,h)$-interleaved and $Y$ and $Z$ are $(k,\ell)$-interleaved. In particular, choose morphisms 
\begin{align*}
\tau &\in \Delta(A)_1(X,\Delta_g(Y)) \qquad \chi \in \Delta(B)_1(Y,\Delta_h(X)) \\
\mu &\in \Delta(B)_1(Y,\Delta_k(Z)) \qquad
\nu \in \Delta(C)_1(Z,\Delta_\ell(Y))
\end{align*}
and 2-morphisms
\begin{align*}
\alpha &\in \cat{C}_2(e_A,gh) \qquad
\beta \in \cat{C}_2(e_B,hg) \\
\gamma &\in \cat{C}_2(e_B,k\ell) \qquad
\delta \in \cat{C}_2(e_C,\ell k)
\end{align*}
such that the following diagrams commute:
\[
\begin{tikzcd}
X \ar[rr, "\Delta(\alpha)_X"] \ar[rd, swap, "\tau"] & & \Delta_{gh}X \\ 
&  \Delta_g Y \ar[ru, swap, "\Delta_g(\chi)" ] & 
\end{tikzcd}
\qquad 
\begin{tikzcd}
Y \ar[rr, "\Delta(\beta)_Y"] \ar[rd, swap, "\chi"] & & \Delta_{hg}Y \\ 
&  \Delta_h X \ar[ru, swap, "\Delta_h(\tau)" ] & 
\end{tikzcd}
\]
\[
\begin{tikzcd}
Y \ar[rr, "\Delta(\gamma)_Y"] \ar[rd, swap, "\mu"] & & \Delta_{k\ell} Y \\ 
&  \Delta_k Z \ar[ru, swap, "\Delta_k(\nu)" ] & 
\end{tikzcd}
\qquad 
\begin{tikzcd}
Z \ar[rr, "\Delta(\delta)_Z"] \ar[rd, swap, "\nu"] & & \Delta_{\ell k} Z \\ 
&  \Delta_\ell Y \ar[ru, swap, "\Delta_\ell(\mu)" ] & 
\end{tikzcd}
\]
Now consider the diagram
\[\begin{tikzcd}
	X && {\Delta_{gh}X} && {\Delta_g \Delta_{k \ell} \Delta_h X = \Delta_{gk} \Delta_{\ell h} X} \\
	& {\Delta_g Y} && {\Delta_g \Delta_k \Delta_\ell Y} \\
	&& {\Delta_g \Delta_k Z = \Delta_{gk} Z}
	\arrow["\tau"', from=1-1, to=2-2]
	\arrow["{\Delta_g(\chi)}"', from=2-2, to=3-3]
	\arrow["{\Delta(\alpha)_X}", from=1-1, to=1-3]
	\arrow["{\Delta_g \Delta_k \nu}"', from=3-3, to=2-4]
	\arrow["{\Delta_{gk\ell}(\chi)}"', from=2-4, to=1-5]
	\arrow["{\Delta_g(\Delta(\gamma)_Y)}", from=2-2, to=2-4]
	\arrow["{\Delta_g(\Delta(\gamma)_{\Delta_h X})}", from=1-3, to=1-5]
	\arrow["{\Delta_g(\chi)}", from=2-2, to=1-3]
	\arrow["{\Delta(1_A \xRightarrow{\varepsilon} gk\ell h)_X}", curve={height=-30pt}, dashed, from=1-1, to=1-5]
\end{tikzcd}\]
The existence of the top (dotted) arrow follows from Lemma \ref{lem:interleavability}. Applying Lemma \ref{lem:interleavability},  together with commutativity of the various subdiagrams, also implies that the entire diagram commutes. The perimeter of this diagram is one of the necessary diagrams to conclude that $X$ and $Z$ are $(gk,\ell h)$-interleaved. The construction of the other diagram is similar. 

We have shown that if $X$ and $Y$ are $(g,h)$-interleaved and $Y$ and $Z$ are $(k,\ell)$-interleaved, then $X$ and $Z$ are $(gk,\ell h)$-interleaved. We also have
\[
\max\{\weight(gk), \weight(\ell h)\} \leq \max\{\weight(g) + \weight(k), \weight(\ell) + \weight(h)\} \leq \max\{\weight(g), \weight(h)\} + \max\{\weight(k), \weight(\ell)\}.
\]
The triangle inequality for $d_\Delta$ follows easily and this concludes the proof.
\end{proof}

The next result follows by checking the definitions.

\begin{prop}\label{prop:G_interleaving_2_functor_interleaving}
    Let $T:\cat{G} \to \cat{End}(\cat{X})$ be a monoidal functor and $\weight$ a monoidal weight on $\cat{G}$. Let $\cat{B}T:\cat{B}\cat{G} \to \cat{Cat}$ be the associated 2-functor from Proposition \ref{prop:monoidal_functors_and_2_functors}. Abusing notation, let $\weight$ denote the induced Lawvere weight on $\cat{B}\cat{G}$ (Remark \ref{rem:induced_lawvere_weight}). Then the interleaving distances on $\cat{X}_0 = \mathrm{Im}(\cat{B}T)$ agree, $d_{T,\weight} = d_{\cat{B}T,\weight}$.
\end{prop}

 We now easily derive the metric properties of $\cat{G}$-interleaving distances claimed in Theorem \ref{thm:actegory_interleaving}.

\begin{proof}[Proof of Theorem \ref{thm:actegory_interleaving}]
    Apply Theorem \ref{thm:interleaving_distance} and Proposition \ref{prop:G_interleaving_2_functor_interleaving}.
\end{proof}

\subsection{Examples of Interleaving Distances}\label{subsect:examples_of_interleaving_distances} The main examples of interest in this paper come from variations on persistent homology; these will be treated in detail below in Section \ref{sec:interleaving_group_actions}. For now, we give a few simple examples of commonly used distances, outside the traditional realm of TDA, which can be encoded in this framework.

\subsubsection{Group Actions}\label{sect:group_actions_on_set}
As was mentioned in the introduction, an inspiration for the definition of $\cat{G}$-interleaving distance was an analogy between interleaving distances and induced metrics on homogeneous spaces. The examples presented in this subsection make precise the connection between these ideas.

Let $\set{G}$ be a group. We denote its elements as $f,g,h,\ldots$, with the identity denoted $e$, and we use multiplicative notation for the group law. This notation coincides with the notation used for monoidal categories (and 2-categories) in anticipation of the following construction of an associated monoidal category. 

\begin{defn}[Associated Monoidal Category]\label{def:associated_monoidal}
Let $\cat{G}$ be the (1-)category with objects $\cat{G}_0 = \set{G}$ and a unique morphism $\alpha_{g,h} \in \cat{G}_1(g,h)$ for every pair $g,h \in \set{G}$. This defines a monoidal category with tensor product $\otimes$ defined by $g \otimes h = gh$ and $\alpha_{f,g} \otimes \alpha_{h,k} = \alpha_{fh,gk}$. We refer to $\cat{G}$ as the \emph{monoidal category associated to $\set{G}$.}
\end{defn}

Suppose that the group $\set{G}$ acts on a set $\set{X}$ (say, from the left); write the action as $\set{G} \times \set{X} \to \set{X}:(g,x) \mapsto gx$. There is a natural groupoid (1-category whose morphisms are all invertible) associated to this action, which we define below. This is, in fact, one of the prototypical constructions of a groupoid---see \cite[Example 3]{brown1987groups} for a discussion on terminology and further references.

\begin{defn}[Action Groupoid]\label{def:action_groupoid}
    The \emph{action groupoid} associated to the action of $\set{G}$ on $\set{X}$ is the (1-)category $\cat{X}$ with $\cat{X}_0 = \set{X}$ and 
\[
\cat{X}_1(x,y) = \{\beta_{g,x,y} \mid \mbox{$g \in \set{G}$ satisfies $gx = y$}\}.
\]
That is, we add a unique 1-morphism for every group element taking $x$ to $y$ via the action. Compositions are then defined by $\beta_{h,y,z}\beta_{g,x,y} :=\beta_{hg,x,z}$. Note that if the action of $\set{G}$ is not transitive, then $\cat{X}_1(x,y)$ may be empty. However, $\cat{X}_1(x,x)$ always contains an identity morphism $\beta_{e,x,x}$ (and this is the unique morphism if and only if $x$ has  trivial stabilizer). 
\end{defn}

We define a monoidal functor $T:\cat{G} \to \cat{End}(\cat{X})$ as follows:
\begin{enumerate}
    \item $T_g:\cat{X} \to \cat{X}$ is defined on objects by $T_g(x) = gx$ and on morphisms by $T_g(\beta_{h,x,y}):= \beta_{ghg^{-1},gx,gy}$,
    \item $T(\alpha_{g,h}):T_g \Rightarrow T_h$ is defined to be the natural transformation with component at $x$ given by $T(\alpha_{g,h})_x = \beta_{hg^{-1},gx,hx}$. Then, for any $\beta_{k,x,y}:x \to y$, diagram \eqref{eqn:natural_transformation_diagram_cat} reads
    \[
    \begin{tikzcd}[row sep=scriptsize, column sep=scriptsize]
    gx \ar[rr, "\beta_{hg^{-1},gx,hx}"] \ar[d,swap,"\beta_{gkg^{-1},gx,gy}"] &  & hx \ar[d, "\beta_{hkh^{-1},hx,hy}"] \\ 
    gy \ar[rr, "\beta_{hg^{-1},gy,hy}"] &&  hy
    \end{tikzcd}
    \]
    and this diagram commutes, so that $T(\alpha_{g,h})$ is a natural transformation.
\end{enumerate}
One can check that all compositions are preserved under these definitions.

\begin{defn}[Invariant Metric]\label{def:right_invariant}
    Let $\set{G}$ be a group. A \emph{right-invariant} (respectively \emph{left-invariant}) metric on $\set{G}$ is a metric $d:\set{G} \times \set{G} \to \R$ with the property that $d(fh,gh) = d(f,g)$ (respectively, $d(hf,hg) = d(f,g)$) for all $f,g,h \in \set{G}$.
\end{defn}

In what follows, we make statements mainly for right-invariant metrics, but similar statments hold for left-invariant metrics. 

\begin{ex}\label{ex:Lie_group_right_invariant}
    Any finite-dimensional Lie group $\set{G}$ admits a right-invariant metric, defined by first choosing an inner product on its Lie algebra, extending this to a Riemannian metric via right translations, then taking $d$ to be the associated geodesic distance.
\end{ex}

\begin{prop}\label{prop:right_invariant_weight}
    Let $\set{G}$ be a group with right-invariant metric $d$ and let $\weight:\set{G} \to \R$ be the map $\weight(g) := d(e,g)$, where $e \in \set{G}$ is the identity. Then $\weight(e) = 0$ and 
    \begin{equation}\label{eqn:right_invt_weight}
    \weight(gh) \leq \weight(g) + \weight(h).
    \end{equation}
\end{prop}

\begin{proof}
    That $\weight(e) = 0$ is obvious and \eqref{eqn:right_invt_weight} is a simple consequence of right-invariance and the other metric properties of $d$:
    \[
    \weight(gh) = d(e,gh) = d(h^{-1},g) \leq d(h^{-1},e) + d(e,g) = d(e,h) + d(e,g) = \weight(g) + \weight(h).
    \]
\end{proof}

Now suppose that $\set{G}$ is endowed with a right-invariant metric $d_\set{G}$. Then Proposition \ref{prop:right_invariant_weight} says that the function $\weight:\set{G} \to \R$ defined by $\weight(g) = d_\set{G}(e,g)$ is a monoidal weight on $\cat{G}$. The following result characterizes the associated $\cat{G}$-interleaving distance.

\begin{prop}\label{prop:group_action_on_set}
Let $\set{G}$ be a group with right-invariant metric $d_{\set{G}}$. Suppose that $\set{G}$ acts transitively on a set $\set{X}$ and let  $T:\cat{G} \to \cat{End}(\cat{X})$ be the monoidal functor defined above. Let $\weight$ be the monoidal weight on $\cat{G}$ associated to $d_\set{G}$. Then, for all $x,y \in \set{X} = \cat{X}_0$,
\begin{equation}\label{eqn:group_action_interleaving_distance}
d_{\Delta,\weight}(x,y) = \inf_{g \in \set{G}} \{d_{\set{G}}(e,g) \mid gx = y\}.
\end{equation}
\end{prop}

This is a rather natural metric on $\set{X}$: intuitively, $\cat{G}$-interleaving distance gives a measurement of the  weight required to move one element of the set to another via the group action (cf. Remark \ref{rem:intuition}). This was the distance \eqref{eqn:group_distance} described in the introduction as a motivation for our interleaving distance framework.

\begin{proof}
Let $x,y \in \set{X}$. Since the action of $\set{G}$ is transitive, an interleaving of $x$ and $y$ always exists. Indeed, for any $g \in \set{G}$ such that $y = gx$ we have the diagrams 
\[
\begin{tikzcd}
x \ar[rr, "T(\alpha_{e,e})_x = \beta_{e,x,x}"] \ar[rd, swap, " \beta_{e,x,x}"] & & T_g T_{g^{-1}}(x)  \\ 
&  T_g(y) = gy = x \ar[ru, swap, "\beta_{e,x,x}" ] & 
\end{tikzcd}
\qquad 
\begin{tikzcd}
y \ar[rr, "T(\alpha_{e,e})_y = \beta_{e,y,y}"] \ar[rd, swap, "\beta_{e,y,y}"] & & T_{g^{-1}} T_g (y) \\ 
&  T_{g^{-1}}(x) = g^{-1}x = y \ar[ru, swap, "\beta_{e,y,y}" ] & 
\end{tikzcd}
\]
so that $x$ and $y$ are $(g,g^{-1})$-interleaved. 

On the other hand, suppose that $x$ and $y$ are $(g,h)$-interleaved. If, without loss of generality, $d_{\set{G}}(e,g) \leq d_{\set{G}}(e,h)$, then the weight $\max\{\weight(g),\weight(h)\}$ of the $(g,h)$-interleaving is at least as large as the weight $\max\{\weight(g),\weight(g^{-1})\}$ of the $(g,g^{-1})$-interleaving. This follows because right-invariance and symmetry of $d_{\set{G}}$ implies $d_{\set{G}}(e,g^{-1}) = d_{\set{G}}(e,g)\leq d_{\set{G}}(e,h)$. The formula for interleaving distance claimed in \eqref{eqn:group_action_interleaving_distance} therefore holds. 
\end{proof}

\begin{rem}\label{rem:inverse_weight}
    Going through the proof, one sees that the requirement that $\weight$ be defined with respect to a right-invariant metric is not necessary. In fact, all that we really used is that $\weight(g) = \weight(g^{-1})$ for all $g \in \set{G}$. If $\weight$ is any monoidal weight with this property, then the proof shows that 
    \[
    d_{\Delta,\weight}(x,y) = \inf_{g \in \set{G}} \{\weight(g) \mid gx = y\}.
    \]
\end{rem}

The following two subsections give some more specific examples of this result.

\subsubsection{Homogeneous Spaces} Consider a compact connected Lie group $\set{G}$. Suppose that $\set{G}$ acts properly and transitively on a smooth manifold $\set{X}$. Then there is a diffeomorphism $\set{G}/\set{K} \to \set{X}$, where $\set{K}$ is a compact stabilizer subgroup of $\set{G}$. Suppose that $\set{G}$ is endowed with a Riemannian metric which is left $\set{G}$-invariant and right $\set{K}$-invariant (such a metric will always exist). One can then induce a Riemannian metric on $\set{X}$ by first choosing a Riemannian structure on $\set{G}/\set{K}$ so that the quotient map $\set{G} \to \set{G}/\set{K}$ is a submersion, and then transferring to $\set{X}$ via the above diffeomorphism. It is a fact\footnote{We could not find a proper reference for this ``folklore" result. Please refer to the thorough explanation in the Math StackExchange answer \cite{3236071}.} that the geodesic distance between two points $x,y \in \set{X}$ is given by
\[
d_{\set{X}}(x,y) = \inf_{g \in \set{G}} \{d_{\set{G}}(e,g) \mid gx = y\}.
\]
Therefore, we have the following corollary of Proposition \ref{prop:group_action_on_set}.

\begin{cor}\label{cor:homogeneous_spaces}
The induced geodesic distance on a homogeneous space, in the sense described above, can be expressed as a  $\cat{G}$-interleaving distance.
\end{cor}

\subsubsection{Shape Analysis on Spaces of Landmarks and Embedded Submanifolds.}\label{sec:landmarks} A central object of study in the fields of statistical shape shape analysis and medical imaging is the space $\mathrm{Emb}(\set{M},\R^n)$ of all smooth submanifolds of $\R^n$ which are diffeomorphic to some fixed template manifold $\set{M}$. In particular, researchers are interested in metrics on this space which are justified for computational and/or theoretical reasons. Frequently, these metrics arise as geodesic distances coming from certain Riemannian metrics, which, in turn, are constructed from Riemannian metrics on the infinite-dimensional Lie group $\mathrm{Diff}(\R^n)$ of diffeomorphisms of $\R^n$ (or subgroups with various prescribed regularity conditions), together with the action of this space on $\mathrm{Emb}(\set{M},\R^n)$. This is the general perspective taken in Grenander's \emph{Pattern Theory} \cite{grenander1996elements}, and is closely tied to the \emph{Large Deformation Diffeomorphic Metric Mapping (LDDMM)} framework, a powerful tool used for registration of medical images \cite{beg2005computing}. When the template manifold $\set{M}$ is a discrete set of $k$ points, $\mathrm{Emb}(\set{M},\R^n)$ is also known as the \emph{configuration space} of $k$ points in $\R^n$. In this subsection, we view shape analysis metrics on this configuration space through the lens of $\cat{G}$-interleaving distances. Similar comparisons exist for more general embedding spaces, but we focus on the configuration space in order to keep the discussion more concrete. 

Let $\set{G}$ be a subgroup of $\mathrm{Diff}(\R^n)$ which is rich enough to act transitively on the set of configurations of $k$ points in $\R^n$ (e.g., the group of compactly supported diffeomorphisms) and suppose that we have chosen a right-invariant Riemannian metric on $\set{G}$. For technical reasons, such metrics are typically chosen in practice via a Hilbert space of time-dependent vector fields. That is, given smooth a path $\gamma:[0,1] \to \set{G}$, the space derivative $\frac{d}{dx} \gamma$ is a time-dependent vector field, so one can consider tangent vectors to $\set{G}$ as being determined by time-dependent vector fields; a Hilbert space structure on this space of vector fields can therefore be used to define a Riemannian metric on $\set{G}$. The details of such constructions are beyond the scope of this paper, and are not necessary in what follows, but we refer the reader to \cite[Section 8]{bauer2014overview} for an overview of the topic. Let $\set{H}$ denote the associated Hilbert space with norm $\|\cdot\|_\set{H}$. We define the \emph{geodesic weight} of a piecewise smooth path $\gamma:[0,1] \to \set{G}$ to be 
\[
\widehat{\weight}(\gamma) = \int_0^1 \left\| \frac{d}{dx} \gamma(t) \right\|^2_\set{H} \; dt.
\]

Let $\set{X} = \mathrm{Emb}(\set{M},\R^n)$, where $\set{M}$ is a collection of $k$ points. Then $\set{G}$ acts on $\set{X}$ transitively. Let $T:\cat{G} \to \cat{X}$ be the monoidal functor from the monoidal category associated to $\set{G}$ to the action groupoid (Definitions \ref{def:associated_monoidal} and \ref{def:action_groupoid}). We define a monoidal weight $\weight$ on $\cat{G}$ by 
\[
\weight(g) = \inf \left\{\widehat{\weight}(\gamma)^{\frac{1}{2}} \mid \gamma:[0,1] \to \set{G}, \; \gamma(0)= e, \; \gamma(1) = g\right\}.
\]
Then \cite[Equation (23)]{bauer2014overview} says that the geodesic distance between configurations $x,y \in \set{X}$ with respect to the Riemannian metric induced by the Riemannian metric on $\set{G}$ is given by 
\[
d(x,y) = \inf_{g \in \set{G}} \{\weight(g) \mid gx = y\}.
\]
Observing that $\weight(g) = \weight(g^{-1})$, Remark \ref{rem:inverse_weight} implies the following.

\begin{cor}\label{cor:configuration_space}
The geodesic distance on the configuration space of $k$ points in $\R^n$, defined as above, can be expressed as a  $\cat{G}$-interleaving distance.
\end{cor}

\subsubsection{Fundamental Groupoid and Length Metrics}\label{sec:fundamental_groupoid} We now give an example of an interleaving distance arising in the more general 2-functor setting. We recover a familiar metric construction in an admittedly convoluted manner; this section is mainly intended to illustrate how the pieces of the 2-functor definition come together. 

Fix a path connected topological space $\set{X}$ and let $\cat{G} = \pi_1(\set{X})$ denote its fundamental groupoid. We consider this as a 1-category by setting $\cat{G}_0 = \set{X}$ and, for $x,y \in \set{X}$, setting $\cat{G}_1(x,y)$ to be the collection of homotopy classes of continuous paths from $x$ to $y$ in $\set{X}$; we use the notation $x \xrightarrow{p} y$ for such a homotopy class. Composition of two homotopy classes is the homotopy class of the concatenation of the underlying paths. We promote this to a 2-category through the following general construction.

\begin{defn}[Indiscrete 2-Category]\label{def:indiscrete}
    Let $\cat{D}$ be a 2-category. The \emph{indiscrete 2-category associated to $\cat{D}$} is the 2-category\footnote{This notation is also intended to evoke classifying space---c.f.\ Remark \ref{rem:delooping}.} $\cat{E}\cat{D}$ with underlying 1-category equal to $\cat{D}$ and with $\cat{E}\cat{D}_2(f,g)$ consisting of a unique 2-morphism $\alpha_{f,g}$ for each $f,g \in \cat{D}_1(x,y)$. Horizontal and vertical composition laws are forced by uniqueness.
\end{defn}

Let $\cat{C} = \cat{E}\cat{G}$ be the indiscrete 2-category associated to the fundamental groupoid of $\set{X}$. We define a 2-functor $\Delta:\cat{C} \to \cat{Cat}$ as follows:
\begin{enumerate}
    \item For all $x \in \cat{C}_0 = \set{X}$, $\Delta(x) = \cat{B}\pi_1(\set{X},x)$, the delooping (Remark \ref{rem:delooping}) of the fundamental group of homotopy classes of loops in $\set{X}$ based at $x$. We denote the unique object of $\deloop \pi_1(X,x)$ as $\star_x$.
    \item For $p \in \cat{C}_1(x,y)$, we define $\Delta_p:\deloop \pi_1(\set{X},x) \rightarrow \deloop \pi_1(\set{X},y)$ to be the functor with $\Delta_p(\star_x) = \star_y$ and for a homotopy class $\ell$ of a loop based at $x$, $\Delta_p(\ell) = p\ell p^{-1}$. That is, $\Delta_p(\ell)$ is the (homotopy class of the) loop based at $y$ obtained by running the path $p$ in reverse, going around the loop $\ell$, then following $p$ in the forward direction. A schematic illustration is shown in Figure \ref{fig:fundamental_group}.
    \begin{figure}
        \centering
        \includegraphics[width = 0.5\textwidth]{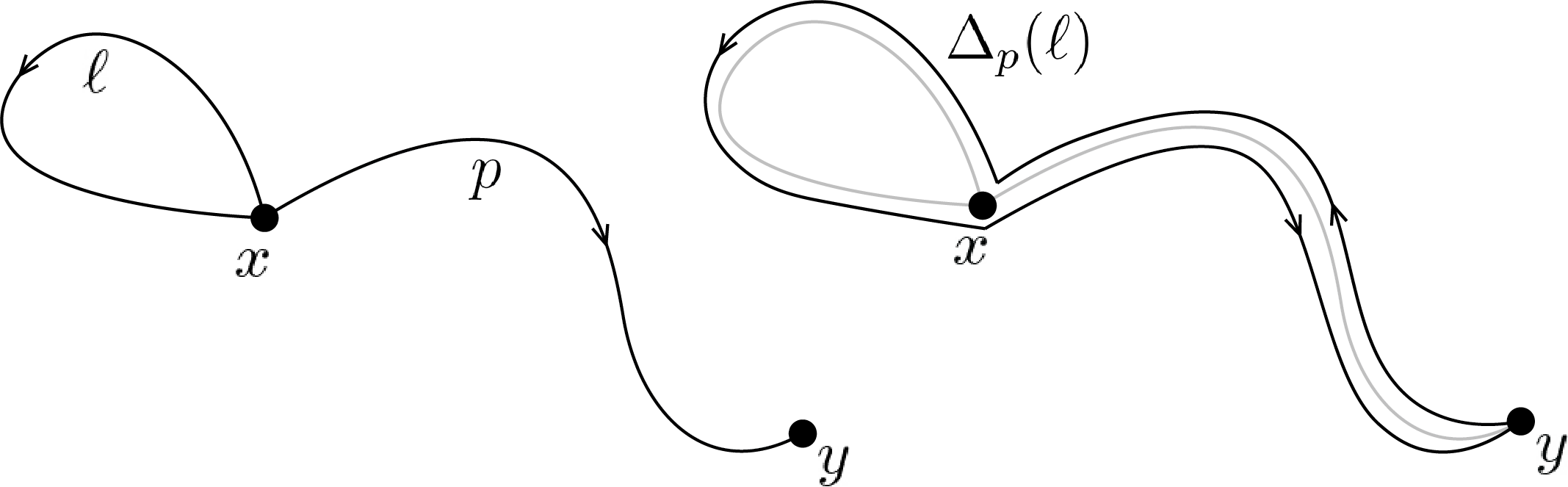}
        \caption{Schematic of the functor $\Delta_f$.}
        \label{fig:fundamental_group}
    \end{figure}
    \item For $p, q:x \to y$, define $\Delta(\alpha_{p,q})$ to be the natural transformation with component \[\Delta(\alpha_{p,q})_{\star_x} = qp^{-1}.\]
\end{enumerate}

Now suppose that $\set{X}$ comes with a \emph{length structure}: a function $\weight:C^0([0,1],\set{X}) \to \R$ from the collection $C^0([0,1],\set{X})$ of all continuous paths $\tilde{p}:[0,1] \to \set{X}$ to $\R$ such that $\weight(\tilde{p}) = 0$ if $\tilde{p}$ is a constant path and such that $\weight$ is subadditive under concatenation of paths: $\weight(\tilde{q} \circ \tilde{p}) \leq \weight(\tilde{q}) + \weight(\tilde{p})$. This descends to a Lawvere weight on $\cat{G}$, which we still denote $\weight$ by abuse of notation, defined by
\[
\weight(p) = \inf \{\weight(\tilde{p}) \mid \mbox{$p$ is the homotopy class of $\tilde{p}$}\}.
\]
We leave it to the reader to verify that the induced interleaving distance on $\mathrm{Im}(\Delta) \approx \set{X}$ is given by
\[
d_\Delta(x,y) = \inf \{\weight(\tilde{p}) \mid \tilde{p} \in C^0([0,1],\set{X}), \, \tilde{p}(0) = x, \, \tilde{p}(1) = y\},
\]
the usual metric induced by a length structure.

\section{Induced 2-Functors and Generalized Persistence Modules}\label{sec:interleaving_group_actions}

We now provide new examples of interleaving distances in the context of TDA, arising from actions of  monoids on posets. Doing so will require a Yoneda-like construction, which may be interesting in more general settings. Our main interest is in the $\cat{G}$-interleaving distance setting, with input data consisting of a monoidal functor. For the sake of generality, we give the Yoneda-like construction at the level of general 2-functors.

\subsection{Generalized Persistence Modules}
Early work in TDA considered persistent homology primarily from a computational perspective, with the barcode invariant of a function defined in a rather combinatorial fashion~\cite{edelsbrunner2002topological} (although earlier works such as \cite{cagliari2001size} already considered categorical aspects). The perspective that the fundamental objects in persistent homology theory are persistence modules was subsequently developed in \cite{chazal2009proximity}, and precisely formalized in categorical language in \cite{bubenik2014categorification}. Much of the theory of TDA can be expressed at the level of \emph{generalized persistence modules}~\cite{bubenik2015metrics}, which are functors of the form $M:\cat{P} \to \cat{Y}$, where $\cat{Y}$ is an arbitrary category and $\cat{P}$ is a \emph{preorder category}; that is $\cat{P}_0 = \set{P}$, where $(\set{P},\leq)$ is a preordered set, or \emph{proset}, and there is a unique morphism in $\cat{P}_1(p,q)$ if and only if $p \leq q$. For convenience, we denote this morphism simply as $p \leq q$.  

In the following, let $\cat{Fun}(\cat{C},\cat{D})$ denote the \emph{1-category of functors} from a category $\cat{C}$ to a category $\cat{D}$ (i.e., objects are functors and morphisms are natural transformations). Then the objects of $\cat{Fun}(\cat{P},\cat{Y})$, with $\cat{P}$ a preorder category, are generalized persistence modules. In~\cite{bubenik2015metrics}, two notions of distance are defined on the space of generalized persistence modules $\cat{Fun}(\cat{P},{Y})_0$---more details of this construction are given in Section \ref{sec:bubenik-daSilva-Scott-interleaving}. The main result of this section (Theorem \ref{thm:generalized_persistence_module_interleaving}) shows that any monoid action on a poset gives rise to a family of interleaving distances on the space of generalized persistence modules (with fixed target category). To state and prove the result, we introduce a general construction for inducing new 2-functors from old.

\subsection{2-Functors Associated to a Fixed 2-Functor Together with an Auxiliary Category}\label{sect:Yoneda_construction}

    Let $\Delta: \cat{C} \to \cat{Cat} $ be a 2-functor and fix a category $\cat{Y} \in (\cat{Cat})_0.$ Let $\cat{C}^{\mathrm{op}}$ denote the 2-category whose 1-morphisms are reversed (but whose 2-morphisms are not). We define an associated 2-functor $\DeltaYY: \cat{C}^{\mathrm{op}} \to \cat{Cat}$ according to the following proposition. We leave the proof, which is straightforward, to the reader.

\begin{prop}\label{prop:yoneda_construction}
    Let $\Delta:\cat{C} \to \cat{Cat}$ be a 2-functor and let $\cat{Y}$ be some fixed 1-category. Then one obtains a well-defined 2-functor $\Delta^\cat{Y}:\cat{C}^{\mathrm{op}} \to \cat{Cat}$ with:
    \begin{enumerate}
        \item for $A \in \cat{C}_0^{\mathrm{op}} = \cat{C}_0$, $\Delta^\cat{Y}(A) = \cat{Fun}(\Delta(A),\cat{Y})$;
        \item for $f \in \cat{C}_1^{\mathrm{op}}(A',A) = \cat{C}_1(A,A')$, 
        \[
        \Delta_f^\cat{Y}(H:\Delta(A') \rightarrow \cat{Y}) = H \circ \Delta_f:\Delta(A) \to \cat{Y}
        \]
        and
        \[
        \Delta_f^\cat{Y}(\lambda:H \Rightarrow H') = \lambda \bullet \Delta_f: \Delta_f^\cat{Y}(H) \Rightarrow \Delta_f^\cat{Y}(H');
        \]
        \item for $\alpha \in \cat{C}_2^{\mathrm{op}}(f,g) = \cat{C}_2(f,g)$, $f,g \in \cat{C}_1^{\mathrm{op}}(A',A)$, $H: \Delta(A') \to \cat{Y}$, $X \in \Delta(A)_0$,
        \[
        \big(\Delta^\cat{Y}(\alpha)_H\big)_X = H \circ \Delta(\alpha)_X.
        \]
    \end{enumerate}
\end{prop}

Recall the correspondence between monoidal functors and 2-functors given by identifying $T:\cat{G} \to \cat{End}(\cat{X})$ with $\cat{B}T:\cat{B}\cat{G} \to \cat{Cat}$ (Proposition \ref{prop:monoidal_functors_and_2_functors}). Using the fact that $\cat{B}\cat{G}^{\mathrm{op}} = \cat{B}\cat{G}$, Proposition \ref{prop:yoneda_construction} has an immediate corollary in the monoidal functor setting.

\begin{cor}\label{cor:monoidal_yoneda}
        Let $T:\cat{G} \to \cat{End}(\cat{X})$ be a monoidal functor and let $\cat{Y}$ be some fixed 1-category. Then one obtains a well-defined monoidal functor $T^\cat{Y}:\cat{G} \to \cat{End}(\cat{Fun}(\cat{X},\cat{Y}))$ with:
    \begin{enumerate}
        \item for $f \in \cat{G}_0$, $T_f^{\cat{Y}} \in \cat{End}(\cat{Fun}(\cat{X},\cat{Y}))$ is defined by
        \[
        T^\cat{Y}_f(H:\cat{X} \rightarrow \cat{Y}) = H \circ T_f: \cat{X} \to \cat{Y}
        \]
        and
        \[
        T_f^\cat{Y}(\lambda:H \Rightarrow H') = \lambda \bullet T_f: T_f^\cat{Y}(H) \Rightarrow T_f^\cat{Y}(H');
        \]
        \item for $\alpha \in \cat{G}_1(f,g)$, $H: \cat{X} \to \cat{Y}$, $X \in \cat{X}_0$,
        \[
        \big(T^\cat{Y}(\alpha)_H\big)_X = H \circ T(\alpha)_X.
        \]
    \end{enumerate}
\end{cor}

\subsection{Proset 2-Functors and Induced Interleaving Distances.}\label{sec:poset_2_functors} Returning to the setting of interleavings of generalized persistence modules, let $\set{G}$ be a monoid acting on a proset $(\set{P},\leq)$. That is, for all $g \in \set{G}$ and $p \in \set{P}$, there is an associated $g(p) \in \set{P}$ such that $h(g(p)) = (hg)(p)$ and $e(p) = p$, and the maps are \emph{monotone} in the sense that $p \leq q$ implies $g(p) \leq g(q)$. 

\begin{defn}[Proset Monoidal Category]\label{def:proset_monoidal_category}
    The \emph{proset monoidal category associated to $\set{G}$} is the monoidal category $\cat{G}^\set{P}$ with:
    \begin{enumerate}
        \item $\cat{G}^\set{P}_0 = G$, with product given by the monoid law of $G$.
        \item For 1-morphisms $g,h$, let
        \[
        \cat{G}^\set{P}_1(g,h) = \left\{\begin{array}{cl}
        \{\alpha_{g,h}\} & \mbox{if $g(p) \leq h(p)$ for all $p \in \set{P}$,} \\
        \emptyset & \mbox{otherwise.}
        \end{array}\right.
        \]
        Composition is given by $\alpha_{h,k}\alpha_{g,h} = \alpha_{g,k}$; this is well-defined, since the existence of $\alpha_{g,h}$ and $\alpha_{h,k}$ imply $g(p) \leq h(p) \leq k(p)$ for all $p$. The tensor product is given by $\alpha_{f,g} \otimes \alpha_{h,k} = \alpha_{fh,gk}$. To see that this is well-defined, observe that $h(p) \leq k(p)$ and $f(p) \leq g(p)$ for all $p \in \set{P}$. This implies that $fh(p) \leq gh(p)$ and that $gh(p) \leq gk(p)$ (since $g$ is order-preserving), so we have $fh(p) \leq gk(p)$.
    \end{enumerate}
\end{defn}

\begin{defn}[Proset Action Monoidal Functor]\label{def:preordered_set_action_functor}
Let $\cat{P}$ be the proset category associated to $\set{P}$. We define the \emph{proset action} monoidal functor $T:\cat{G}^\set{P} \to \cat{End}(\cat{P})$ as follows:
\begin{enumerate}
    \item For $g \in \set{G}$, $T_g:\cat{P} \to \cat{P}$ is the functor defined on objects by $T_g(p) = g(p)$ and on morphisms by $T_g(p \leq q) = g(p) \leq g(q)$.
    \item For a morphism $\alpha_{g,h}$, $T(\alpha_{g,h}):T_g \Rightarrow T_h$ is the natural transformation with component at $p \in \set{P}$ given by $T(\alpha_{g,h})_p = g(p) \leq h(p)$. 
\end{enumerate}
\end{defn}

This leads to the main result of this section. 

\begin{thm}\label{thm:generalized_persistence_module_interleaving}
    Let $\set{P}$ be a proset with associated proset category $\cat{P}$ and let $\cat{Y}$ be some other category. Any action on $\set{P}$ by a monoid endowed with a monoidal weight gives rise to an extended pseudometric on the space of generalized persistence modules $\cat{Fun}(\cat{P},\cat{Y})_0$.
\end{thm}

\begin{proof}
    Let $\set{G}$ be a monoid acting on $\set{P}$, endowed with a monoidal weight $\weight$, and let $T:\cat{G}^\set{P} \to \cat{End}(\cat{P})$ be the proset action monoidal functor. Corollary \ref{cor:monoidal_yoneda} yields an induced monoidal functor $T^\cat{Y}:\cat{G}^\set{P} \to \cat{End}(\cat{Fun}(\cat{P},\cat{Y}))$. By Remark \ref{rem:monoidal_weight_monoid} and the defintion of $\cat{G}^\set{P}$, a monoidal weight on $\set{G}$ can be considered as a monoidal weight on $\cat{G}^\set{P}$. Theorem \ref{thm:actegory_interleaving} says that the $\cat{G}^\set{P}$-interleaving distance $d_{T^\cat{Y},\weight}$, defines an extended pseudometric on $\cat{Fun}(\cat{P},\cat{Y})_0$. 
\end{proof}

\begin{rem}[Explicit $\cat{G}$-Interleavings]\label{rem:explicit_2_functor_interleavings}
    We can describe the metrics which arise from the construction in Theorem \ref{thm:generalized_persistence_module_interleaving} somewhat explicitly. Given an action of a monoid $\set{G}$, endowed with a monoidal weight $\weight$, on a proset $\set{P}$, the general definition of $\cat{G}^\set{P}$-interleaving says that generalized persistence modules $M,N:\cat{P} \to \cat{Y}$ are $(g,h)$-interleaved (for $g,h \in \set{G}$) with respect to $T^\cat{Y}$ (where $T:\cat{G}^\set{P} \to \cat{End}(\cat{P})$ is the proset action monoidal functor) if there are natural transformations (i.e., morphisms in $\cat{Fun}(\cat{P},\cat{Y})_1$) $\phi:M \Rightarrow T^\cat{Y}_g(N)$ and $\psi:N \Rightarrow T^{\cat{Y}}_h(M)$ and 2-morphisms $\alpha_{e,gh}$ and $\alpha_{e,hg}$ in $\cat{G}^\set{P}_2$ such that the diagrams
    \[\begin{tikzcd}
	{M} && {T^\mathsf{Y}_{gh}(M)} && {N} && {T^\mathsf{Y}_{hg}(N)} \\
	& {T^\mathsf{Y}_g(N)} &&&& {T^\mathsf{Y}_h(M)}
	\arrow["{\phi}"', Rightarrow, from=1-1, to=2-2]
	\arrow["{T^\mathsf{Y}_g(\psi)}"', Rightarrow, from=2-2, to=1-3]
	\arrow["{T^\mathsf{Y}(\alpha_{e,gh})_M}", Rightarrow, from=1-1, to=1-3]
	\arrow["{\psi}"', Rightarrow, from=1-5, to=2-6]
	\arrow["{T^\mathsf{Y}_h(\phi)}"', Rightarrow, from=2-6, to=1-7]
	\arrow["{T^\mathsf{Y}(\alpha_{e,hg})_N}", Rightarrow, from=1-5, to=1-7]
\end{tikzcd}\]
    commute. This means that for all $p \in \set{P}$, the diagrams
\[\begin{tikzcd}
	{M(p)} && {T^\mathsf{Y}_{gh}(M)(p)} && {N(p)} && {T^\mathsf{Y}_{hg}(N)(p)} \\
	& {T^\mathsf{Y}_g(N)(p)} &&&& {T^\mathsf{Y}_h(M)(p)}
	\arrow["{\phi_p}"', from=1-1, to=2-2]
	\arrow["{T^\mathsf{Y}_g(\psi)_p}"', from=2-2, to=1-3]
	\arrow["{\big(T^\mathsf{Y}(\alpha_{e,gh})_M\big)_p}", from=1-1, to=1-3]
	\arrow["{\psi_p}"', from=1-5, to=2-6]
	\arrow["{T^\mathsf{Y}_h(\phi)_p}"', from=2-6, to=1-7]
	\arrow["{\big(T^\mathsf{Y}(\alpha_{e,hg})_N\big)_p}", from=1-5, to=1-7]
\end{tikzcd}\]
    commute. Further unpacking the definitions, this is the case if and only if the following diagrams commute
\begin{equation}\label{eqn:generalized_module_interleaving}\begin{tikzcd}
	{M(p)} && {M(gh(p))} && {N(p)} && {N(hg(p))} \\
	& {N(g(p))} &&&& {M(h(p))}
	\arrow["{\phi_p}"', from=1-1, to=2-2]
	\arrow["{\psi_{g(p)}}"', from=2-2, to=1-3]
	\arrow["{M(p \leq gh(p))}", from=1-1, to=1-3]
	\arrow["{\psi_p}"', from=1-5, to=2-6]
	\arrow["{\phi_{h(p)}}"', from=2-6, to=1-7]
	\arrow["{N(p \leq hg(p))}", from=1-5, to=1-7]
\end{tikzcd}\end{equation}
    for all $p \in \set{P}$. Therefore, the interleaving distance of Theorem \ref{thm:generalized_persistence_module_interleaving} is given by
    \[
    d_{T^\cat{Y},\weight}(M,N) = \inf \{\max\{\weight(g),\weight(h)\} \mid \mbox{$\exists$ $\phi,\psi$ s.t. the diagrams \eqref{eqn:generalized_module_interleaving} commute $\forall$ $p \in \set{P}$}\}.
    \]
\end{rem}

\subsection{Examples of 2-Functor Interleaving Distances for Generalized Persistence Modules}\label{subsec:examples_of_2_functor_interleaving}

We now present some examples of interleaving distances arising from  Theorem \ref{thm:generalized_persistence_module_interleaving}.

\subsubsection{Bubenik-de Silva-Scott Interleaving Distances}\label{sec:bubenik-daSilva-Scott-interleaving} We now fill in some details of the interleaving constructions of Bubenik, de Silva and Scott in \cite{bubenik2015metrics}, first mentioned at the beginning of this section.

\begin{defn}[$\omega$-Interleaving Distance~\cite{bubenik2015metrics}]
    Let $(\set{P},\leq)$ be a proset. A function $g:\set{P} \to \set{P}$ is called \emph{monotone} if $p \leq q$ implies $g(p) \leq g(q)$. A monotone function $g$ is called a \emph{translation} if $p \leq g(p)$ for all $p \in \set{P}$. The collection of translations of $\set{P}$ is denoted $\mathrm{Trans}_\set{P}$; it is easy to show that $\mathrm{Trans}_\set{P}$ is a monoid. The identity is the identity map $e_\set{P}$. A \emph{sublinear projection} is a map $\omega:\mathrm{Trans}_\set{P} \to [0,\infty]$ such that $\omega(e_\set{P}) = 0$ and $\omega(h \circ g) \leq \omega(h) + \omega(g)$. 
    
    The above data leads to an extended pseudometric on $\cat{Fun}(\cat{P},{Y})_0$. Given $M \in \cat{Fun}(\cat{P},\cat{Y})_0$ and $g \in \mathrm{Trans}_\set{P}$, let $gM \in \cat{Fun}(\cat{P},\cat{Y})_0$ be the functor with $gM(p) = M(g(p))$ and $gM(p \leq q) = M(g(p) \leq g(q))$. There is a natural transformation $\eta^M_g:M \Rightarrow gM$ with component $(\eta^M_g)_p:M(p) \rightarrow M(g(p))$ induced by $p \leq g(p)$. We say that functors $M,N:\cat{P} \to \cat{Y}$ are \emph{$t$-interleaved with respect to $\omega$} if there are translations $g$ and $h$ of $\set{P}$ with $\omega(g),\omega(h) \leq t$ and natural transformations $\phi:M \Rightarrow gN$ and $\psi:N \Rightarrow hM$ such that the diagrams
    \[\begin{tikzcd}
    	{M(p)} & {} & {M(gh(p))} && {N(p)} && {N(hg(p))} \\
    	& {N(g(p))} &&&& {M(h(p))}
    	\arrow["{\phi_p}"', from=1-1, to=2-2]
    	\arrow["{\phi_{g(p)}}"', from=2-2, to=1-3]
    	\arrow["{M(p \leq gh(p))}", from=1-1, to=1-3]
    	\arrow["{\psi_p}"', from=1-5, to=2-6]
    	\arrow["{\phi_{h(p)}}"', from=2-6, to=1-7]
    	\arrow["{N(p \leq hg(p))}", from=1-5, to=1-7]
    \end{tikzcd}\]
    commute for all $p \in \set{P}$. 
    The \emph{$\omega$-interleaving distance} between $M$ and $N$ is
        \[
        d_\omega(M,N) = \inf\{t \geq 0 \mid \mbox{$M$ and $N$ are $t$-interleaved}\}.
        \]
\end{defn}

Comparing this definition with Remark \ref{rem:explicit_2_functor_interleavings}, the following result becomes clear.

\begin{prop}\label{prop:bubenik_de_silva_scott_interleaving}
    Let $(\set{P},\leq)$ be a proset and let $\omega$ be a sublinear family for $\mathrm{Trans}_\mathcal{P}$. Let $\cat{Y}$ be an arbitrary category. Then the 2-functor interleaving distance $d_{T^\cat{Y},\weight}$ from Theorem \ref{thm:generalized_persistence_module_interleaving} is equal to the $\omega$-interleaving distance $d_\omega$ on $\cat{Fun}(\cat{P},\cat{Y})_0$ when we take the monoid to be $\set{G} = \mathrm{Trans}_\set{P}$ and the monoidal weight to be $\weight = \omega$.
\end{prop}

\begin{rem}
    It is shown in \cite[Theorem 3.5]{de2018theory} that another Bubenik-de Silva-Scott construction of interleaving distance, depending on a structure called a \emph{superlinear family of translations}, can be realized as a flow interleaving distance. While the interleaving distances induced by sublinear projections and superlinear families are sometimes related \cite[Theorem 3.24]{bubenik2015metrics}, the concepts are, in general, distinct.
    We showed in Proposition \ref{prop:equivalence_to_flow} that flow interleaving distances can always be realized as 2-functor interleaving distances. Putting all of this together, we see that the $\cat{G}$-interleaving framework fully generalizes the interleaving distances of Bubenik, de Silva and Scott.
\end{rem}

\subsubsection{Multiplication Group}\label{sec:multiplication_group} We next present a simple example of an interleaving distance arising in the setting of Theorem \ref{thm:generalized_persistence_module_interleaving} which differs from the classical notion in TDA.

Let $(\set{P},\leq)$ be the poset of non-negative real numbers. Let $\set{G}$ be the group of positive real numbers under multiplication. Then $\set{G}$ acts on $\set{P}$ (but does not contain and is not contained in the monoid of translations $\mathrm{Trans}_\set{P}$). We define a monoidal weight $\weight$ on $\cat{G}^\set{P}$ by 
$
\weight(t) = |\log(t)|.
$
Then we obtain an interleaving distance on $ d_{T^\cat{Y},\weight}$ on any space of generalized persistence modules $\cat{Fun}(\cat{P},\cat{Y})_0$. 

\begin{ex}\label{ex:multiplication_group}
To see that the behavior of the induced interleaving distance is qualitatively different than that of classical interleaving distance, let us consider a basic example, where $\cat{Y} = \cat{Vec}$, the category of vector spaces over (say) $\R$. Functors $\cat{P}\to \cat{Vec}$ are herein referred to as \emph{classical persistence modules}, as they are the basic objects in the foundational theory of TDA \cite{chazal2009proximity}. Taking the sublinear projection $\omega:\mathrm{Trans}_\mathcal{P} \to \R$ defined by $\omega(g) = \sup\{g(p)-p \mid p \in \set{P}\}$ yields the \emph{classical interleaving distance} for classical persistence modules, as studied in \cite{chazal2009proximity}. Under mild tameness assumptions, this distance agrees with the \emph{bottleneck distance} $d_\mathrm{B}$ between persistence diagrams (see, e.g., \cite{chazal2009proximity,lesnick2015theory,bauer2015induced}), which is defined as a more straightforward optimization problem and was an essential component in the early research on TDA \cite{cohen2005stability}.

For $a \leq b$, let $M_{[a,b)}:\cat{P} \to \cat{Vec}$ denote the \emph{interval module} with
\[
M_{[a,b)}(r) = \left\{
\begin{array}{cl}
\R & \mbox{if $r \in [a,b)$} \\
0 & \mbox{otherwise.}
\end{array}\right.
\]
and where $M_{[a,b)}(r \leq s)$ is the identity map on $\R$ whenever possible and is otherwise the zero map. If $a = b$, $M_{[a,a)} = M_\emptyset$, the persistence module which is identically zero. It is well-known that the classical interleaving distance (or bottleneck distance) between $M_{[a,b)}$ and $M_\emptyset$ is given by $d_\mathrm{B}(M_{[a,b)},M_\emptyset) = \frac{b-a}{2}$. On the other hand, one can show that the $\cat{G}^\set{P}$-interleaving distance is
\[
d_{T^\cat{Y},\weight}(M_{[a,b)},M_\emptyset) = \frac{1}{2}(\log(b) - \log(a)).
\]
The classical interleaving distance and $\cat{G}^\set{P}$-interleaving distance behave qualitatively differently; for example
\[
d_\mathrm{B}(M_{[a,a+1)},M_\emptyset) = \frac{1}{2}
\]
for all $a$, whereas
\[
\lim_{a \rightarrow \infty} d_{T^\cat{Y},\weight}(M_{[a,a+1)},M_\emptyset) = 0.
\]
\end{ex}

\begin{rem}\label{rem:interpretation_of_multiplication_group}
 We end this subsection with an informal remark, geared towards researchers dealing with practical applications of TDA. The above example is meant to be a simple expository one, and we do not advocate it as a serious replacement for standard metrics in TDA. However, even this toy example suggests the potential of exploring alternate metrics in a data analysis context. An interval module $M_{[a,b)}$ is typically interpreted in a TDA setting as a persistent topological feature in a data set, subject to some notion of filtration (e.g., persistent homology barcodes of a Vietoris-Rips complex associated to point cloud data; see \cite{carlsson2014topological}). The distance from one of these interval modules to the zero module with respect to the bottleneck distance---i.e., $\frac{1}{2}(b-a)$---can be understood as the ``importance" of a feature. The behavior exhibited above by $d_{T^\cat{Y},\weight}$ says that the importance of a feature with respect to this metric not only depends on its persistence (i.e., distance between endpoints), but on the temporal location of the feature---that is, the ``birth time", or lower limit of the interval. The interleaving distance induced by the multiplication group says that features born near zero become infinitely important, while features with a fixed length have importance decaying to zero as the birth time increases. This agrees with the assertion of the recent paper~\cite{bobrowski2023universal} that the ratio $b/a$ is a more relevant indicator of statistical significance of a topological feature $M_{[a,b)}$ than the traditional persistence $\frac{1}{2}(b-a)$. 
 
 We believe that it would be useful to incorporate different importance  dependencies into metrics, depending on the data analysis context. Such a focus on learning appropriate TDA metrics based on data structure is advocated for in, e.g., \cite{chacholski2020metrics}, and our framework gives a rich theory-backed source of such metrics. Designing data-driven metrics based on these generalized interleaving distances will be the subject of future research.
\end{rem}

\subsubsection{Diffeomorphism Group of $\R$} We now consider the case where $(\set{P},\leq)$ is the real line, but extend the class of monoids under consideration to diffeomorphism groups; these are well-studied by certain communities in geometric data science \cite{srivastava2010shape,bauer2014overview,srivastava2016functional,needham2020simplifying,bauer2022elastic}, and the ideas presented here are meant to serve as a bridge between these techniques in TDA and geometric shape analysis (see also Section \ref{sec:landmarks}). 

Let $\mathrm{Diff}_c(\R)$ denote the group of orientation-preserving, compactly supported (i.e., differing from the identity map only on a compact set) smooth diffeomorphisms of $\R$ and let $\set{G}$ be a subgroup of $\mathrm{Diff}_c(\R)$. Then $\set{G}$ acts on $\set{P}$ (by $g \cdot p = g(p)$), so Theorem \ref{thm:generalized_persistence_module_interleaving} implies that we obtain an interleaving distance on the space of generalized persistence modules once we have chosen a monoidal weight on $\set{G}$. Recall from Proposition \ref{prop:right_invariant_weight} that such a weight is induced by any right-invariant metric on $\set{G}$. This discussion is summarized in the following proposition.

\begin{prop}\label{prop:diffeos_to_interleaving}
    Let $(\set{P},\leq)$ be the real line and let $\set{G}$ be a subgroup of the compactly-supported orientation-preserving diffeomorphism group $\mathrm{Diff}_c(\R)$. Any choice of right-invariant metric on the diffeomorphism group gives rise to an interleaving distance on the space of generalized persistence modules $\cat{Fun}(\cat{P},\cat{Y})_0$, where $\cat{Y}$ is any category.
\end{prop}

\begin{rem}
    A natural question is whether any right-invariant metrics on $\mathrm{Diff}_c(\R)$ exist, and, if so, whether they are computationally feasible to work with---this matter was discussed briefly for diffeomorphisms of $\R^n$ in Section \ref{sec:landmarks}, but we add more discussion here, since there are simpler right-invariant metrics in this one-dimensional setting. The details of the construction of these metrics and their computational approximation require substantial knowledge of functional analysis and optimization theory, and are beyond the scope of the present paper. We sketch briefly the main ideas below, but their development remains a direction of future research.

    The group $\mathrm{Diff}_c(\R)$ is a Lie group, but it is infinite-dimensional, modeled locally on a Fr\'{e}chet vector space \cite{hamilton1982inverse}. There are several candidates for right-invariant metrics in the literature, and we describe one family of them here. Consider the weak (see, e.g., \cite{abraham2012manifolds}) Riemannian metric $\langle \cdot, \cdot \rangle$ on $\mathrm{Diff}_c(\R)$ defined at a basepoint $g \in \mathrm{Diff}_c(\R)$ and tangent vectors $x,y \in T_g \mathrm{Diff}_c(\R)$ (the tangent space is isomorphic to the space of smooth maps $\R \to \R$ which are equal to zero off a compact set \cite[Section 43.1]{kriegl1997convenient}) by
    \[
    \langle x, y \rangle_g = \int_\R \sum_{j=0}^k a_j \cdot D_g^j x(t) \cdot D_g^j y(t) \cdot g'(t) dt,
    \]
    where $D_g$ is the operator $\frac{1}{g'(t)}\frac{d}{dt}$, $k$ is a nonnegative integer and the $a_j$'s are non-negative real number hyperparameters. Riemannian metrics of this form are referred to as \emph{Sobolev-type} metrics, and are studied intensely by researchers interested in applications of Riemannian geometry to shape analysis~\cite{michor2007overview,bauer2014overview,bruveris2014geodesic}, although these works tend to focus on similar metrics on spaces of immersions of $\R$ or the unit circle into a higher-dimensional Euclidean space. It is straightforward to check that geodesic distance with respect to this metric is a right-invariant metric on $\mathrm{Diff}_c(\R)$, as desired. Moreover, there has been substantial development of computational pipelines for computing geodesic distances for Sobolev-type metrics~\cite{bauer2017numerical} (once again focusing on the setting of immersions into higher-dimensional spaces). The existence of theory and numerical pipelines for shape analysis applications suggest the potential for developing the application of these metrics to interleaving distances in future work.
\end{rem}

\subsubsection{Multiparameter Persistence} Multiparameter persistence~\cite{lesnick2015theory,botnan2022introduction} generally concerns generalized persistence modules where the proset $(\set{P},\leq)$ is the poset $\R^n$, endowed with the product partial order: $(r_1,\ldots,r_n) \leq (s_1,\ldots,s_n)$ if and only if $r_i \leq s_i$ for all $i$. To distinguish vectors from scalars, we use the notation $\mathbf{r} = (r_1,\ldots,r_n)$ in this subsection. We use the notation $M:\cat{P} \to \cat{Y}$ for a multiparameter persistence module. 

The standard notion of interleaving distance between multiparameter persistence modules (see, e.g., \cite{botnan2022introduction}) is easiest to state using the notion of a category with a flow: the \emph{standard flow} $T^\mathrm{st}:\R_{\geq 0} \rightarrow \cat{End}(\cat{Fun}(\cat{P},\cat{Y}))$ is given by 
\[
T^\mathrm{st}_t M(\mathbf{r}) = M(\mathbf{r} + t \mathbf{1})
\]
where $\mathbf{1}$ is the vector of all ones. For $s \leq t$, $(T^\mathrm{st}_{s\leq t})_M:T^\mathrm{st}_s M \Rightarrow T^\mathrm{st}_t M$ has component at $\mathbf{r}$ given by
\[
M(\mathbf{r} + s \mathbf{1} \leq \mathbf{r} + t \mathbf{1}).
\]
Then the interleaving distance between $M$ and $N$ is the infimum over $t$ such that there exist natural transformations $\phi:M \Rightarrow T^\mathrm{st}_t N$ and $\psi:N \Rightarrow T^\mathrm{st}_t M$ so that the diagrams
\[\begin{tikzcd}
	{M(\mathbf{r})} && {M(\mathbf{r} + 2t\mathbf{1})} && {N(\mathbf{r})} && {N(\mathbf{r} + 2t\mathbf{1})} \\
	& {N(\mathbf{r} + t \mathbf{1})} &&&& {M(\mathbf{r} + t\mathbf{1})}
	\arrow["{\phi_\mathbf{r}}"', from=1-1, to=2-2]
	\arrow["{\psi_{\mathbf{r} + t\mathbf{1}}}"', from=2-2, to=1-3]
	\arrow["{M(\mathbf{r} \leq \mathbf{r} + 2t\mathbf{1})}", from=1-1, to=1-3]
	\arrow["{\psi_{\mathbf{r}}}"', from=1-5, to=2-6]
	\arrow["{\phi_{\mathbf{r} + t \mathbf{1}}}"', from=2-6, to=1-7]
	\arrow["{N(\mathbf{r} \leq \mathbf{r} + 2t \mathbf{1})}", from=1-5, to=1-7]
\end{tikzcd}\]
commute for all $\mathbf{r} \in \R^n$. We denote this interleaving distance as $d_{T^\mathrm{st}}^\mathrm{Fl}$ for the remainder of this subsection.

Using the construction of Theorem \ref{thm:generalized_persistence_module_interleaving}, we can define interleaving distances which are natural, but more refined than the version described above. Let $\set{G} = \R^n$ denote the group (under addition) of vectors in $\R^n$ with non-negative entries. Then $\set{G}$ acts on $\set{P}$ via $\mathbf{r} \cdot \mathbf{p} = \mathbf{p} + \mathbf{r}$. The function $\weight:\set{G} \to \R$ defined by $\weight(\mathbf{t}) = \|\mathbf{t}\|_p$ ($\|\cdot\|_p$ denoting the Minkowski $p$-norm) gives a Lawvere weight on $\cat{G}^\set{P}$. This yields an interleaving distance on $\cat{Fun}(\cat{P},\cat{Y})_0$. Explicitly, let $T:\cat{G}^\set{P} \to \cat{End}(\cat{Fun}(\cat{P},\cat{Y}))$ be the monoidal functor from Theorem \ref{thm:generalized_persistence_module_interleaving}; from Remark \ref{rem:explicit_2_functor_interleavings}, multiparameter persistence modules $M,N: \cat{P} \to \cat{Y}$ are $(\mathbf{s},\mathbf{t})$-interleaved if there exist natural transformations $\phi:M \Rightarrow T_\mathbf{s} N$ and $\psi:N \Rightarrow T_\mathbf{t} M$ such that the diagrams
\[\begin{tikzcd}
	{M(\mathbf{r})} && {M(\mathbf{r} + \mathbf{s} + \mathbf{t})} && {N(\mathbf{r})} && {N(\mathbf{r} + \mathbf{t} + \mathbf{s})} \\
	& {N(\mathbf{r} + \mathbf{s})} &&&& {M(\mathbf{r} + \mathbf{t})}
	\arrow["{\phi_\mathbf{r}}"', from=1-1, to=2-2]
	\arrow["{\psi_{\mathbf{r} + \mathbf{s}}}"', from=2-2, to=1-3]
	\arrow["{M(\mathbf{r} \leq \mathbf{r} + \mathbf{s} + \mathbf{t})}", from=1-1, to=1-3]
	\arrow["{\psi_{\mathbf{r}}}"', from=1-5, to=2-6]
	\arrow["{\phi_{\mathbf{r} + \mathbf{t}}}"', from=2-6, to=1-7]
	\arrow["{N(\mathbf{r} \leq \mathbf{r} + \mathbf{t} + \mathbf{s})}", from=1-5, to=1-7]
\end{tikzcd}\]
commute for all $\mathbf{r} \in \R^n$. The associated $\cat{G}^\set{P}$-interleaving distance $d_{T^\cat{Y},\weight}$ is then expressed as the infimum of $\max\{\|\mathbf{s}\|_p,\|\mathbf{t}\|_p\}$ such that there is an $(\mathbf{s},\mathbf{t})$-interleaving. 

\begin{ex}\label{ex:rectangle_modules}
    To see that the $\cat{G}^\set{P}$-interleaving distance differs qualitatively from the standard one, let us consider a simple example. For $\mathbf{a},\mathbf{b} \in \R^n$ with $\mathbf{a} \leq \mathbf{b}$, let $[\mathbf{a},\mathbf{b})$ denote the set of $\mathbf{c} \in \R^n$ such that $\mathbf{a} \leq \mathbf{c} < \mathbf{b}$.  Similar to the one-parameter setting of Section \ref{sec:multiplication_group}, we consider \emph{interval modules} (also called \emph{rectangle modules}~\cite{botnan2020rectangle}) of the form $M_{[\mathbf{a},\mathbf{b})}:\cat{P} \to \cat{Vec}$, where 
\[
M_{[\mathbf{a},\mathbf{b})}(\mathbf{r}) = \left\{
\begin{array}{cl}
\R & \mbox{if $\mathbf{r} \in [\mathbf{a},\mathbf{b})$} \\
0 & \mbox{otherwise,}
\end{array}\right.
\]
and where $M_{[\mathbf{a},\mathbf{b})}(\mathbf{r} \leq \mathbf{s})$ is the identity on $\R$ whenever possible, and otherwise the zero map. 

\begin{figure}
    \centering
    \includegraphics[width = 0.3\textwidth]{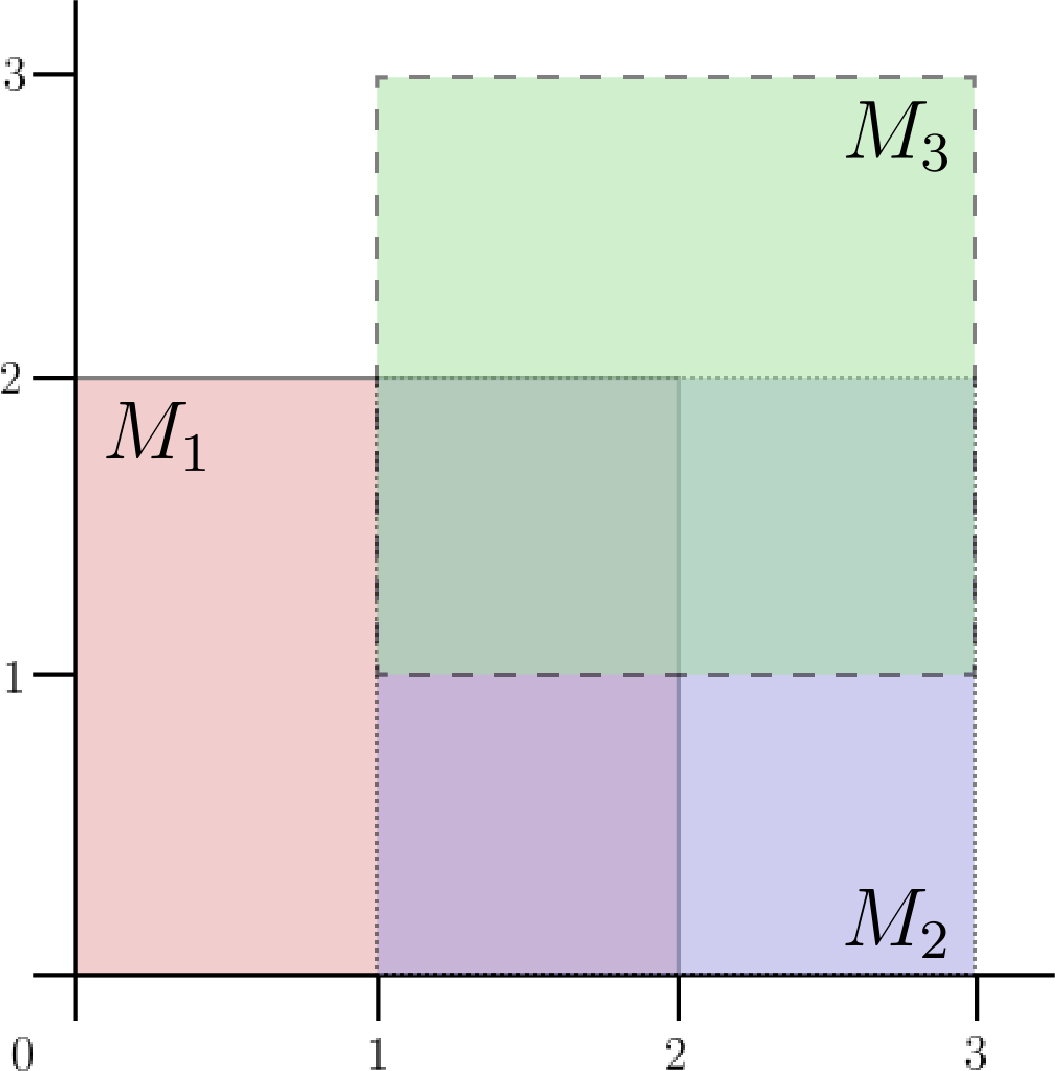}
    \caption{Rectangles used to define the modules in Example \ref{ex:rectangle_modules}. Differing dash patterns on the boundaries are used to distinguish the rectangles, not to indicate (lack of) inclusion in the rectangle.}
    \label{fig:rectangle_modules}
\end{figure}

As a simple example, consider the interval modules over $\R^2$:
\[
M_1 = M_{[(0,0),(2,2))}, \quad M_2 = M_{[(1,0),(3,2))}, \quad M_3 = M_{[(1,1),(3,3))}
\]
(see Figure \ref{fig:rectangle_modules}). Using the conditions for existence of morphisms between interval modules described in~\cite{dey2018computing}, one can show that the flow interleaving distance satisfies
\[
d_{T^\mathrm{st}}^\mathrm{Fl}(M_1,M_2) = d_{T^\mathrm{st}}^\mathrm{Fl}(M_1,M_3) = 1,
\]
while the 2-functor interleaving distances with, say, $p=2$ in the Lawvere weight, are given by
\[
d_{T^\cat{Y},\weight}(M_1,M_2) = 1 \quad \mbox{ and } \quad d_{T^\cat{Y},\weight}(M_1,M_3) = \sqrt{2}.
\]
(The former is realized by shifts $(\mathbf{s},\mathbf{t}) = ((0,0),(1,0))$ and the latter by $(\mathbf{s},\mathbf{t}) = ((0,0),(1,1))$.) That is, the $\cat{G}^\set{P}$-interleaving distance respects the geometrical intuition that $M_2$ is closer to $M_1$ than $M_3$, while the flow interleaving distance does not. 
\end{ex} 

\section{Weighted 2-Category Interleaving Distances}\label{sec:weighted_2_category_interleaving}

In this section, we introduce a new family of interleaving distances which generalize those introduced above. We also show that this construction generalizes another family of interleaving distances associated to \emph{locally persistent categories} \cite{scoccola2020locally}.

\subsection{2-Weighted 2-Categories} Let $\cat{C}$ be a 2-category. We generalize the notion of a Lawvere weight as follows.

\begin{defn}\label{def:lawvere_2_weight}
    A \emph{Lawvere 2-weight} on $\cat{C}$ is a pair $\weight = (\weight_1,\weight_2)$, where the functions $\weight_i:\cat{C}_i \to \R_{\geq 0}$ satisfy the following properties:
    \begin{enumerate}
        \item (Zero on Identities) We have $\weight_1(1_A) = 0$ for all $A \in \cat{C}_0$ and $\weight_2(1_f) = 0$ for all $f \in \cat{C}_1$.
        \item (Triangle Inequalities) For all $f \in \cat{C}_1(A,B)$ and $g \in \cat{C}_2(B,C)$, 
        \[
        \weight_1(gf) \leq \weight_1(g) + \weight_1(f).
        \]
        For all $h,j \in \cat{C}_1(A,B)$ and $k \in \cat{C}_1(B,C)$, and $\alpha \in \cat{C}_2(f,h)$, $\beta \in \cat{C}_2(h,j)$, and $\gamma \in \cat{C}_2(g,k)$, we have
        \[
        \weight_2(\beta\alpha) \leq \weight_2(\beta) + \weight_2(\alpha) \qquad \mbox{and} \qquad \weight_2(\gamma \bullet \alpha) \leq \weight_2(\gamma) + \weight_2(\alpha).
        \]
    \end{enumerate}
    A pair $(\cat{C},\weight)$ consisting of a 2-category $\cat{C}$ and Lawvere 2-weight $\weight$ will be referred to as a \emph{2-weighted 2-category}.
\end{defn}

This expanded notion of weight leads to a new family of interleaving distances on $\cat{C}_0$.

\begin{defn}[2-Weighted 2-Category Interleaving Distance]\label{def:2-weighted-2-category-interleaving}
    Let $(\cat{C},\weight)$ be a 2-weighted 2-category and let $A,B \in \cat{C}_0$. A \emph{$(g,h,\alpha,\beta)$-interleaving of $A$ and $B$} consists of 1-morphisms $g:A \to B$ and $h:B \to A$ and 2-morphisms $\alpha:1_A \Rightarrow hg$ and $\beta:1_B \Rightarrow gh$.
    These morphisms are summarized in the following diagram
\[\begin{tikzcd}
	A && B \\
	\\
	A && B
	\arrow["g", from=1-1, to=1-3]
	\arrow["h", from=3-3, to=3-1]
	\arrow[""{name=0, anchor=center, inner sep=0}, "{1_A}"', curve={height=12pt}, from=1-1, to=3-1]
	\arrow[""{name=1, anchor=center, inner sep=0}, "hg", curve={height=-12pt}, from=1-1, to=3-1]
	\arrow[""{name=2, anchor=center, inner sep=0}, "gh"', curve={height=12pt}, from=1-3, to=3-3]
	\arrow[""{name=3, anchor=center, inner sep=0}, "{1_B}", curve={height=-12pt}, from=1-3, to=3-3]
	\arrow["\alpha"', shorten <=5pt, shorten >=5pt, Rightarrow, from=0, to=1]
	\arrow["\beta", shorten <=5pt, shorten >=5pt, Rightarrow, from=3, to=2]
\end{tikzcd}\]
    If, for $t \geq 0$, we have
    \[
    \max\{\weight_1(g),\weight_1(h),\weight_2(\alpha),\weight_2(\beta)\} \leq t,
    \]
    then we refer to the $(g,h,\alpha,\beta)$-interleaving more succinctly as a \emph{$t$-interleaving}. 
    If a $t$-interleaving of $A$ and $B$ exists, we say that $A$ and $B$ are \emph{$t$-interleaved}.
    The \emph{interleaving distance} between $A$ and $B$ is
    \[
    d_{\cat{C},\weight}(A,B) := \inf \{t \geq 0 \mid \mbox{$A$ and $B$ are $t$-interleaved}\}.
    \]
\end{defn}

We now show that this notion of interleaving distance is an extended pseudometric. The proof is essentially the same as that of Lemma \ref{lem:interleavability} and Theorem \ref{thm:interleaving_distance}; we will show below that 2-functor interleaving distances can be realized as 2-weighted 2-category interleaving distances, so that Theorem \ref{thm:interleaving_distance} actually follows as a corollary.

\begin{thm}\label{thm:interleaving_weighted_2_category}
    Let $(\cat{C},\weight)$ be a 2-weighted 2-category. The interleaving distance $d_{\cat{C},\weight}$ defines an extended pseudometric on $\cat{C}_0$. 
\end{thm}

\begin{proof}
    Symmetry is obvious, so we only need to check the triangle inequality. Suppose that $A,B,C \in \cat{C}_0$ such that $A$ and $B$ are $t$-interleaved, say via $g:A \to B$, $h:B \to A$, $\alpha:1_A \Rightarrow hg$ and $\beta:1_B \Rightarrow gh$, and that $B$ and $C$ are $t'$-interleaved, say via $k:B \to C$, $\ell:C \to B$, $\gamma:1_B \Rightarrow \ell k$ and $\delta:1_C \Rightarrow k \ell$. Diagramatically, we have
    \[
    \begin{tikzcd}
	A && B \\
	\\
	A && B
	\arrow["g", from=1-1, to=1-3]
	\arrow["h", from=3-3, to=3-1]
	\arrow[""{name=0, anchor=center, inner sep=0}, "{1_A}"', curve={height=12pt}, from=1-1, to=3-1]
	\arrow[""{name=1, anchor=center, inner sep=0}, "hg", curve={height=-12pt}, from=1-1, to=3-1]
	\arrow[""{name=2, anchor=center, inner sep=0}, "gh"', curve={height=12pt}, from=1-3, to=3-3]
	\arrow[""{name=3, anchor=center, inner sep=0}, "{1_B}", curve={height=-12pt}, from=1-3, to=3-3]
	\arrow["\alpha"', shorten <=5pt, shorten >=5pt, Rightarrow, from=0, to=1]
	\arrow["\beta", shorten <=5pt, shorten >=5pt, Rightarrow, from=3, to=2]
\end{tikzcd} \quad \mbox{ and } \quad 
\begin{tikzcd}
	B && C \\
	\\
	B && C
	\arrow["k", from=1-1, to=1-3]
	\arrow["\ell", from=3-3, to=3-1]
	\arrow[""{name=0, anchor=center, inner sep=0}, "{1_B}"', curve={height=12pt}, from=1-1, to=3-1]
	\arrow[""{name=1, anchor=center, inner sep=0}, "\ell k", curve={height=-12pt}, from=1-1, to=3-1]
	\arrow[""{name=2, anchor=center, inner sep=0}, "k \ell"', curve={height=12pt}, from=1-3, to=3-3]
	\arrow[""{name=3, anchor=center, inner sep=0}, "{1_C}", curve={height=-12pt}, from=1-3, to=3-3]
	\arrow["\gamma"', shorten <=5pt, shorten >=5pt, Rightarrow, from=0, to=1]
	\arrow["\delta", shorten <=5pt, shorten >=5pt, Rightarrow, from=3, to=2]
\end{tikzcd}
    \]
    We claim that there exist $2$-morphisms $\phi$ and $\psi$ fitting into the following diagram
    \[
    \begin{tikzcd}
	A && C \\
	\\
	A && C
	\arrow["kg", from=1-1, to=1-3]
	\arrow["h\ell", from=3-3, to=3-1]
	\arrow[""{name=0, anchor=center, inner sep=0}, "{1_A}"', curve={height=12pt}, from=1-1, to=3-1]
	\arrow[""{name=1, anchor=center, inner sep=0}, "h\ell k g", curve={height=-12pt}, from=1-1, to=3-1]
	\arrow[""{name=2, anchor=center, inner sep=0}, "kgh\ell"', curve={height=12pt}, from=1-3, to=3-3]
	\arrow[""{name=3, anchor=center, inner sep=0}, "{1_C}", curve={height=-12pt}, from=1-3, to=3-3]
	\arrow["\phi"', shorten <=5pt, shorten >=5pt, Rightarrow, from=0, to=1]
	\arrow["\psi", shorten <=5pt, shorten >=5pt, Rightarrow, from=3, to=2]
\end{tikzcd}
    \]
    Indeed, consider the diagram
\[\begin{tikzcd}
	A && B && B && A
	\arrow["h"{description}, from=1-5, to=1-7]
	\arrow["g"{description}, from=1-1, to=1-3]
	\arrow[""{name=0, anchor=center, inner sep=0}, "{\ell k}"{description, pos=0.3}, curve={height=-12pt}, from=1-3, to=1-5]
	\arrow[""{name=1, anchor=center, inner sep=0}, "{1_B}"{description, pos=0.3}, curve={height=12pt}, from=1-3, to=1-5]
	\arrow["{h\ell k g}"{description}, curve={height=-30pt}, from=1-1, to=1-7]
	\arrow[""{name=2, anchor=center, inner sep=0}, "hg"{description, pos=0.7}, curve={height=25pt}, from=1-1, to=1-7]
	\arrow[""{name=3, anchor=center, inner sep=0}, "{1_A}"{description, pos=0.7}, curve={height=50pt}, from=1-1, to=1-7]
	\arrow["\gamma"', shorten <=2pt, shorten >=2pt, Rightarrow, from=1, to=0]
	\arrow["\alpha"', shorten <=2pt, shorten >=2pt, Rightarrow, from=3, to=2]
\end{tikzcd}\]
    There is a 2-morphism $hg \Rightarrow h\ell k g$ given by whiskering, so that we get a 2-morphism $1_A \Rightarrow h\ell k g$ defined by 
    $
    \phi = (h \bullet \gamma \bullet g) \circ \alpha
    $.
    Similarly, $\psi : 1_C \Rightarrow kgh\ell$ is defined by
    $
    \psi = (k \bullet \beta \bullet \ell) \circ \delta
    $.
    Now consider the energies of the various morphisms
    \begin{align*}
        \weight_1(kg) &\leq \weight_1(k) + \weight_1(g) \leq t + t' \\
        \weight_2(\phi) &\leq \weight_2(h\bullet \gamma \bullet g) + \weight_2(\alpha) \leq \weight_2(1_h) + \weight_2(\gamma) + \weight_2(1_g) + \weight_2(\alpha) \leq 0 + t' + 0 + t,
    \end{align*}
    and, likewise, $\weight_1(h\ell),\weight_2(\psi) \leq t + t'$. It follows that $d_{\cat{C},\weight}$ satisfies the triangle inequality.
\end{proof}

\subsection{Connection to Lawvere Metric Spaces and Examples}

In \cite{lawvere1973metric}, Lawvere shows that the notion of a weighted 1-category (Definition \ref{def:lawvere_weight}) is equivalent to that of a Lawvere metric space (see Remark \ref{rem:lawvere_weight}). In particular, given a weighed 1-category $(\cat{C},\weight)$, one constructs a Lawvere metric space as $(\cat{C}_0,d^\mathrm{Law}_{\cat{C},\weight})$, where 
\[
d^\mathrm{Law}_{\cat{C},\weight}(A,B) = \inf \{\mathrm{W}(g) \mid g \in \cat{C}_1(A,B)\}.
\]
One can symmetrize this distance function to obtain an extended pseudometric $\hat{d}^\mathrm{Law}_{\cat{C},\weight}$ with
\[
\hat{d}^\mathrm{Law}_{\cat{C},\weight}(A,B) = \max\{d^\mathrm{Law}_{\cat{C},\weight}(A,B),d^\mathrm{Law}_{\cat{C},\weight}(B,A)\}.
\]
Now consider the indiscrete 2-category $\cat{E}\cat{C}$ (Definition \ref{def:indiscrete}) with Lawvere 2-weight $\widehat{\weight}$ defined by setting $\widehat{\weight}_1 = \weight$ and setting $\widehat{\weight}_2$ to be identically zero. It is straightforward to prove that this recovers the symmetrized Lawvere construction:

\begin{prop}\label{prop:lawvere_metrics_interleaving}
    Using the notation above, the interleaving distance $d_{\cat{E}\cat{C}, \widehat{\mathrm{W}}}$ is equal to the symmetrized Lawvere metric $\hat{d}^\mathrm{Law}_{\cat{C},\weight}$. 
\end{prop}

Since any extended pseudometric can be realized via the Lawvere construction for some weighted 1-category, we have the following corollary.

\begin{cor}
    Any extended pseudometric can be represented as an interleaving distance associated to a 2-weighted 2-category.
\end{cor}

\begin{rem}
    In fact, any extended pseudometric space can be realized as a 2-functor interleaving distance or as an interleaving distance associated to a Locally Persistent Category (LPC)~\cite{scoccola2020locally} (see Section \ref{sec:locally_persistent_categories} below for the definition)---constructions which we show in the following subsections to be generalized by 2-weighted 2-category interleaving distance. We echo a sentiment expressed in \cite{scoccola2020locally} in the setting of LPCs: although any extended pseudometric can be realized as an interleaving distance in this somewhat trivial way, the interest in the interleaving distance construction comes from situations where it arises in a natural context or where it leads to notions of distance which are not obviously constructed in some other manner.
\end{rem}

On the other hand, given a 2-weighted 2-category $(\cat{C},\weight)$, there is an associated 1-category coming from the underlying 1-category, denoted (by an abuse of notation) as $\cat{C}$, with the underlying weight $\weight_1$. The following result is obvious.

\begin{prop}\label{prop:weighted_category_inequality}
    Let $(\cat{C},\weight)$ be a 2-weighted 2-category with associated weighted 1-category $(\cat{C},\weight_1)$. Then, for any $A,B \in \cat{C}_0$, 
    \[
    \hat{d}_{\cat{C},\weight_1}^\mathrm{Law}(A,B) \leq d_{\cat{C},\weight}(A,B).
    \]
    If $\mathrm{W}_2$ is identically zero, then this is an equality.
\end{prop}

\begin{rem}
    Propositions \ref{prop:lawvere_metrics_interleaving} and  \ref{prop:weighted_category_inequality} support our intuitive understanding of the interleaving distance associated to a 2-weighted 2-category: it is a generalization of the Lawvere categorical representation of a metric space, with the inclusion of additional obstructions on allowable morphisms between objects (in the form of weighted 2-morphisms).
\end{rem}

We show below that other interleaving constructions can be realized in the 2-weighted 2-category framework; namely 2-functor interleaving distances in Section \ref{sec:2_functor_interleaving_weighte_2_categories} and interleaving distances for Locally Persistent Categories in Section \ref{sec:locally_persistent_categories}. Before we move on to these generalities, we provide some specific examples of 2-weighted 2-category interleaving distances which illustrate the ideas above, focusing on the theme of Gromov-Hausdorff-type distances.

\subsubsection{Altered Gromov-Hausdorff Distance} Let us recall the basic construction of Gromov-Hausdorff distance. Let $\set{X} = (\set{X},d_\set{X})$ and $\set{Y} = (\set{Y},d_\set{Y})$ be metric spaces. The \emph{Gromov-Hausdorff (GH) distance}~\cite{burago2022course} between $\set{X}$ and $\set{Y}$ is given by
\begin{equation}\label{eqn:gromov-hausdorff-distance}
d_{\mathrm{GH}}(\set{X},\set{Y}) = \frac{1}{2} \inf_{f,g} \max \{\mathrm{dis}(f),\mathrm{dis}(g),\mathrm{codis}(f,g)\},
\end{equation}
where the infimum is defined over (not necessarily continuous) maps $f:\set{X} \to \set{Y}$ and $g:\set{Y} \to \set{X}$, the \emph{distortion} of $f$ is defined as
\[
\mathrm{dis}(f) := \sup_{x,x' \in \set{X}} | d_\set{X}(x,x') - d_\set{Y}(f(x),f(x')) |,
\]
with $\mathrm{dis}(g)$ defined similarly, and the \emph{codistortion} of $f$ and $g$ is 
\[
\mathrm{codis}(f,g)  := \sup_{x \in \set{X}, y \in \set{Y}} |d_\set{X}(x,g(y)) - d_\set{Y}(y,f(x))|.
\]
Intuitively, the distortion terms measure how much the maps distort the metric, whereas the codistortion measures how far $f$ and $g$ are from being inverses.

The Gromov-Hausdorff distance was characterized as an interleaving distance in a locally persistent category (see Section \ref{sec:locally_persistent_categories}) in \cite[Section 6.2]{scoccola2020locally}. On the other hand, one might observe that the structure of GH distance is quite similar to that of an interleaving distance on a 2-weighted 2-category. Here, we make this observation precise by showing that a slight variant of GH distance is naturally realized in this context. We begin by altering the definition of GH distance to more directly reflect the intuition for its definition described above.

\begin{defn}
    The \emph{altered Gromov-Hausdorff distance} between metric spaces $\set{X}$ and $\set{Y}$ is 
    \[
    \tilde{d}_\mathrm{GH}(\set{X},\set{Y}) := \frac{1}{2} \inf_{f,g} \max\left\{\mathrm{dis}(f), \mathrm{dis}(g), \widetilde{\mathrm{codis}}(f,g) \right\},
    \]
    where the infimum is over (not necessarily continuous) maps $f:\set{X} \to \set{Y}$ and $g:\set{Y} \to \set{X}$, $\mathrm{dis}(f)$ and $\mathrm{dis}(g)$ are defined as above, and 
    \[
    \widetilde{\mathrm{codis}}(f,g) = \widetilde{\mathrm{codis}}_{\set{X},\set{Y}}(f,g) := \max \left\{\sup_{x \in \set{X}} d_\set{X}(x,g f(x)), \sup_{y \in \set{Y}} d_\set{Y}(y,fg(y))\right\}.
    \]
\end{defn}

The altered version of GH distance is bi-Lipschitz equivalent to the classical version, as we now demonstrate.

\begin{prop}\label{prop:bilipschitz_GH}
    For any metric spaces $\set{X}$ and $\set{Y}$,
    \[
    \tilde{d}_\mathrm{GH}(\set{X},\set{Y}) \leq d_\mathrm{GH}(\set{X},\set{Y}) \leq 2\tilde{d}_\mathrm{GH}(\set{X},\set{Y}).
    \]
\end{prop}

\begin{proof}
    Let $f:\set{X} \to \set{Y}$ and $g:\set{Y} \to \set{X}$ such that $\mathrm{dis}(f)$, $\mathrm{dis}(g)$ and $\mathrm{codis}(f,g)$ are all less than or equal to $r$ for some $r \geq 0$. Then, for any $x \in \set{X}$,
    \[
    d_\set{X}(x,g f(x)) = d_\set{X}(x,gf(x)) - d_\set{Y}(f(x),f(x)) \leq r,
    \]
    where we have used the codistortion bound, taking $y = f(x)$. Likewise, $d_\set{Y}(y,f g(y)) \leq r$ for all $y \in \set{Y}$, and these statements impy that $\widetilde{\mathrm{codis}}(f,g) \leq r$. This proves that $\tilde{d}_\mathrm{GH}(\set{X},\set{Y}) \leq d_\mathrm{GH}(\set{X},\set{Y})$.

    To show the other inequality, suppose we have $f:\set{X} \to \set{Y}$ and $g:\set{Y} \to \set{X}$ with $\mathrm{dis}(f)$, $\mathrm{dis}(g)$ and $\widetilde{\mathrm{codis}}(f,g)$ all bounded above by $r$. Then, for all $x \in \set{X}$ and $y \in \set{Y}$,
    \begin{align*}
        d_\set{X}(x,g(y)) &\leq d_\set{X}(x,g f(x)) + d_\set{X}(g f(x),g(y)) \\
        &\leq r + d_\set{Y}(f(x),y) + r,
    \end{align*}
    where we have used the bounds on distortion and the altered codistortion. It follows that $d_\set{X}(x,g(y)) - d_\set{Y}(f(x),y) \leq 2r$ and a similar argument shows that $d_\set{Y}(f(x),y) - d_\set{X}(x,g(y)) \leq 2r$, so that $\mathrm{codis}(f,g) \leq 2r$. Therefore $d_\mathrm{GH}(\set{X},\set{Y}) \leq 2\tilde{d}_\mathrm{GH}(\set{X},\set{Y})$.
\end{proof}

We now realize the altered Gromov-Hausdorff distance as a 2-weighted 2-category interleaving distance. The 2-category $\cat{GH}$ is defined as follows:
\begin{enumerate}
    \item $\cat{GH}_0$ consists of metric spaces;
    \item $\cat{GH}_1(\set{X},\set{Y})$ contains set maps $f:\set{X} \to \set{Y}$ (note that this is different than the usual categorical structure on the space of metric spaces, where one typically considers Lipschitz maps);
    \item for functions $f,g:\set{X} \to \set{Y}$,
    \begin{enumerate}
        \item if $f = g$, $\cat{GH}_2(f,f)$ contains a single abstract identity 2-morphism $1_f$;
        \item if $\set{X} = \set{Y}$ and $f = 1_\set{X}$, $\cat{GH}_2(1_\set{X},g)$ contains a single abstract 2-morphism, denoted $\hat{g}$---if this coincides with the previous case, we have $1_g = \hat{g}$;
        \item in any other case, $\cat{GH}_2(f,g)$ is empty.
    \end{enumerate}
        Compositions of 2-morphisms are defined by first declaring the identity 2-morphisms to behave as identities. It then only remains to define horizontal composition in case of the diagram
    \[\begin{tikzcd}
    	\set{X} && \set{X} && \set{X}
    	\arrow[""{name=0, anchor=center, inner sep=0}, "{1_\set{X}}", curve={height=-12pt}, from=1-1, to=1-3]
    	\arrow[""{name=1, anchor=center, inner sep=0}, "{1_\set{X}}", curve={height=-12pt}, from=1-3, to=1-5]
    	\arrow[""{name=2, anchor=center, inner sep=0}, "f"', curve={height=12pt}, from=1-1, to=1-3]
    	\arrow[""{name=3, anchor=center, inner sep=0}, "g"', curve={height=12pt}, from=1-3, to=1-5]
    	\arrow["{\hat{f}}", shorten <=3pt, shorten >=3pt, Rightarrow, from=0, to=2]
    	\arrow["{\hat{g}}", shorten <=3pt, shorten >=3pt, Rightarrow, from=1, to=3]
    \end{tikzcd}\]
        in which case we get a well-defined 2-morphism $\hat{g} \bullet \hat{f} = \widehat{g \circ f}$. 
\end{enumerate}
Define a Lawvere 2-weight $\weight$ on $\cat{GH}$ by setting $\weight_1(f) := \mathrm{dis}_{\set{X},\set{Y}}(f)$ and $\weight_2(1_f) = 0$ for $f:\set{X} \to \set{Y}$ and, for $\hat{g}:1_\set{X} \Rightarrow g$, 
\[
\weight_2(\hat{g}) := \sup_{x \in \set{X}} d_X(x,g(x)).
\]
To see that this gives a well-defined Lawvere 2-energy, we need to verify the triangle inequalities. For $\weight_1$, this follows from standard properties of the distortion functional and for $\weight_2$, this follows by checking cases.

\begin{prop}\label{prop:altered_GH_interleaving}
    The altered Gromov-Hausdorff distance is equal to the 2-weighted 2-category interleaving distance associated to $(\cat{GH},\weight)$. that is, for metric spaces $\set{X}$ and $\set{Y}$, we have
    \[
    \tilde{d}_\mathrm{GH}(\set{X},\set{Y}) = d_{\cat{GH},\weight}(\set{X},\set{Y})
    \]
\end{prop}

\begin{proof}
The interleaving distance between $\set{X}$ and $\set{Y}$ is computed by optimizing over diagrams
    \[\begin{tikzcd}
	\set{X} && \set{Y} \\
	\\
	\set{X} && \set{Y}
	\arrow["g", from=1-1, to=1-3]
	\arrow["h", from=3-3, to=3-1]
	\arrow[""{name=0, anchor=center, inner sep=0}, "{1_\set{X}}"', curve={height=12pt}, from=1-1, to=3-1]
	\arrow[""{name=1, anchor=center, inner sep=0}, "hg", curve={height=-12pt}, from=1-1, to=3-1]
	\arrow[""{name=2, anchor=center, inner sep=0}, "gh"', curve={height=12pt}, from=1-3, to=3-3]
	\arrow[""{name=3, anchor=center, inner sep=0}, "{1_\set{Y}}", curve={height=-12pt}, from=1-3, to=3-3]
	\arrow["\alpha_{hg}"', shorten <=5pt, shorten >=5pt, Rightarrow, from=0, to=1]
	\arrow["\alpha_{gh}", shorten <=5pt, shorten >=5pt, Rightarrow, from=3, to=2]
\end{tikzcd}\]
the quantity $\max\{\weight_1(g), \weight_1(h), \weight_2(\alpha_{hg}), \weight_2(\alpha_{gh})\}$. 
Observing that $\max\{\weight_2(\alpha_{hg}), \weight_2(\alpha_{gh})\} = \widetilde{\mathrm{codis}}(g,h)$, it follows that this is the same as the optimization problem which defines the altered GH distance.
\end{proof}

Another alteration of GH distance was studied in \cite{memoli2012some}. The \emph{modified Gromov-Hausdorff distance}, $\hat{d}_\mathrm{GH}$, is defined by dropping the codistortion term altogether:
\[
\hat{d}_\mathrm{GH}(\set{X},\set{Y}) = \inf_{f,g} \max \{\mathrm{dis}(f),\mathrm{dis}(g)\}.
\]
It is shown in \cite[Theorem 4.1]{memoli2012some} that $\hat{d}_\mathrm{GH}$ defines a metric on the space of isometry classes of compact metric spaces, which clearly satisfies $\hat{d}_\mathrm{GH} \leq d_\mathrm{GH}$. The  fact that $\hat{d}_\mathrm{GH}$ is bounded above by $d_\mathrm{GH}$ fits into the discussion above, as is summarized in the following result. The proof follows directly from the definitions and the results referenced in the statement.

\begin{prop}
    Let $\cat{M}$ and $\cat{GH}$ as above. The modified Gromov-Hausdorff distance is equal to the distance obtained through the symmetrized Lawvere construction; that is, for any metric spaces $\set{X}$ and $\set{Y}$,
    \[
    \hat{d}_\mathrm{GH}(\set{X},\set{Y}) = \hat{d}_{\cat{M},\mathrm{W}_1}^\mathrm{Law}(\set{X},\set{Y}).
    \]
    The fact that $\hat{d}_\mathrm{GH} \leq d_\mathrm{GH}$ is therefore an instance of Proposition \ref{prop:lawvere_metrics_interleaving} (together with Propositions \ref{prop:bilipschitz_GH} and \ref{prop:altered_GH_interleaving}).
\end{prop}

\subsubsection{The Gromov-Hausdorff Bicategory} The Gromov-Hausdorff distance is categorified and generalized in \cite{bubenik2017interleaving}, using the formalism of bicategories (see Remark \ref{rem:2-categories_and_bicategories}). Describing the construction in detail would require a significant detour, so we only very  briefly summarize it here, and then remark on the obvious connection to the notion of 2-weighted 2-category interleaving distance.

It is well known (see, e.g., \cite{burago2022course}) that the GH distance can alternatively be realized as 
\begin{equation}\label{eqn:embedding_GH}
d_\mathrm{GH}(\set{X},\set{Y}) = \inf_{\set{Z},f,g} d^\set{Z}_\mathrm{H}(f(X),g(Y)),
\end{equation}
where the infimum is over all metric spaces $\set{Z}$ and isometric embeddings $f:\set{X} \to \set{Z}$ and $g:\set{Y} \to \set{Z}$, and where $d^\set{Z}_\mathrm{H}$ denotes the Hausdorff distance between subsets of $\set{Z}$. 

Inspired by the formulation \eqref{eqn:embedding_GH} of GH distance, the \emph{Gromov-Hausdorff bicategory}~\cite[Section 5]{bubenik2017interleaving} $\cat{WEmb}$ is defined to be a bicategory whose objects are small categories endowed with Lawvere weights. A 1-morphism between $(\cat{C},\mathrm{W}),(\cat{C}',\mathrm{W}') \in \cat{WEmb}_0$ is an \emph{embedding pair}, which is a triple $\big((\cat{C}'',\mathrm{W}''), F, F'\big)$, where $(\cat{C}'',\mathrm{W}'')$ is a weighted category, $F:\cat{C}\to \cat{C}''$ and $F':\cat{C}' \to \cat{C}''$ are full and faithful functors which are injective on objects, and which preserve weights. A generalized notion of Hausdorff distance can be associated to an embedding pair, and this is used as a Lawvere weight for $\cat{WEmb}$. There is a natural notion of a morphism between embedding pairs, and these form the 2-morphisms for $\cat{WEmb}$, but we omit a detailed description here.

The Lawvere metric resulting from the construction of $\cat{WEmb}$ described above gives a generalized, categorified notion of Gromov-Hausdorff distance between weighted 1-categories. We remark that, in light of Proposition \ref{prop:weighted_category_inequality}, this generalized GH distance can essentially be viewed as an interleaving distance associated to a 2-weighted 2-category (by extending the Lawvere weight to be zero on 2-morphisms). The qualification in this statement could be removed by extending the definitions and results above to handle 2-weighted bicategories.

\subsection{2-Functor Interleaving Distances and 2-Weighted 2-Categories}\label{sec:2_functor_interleaving_weighte_2_categories} Let us now show how to realize 2-functor interleaving distances as interleaving distances for 2-weighted 2-categories. 

\begin{defn}[2-Weighted 2-Category from a 2-Functor]\label{def:2_functor_to_weighted_2_category}
Let $\Delta:\cat{C} \to \cat{Cat}$ be a 2-functor and suppose that $\cat{C}$ is endowed with a Lawvere weight $\mathrm{W}$ (in the sense of Definition \ref{def:lawvere_weight}). The \emph{2-weighted 2-category associated to $\Delta$}, denoted $(\cat{C}^\Delta,\mathrm{W}^\Delta)$, is defined as follows. The 2-category $\cat{C}^{\Delta}$ is defined by: 
\begin{enumerate}
    \item $\cat{C}^\Delta_0 = \mathrm{Im}(\Delta)$, the object-image of $\Delta$ (Definition~\ref{defn:object_image}).
    \item For $A,B \in \cat{C}_0$ and $X \in \Delta(A)_0$, $Y \in \Delta(B)_0$, 
    \[
    \cat{C}^\Delta_1(X,Y) = \{(g,\phi) \mid g \in \cat{C}_1(B,A), \, \phi \in \Delta(A)_1(X,\Delta_g(Y))\};
    \]
    composition of $(g,\phi):X \to Y$ and $(h,\psi):Y \to Z$ (with $Z \in \Delta(C)_0)$) is given by 
    \[
    (h,\psi) \circ (g,\phi) := (gh,\Delta_g(\psi) \circ \phi),
    \]
    as is illustrated in the diagram
    \[
    X \xrightarrow{\phi} \Delta_g(Y) \xrightarrow{\Delta_g(\psi)} \Delta_{gh}(Z).
    \]
    \item\label{defn:CDelta_2morph} For 1-morphisms $(g,\phi),(h,\psi):X \to Y$ with $X \in \Delta(A)_0$ and $Y \in \Delta(B)_0$, the collection of 2-morphisms $\cat{C}^\Delta_2((g,\phi),(h,\psi))$ is defined by cases as follows. 
    \begin{enumerate}
        \item If $X = Y \in \Delta(A)_0$, $g = 1_A$, $\phi = 1_X$, and there exists $\alpha:1_A \Rightarrow h$ such that $\Delta(\alpha)_X = \psi$, we define $\cat{C}^\Delta_2((g,\phi),(h,\psi))$ to contain a single abstract element, denoted $\widehat{\psi}$.
        \item If $(g,\phi) = (h,\psi)$, we define $\cat{C}^\Delta_2((g,\phi),(h,\psi))$ to contain a single abstract element, denoted $1_{g,\phi}$; it is possible for this case to coincide with the former case, and then we have $1_{g,\phi} = \widehat{1}_X$. 
        \item In any other case, we define $\cat{C}^\Delta_2((g,\phi),(h,\psi))$ to be empty. 
    \end{enumerate} 
    Composition is defined by first declaring each $1_{g,\phi}$ to act as an identity. It then remains to define vertical and horizontal compositions when both 2-morphisms are of type (a), as defined above. For vertical compositions, the definitions force one of the involved 2-morphisms to be an identity (i.e., of the form $1_{g,\phi}$), so we actually only need to consider the case of horizontal compositions. Suppose we have the following diagram
\[\begin{tikzcd}
	X && X && X
	\arrow[""{name=0, anchor=center, inner sep=0}, "{(1_A,1_X)}", curve={height=-12pt}, from=1-1, to=1-3]
	\arrow[""{name=1, anchor=center, inner sep=0}, "{(1_A,1_X)}", curve={height=-12pt}, from=1-3, to=1-5]
	\arrow[""{name=2, anchor=center, inner sep=0}, "{(g,\phi)}"', curve={height=12pt}, from=1-1, to=1-3]
	\arrow[""{name=3, anchor=center, inner sep=0}, "{(k,\psi)}"', curve={height=12pt}, from=1-3, to=1-5]
	\arrow["{\widehat{\phi}}", shorten <=3pt, shorten >=3pt, Rightarrow, from=0, to=2]
	\arrow["{\widehat{\psi}}", shorten <=3pt, shorten >=3pt, Rightarrow, from=1, to=3]
\end{tikzcd}\]
    Then, by definition, there are 2-morphisms $\alpha:1_A \Rightarrow g$ and $\beta:1_A \Rightarrow k$ such that $\phi = \Delta(\alpha)_X$ and $\psi = \Delta(\beta)_X$. Our rules of compositions for 1-morphisms give the diagram
\[\begin{tikzcd}
	X && X && X \\
	& {}
	\arrow["{(1_A,1_X)}"{description}, curve={height=-6pt}, from=1-1, to=1-3]
	\arrow["{(1_A,1_X)}"{description}, curve={height=-6pt}, from=1-3, to=1-5]
	\arrow["{(g,\Delta(\alpha)_X)}"{description}, curve={height=6pt}, from=1-1, to=1-3]
	\arrow["{(k,\Delta(\beta)_X)}"{description}, curve={height=6pt}, from=1-3, to=1-5]
	\arrow["{\big(gk,\Delta_g(\Delta(\beta)_X)\Delta(\alpha)_X\big)}"', curve={height=30pt}, from=1-1, to=1-5]
	\arrow["{(1_A,1_X)}", curve={height=-30pt}, from=1-1, to=1-5]
\end{tikzcd}\]
Let \begin{equation}\label{eqn:2morph_comp} \omega := \Delta_g(\Delta(\beta)_X)\Delta(\alpha)_X.\end{equation} By \eqref{eqn:horizontal_composition_2_functor}, we have that $\omega = \Delta(\alpha \bullet \beta)_X$, where $\alpha \bullet \beta:1_A \Rightarrow gk$, so we get a well-defined composition by setting $\widehat{\psi} \bullet \widehat{\phi} := \widehat{\omega}$. 
\end{enumerate}
We define a Lawvere 2-weight $\weight^\Delta$ on $\cat{C}^\Delta$ by 
\[
\weight_1^\Delta((g,\phi)) := \weight(g) \quad \mbox{and} \quad \weight_2^\Delta(\Delta(\alpha)) := 0.
\]
\end{defn}

We will show below that for any 2-functor $\Delta:\cat{C} \to \cat{Cat}$ and Lawvere weight $\weight$ on $\cat{C}$, the 2-functor interleaving distance on $\mathrm{Im}(\Delta) = \cat{C}_0^\Delta$ agrees with the 2-weighted 2-category interleaving distance, $d_{\Delta,\weight} = d_{\cat{C}^\Delta,\weight^\Delta}$. In fact, we will prove a stronger structural result showing that the assignment $(\cat{C},\mathrm{W}) \mapsto (\cat{C}^\Delta,\mathrm{W}^\Delta)$ is natural. This will require some additional concepts from the theory of 2-categories.

\subsubsection{2-Natural Transformations} We begin by recalling the notion of a 2-natural transformation \cite[Prop. 4.2.11]{johnson20212}, appropriately adjusted for our purposes.

\begin{defn}\label{defn:2-nat_transf} We let $\cat{2Fun}(\cdot, \cat{Cat})$ denote the category of 2-functors into $\cat{C}$, $\Delta: \cat{C} \to \cat{Cat}$, where $\cat{C}$ is an arbitrary 2-category. Given 2-functors $\Delta: \cat{C} \to \cat{Cat}$ and $\nabla: \cat{D} \to \cat{Cat}$, a \emph{morphism} $(F, \eta): \Delta \to \nabla$ is given by a 2-functor $F: \cat{C} \to \cat{D}$ and a 2-natural transformation $\eta: \Delta \Rightarrow \nabla \circ F$ (this definition is recalled below). Diagrammatically, this is 
\[\begin{tikzcd}
	\cat{C} && \cat{D} \\
	\\
	& \cat{Cat}
	\arrow["F", from=1-1, to=1-3]
	\arrow[""{name=0, anchor=center, inner sep=0}, "\Delta"', from=1-1, to=3-2]
	\arrow["\nabla", from=1-3, to=3-2]
	\arrow["\eta", shorten <=8pt, shorten >=13pt, Rightarrow, from=0, to=1-3]
\end{tikzcd}\]
If $\Delta_1: \cat C \to \cat{Cat}$, $\Delta_2: \cat D \to \cat{Cat}$ and $\Delta_3: \cat E \to \cat{Cat}$ are 2-functors; $F: \cat C \to \cat D$ and $G:\cat D \to \cat E$ are 2-functors; and $\eta: \Delta_1 \Rightarrow \Delta_2 \circ F$ and $\nu: \Delta_2 \Rightarrow \Delta_3 \circ G$ are 2-natural transformations, define the composition to be 
\[
(G, \nu) \circ (F, \eta) =(G \circ F, (\nu \bullet F) \circ \eta).
\]
We now recall that a \emph{2-natural transformation} $\eta:\Delta \Rightarrow \nabla \circ F$ consists of the following data: 
\begin{enumerate}
    \item For each $A \in \cat{C}_0$, a functor $\eta_A: \Delta(A) \to \nabla(F(A))$.
    \item For each $f: A \to B$ in $\cat{C}_1$, a commutative diagram 
    \begin{equation}\label{diagm:1-comm}
\begin{tikzcd}[row sep=large,column sep=large,every label/.append
style={font=\normalsize}]
\Delta(A) \ar[d, swap, "\Delta_f"] \ar[r, "\eta_A"] & \nabla(F(A)) \ar[d, "(\nabla\circ F)_f"]\\
\Delta(B) \ar[r, "\eta_B"] & \nabla(F(B))
\end{tikzcd}
\end{equation}
so that $\eta_B \circ \Delta_f = (\nabla \circ F)_f \circ \eta_A = \nabla_{F(f)}\circ \eta_A$ as 1-morphisms in $\cat{Cat}$ (i.e., functors).

\item For each pair $f,g: A \to B$ in $\cat{C}_1$ and each
 $\alpha:f \to g$  in $\cat{C}_2$, a commutative diagram 
\begin{equation}\label{diag:2-comm}
\begin{tikzcd}
(\nabla \circ F)_f \circ \eta_A \ar[r, "="] \ar[d, swap, "(\nabla \circ F)(\alpha) \bullet 1_{\eta_A}"] & \eta_B \circ \Delta_f \ar[d, "1_{\eta_B} \bullet \Delta(\alpha)"]\\
(\nabla \circ F)_g \circ \eta_A \ar[r, "="] & \eta_B \circ \Delta_g 
\end{tikzcd}
\end{equation}
We remind the reader that $\bullet$ denotes horizontal composition (see  Definition~\ref{defn:2cat} and Remark~\ref{rem:whiskering}).
\end{enumerate}
\end{defn}

We extract a useful identity from Diagram~\ref{diag:2-comm}. The equalities along the top and bottom are those established in Diagram~\ref{diagm:1-comm}. The natural transformation on the left-hand side of the diagram may be realized as the whiskering of $\eta_{A}$ with $(\nabla\circ F)(\alpha)$:
\[\begin{tikzcd}
	{\Delta(A)} & {\nabla(F(A))} &&& {\nabla(F(B)).}
	\arrow["{\eta_A}", from=1-1, to=1-2]
	\arrow[""{name=0, anchor=center, inner sep=0}, "{(\nabla\circ F)_f}", curve={height=-30pt}, from=1-2, to=1-5]
	\arrow[""{name=1, anchor=center, inner sep=0}, "{(\nabla\circ F)_f}"', curve={height=30pt}, from=1-2, to=1-5]
	\arrow["{(\nabla\circ F)(\alpha)}"{description}, shorten <=8pt, shorten >=8pt, Rightarrow, from=0, to=1]
\end{tikzcd}\]
For $X \in \Delta(A)_0$, the $X$-component of this transformation is given by the $\eta_A(X)$-component of $(\nabla\circ F)(\alpha)$, i.e., $\left((\nabla\circ F)(\alpha) \bullet \eta_A\right)_X = (\nabla\circ F)(\alpha)_{\eta_A(X)}.$ Similarly, the natural transformation on the right-hand side of Diagram~\ref{diag:2-comm} may be realized as the whiskering of $\Delta(\alpha)$ and $\eta_B$: 
 \[
    \begin{tikzcd}
      \Delta(A) \arrow[rr, bend left=50, "\Delta_f", ""{name=U,inner sep=1pt,below}]
      \arrow[rr, bend right=50, "\Delta_g"{below}, ""{name=D,inner sep=1pt}]
      & & \Delta(B)
      \arrow[shorten <=5pt,shorten >=5pt, Rightarrow, from=U, to=D, "\Delta(\alpha)"]
      \ar[r, "\eta_B"]
      & \nabla(F(B)) .
    \end{tikzcd}
    \]
For $X \in \Delta(A)_0$, the $X$-component of this transformation is obtained by applying the functor $\eta_B$ to the $X$-component of $\Delta(\alpha)$, i.e., $\left(\eta_B \bullet \Delta(\alpha)\right)_X = \eta_B(\Delta(\alpha)_X).$ The commutativity of Diagram~\ref{diag:2-comm} implies that $\left((\nabla\circ F)(\alpha) \bullet \eta_A\right)_X  = \left(\eta_B \bullet \Delta(\alpha)\right)_X$. Combining these facts yields 

\begin{equation}\label{eqn:2morcomp}
    \eta_B(\Delta(\alpha)_X) = (\nabla\circ F)(\alpha)_{\eta_A(X)}.
\end{equation}

\subsubsection{Naturality of $(\cat{C}^\Delta,\mathrm{W}^\Delta)$} We now show that the map $\Delta \mapsto (\cat{C}^\Delta,\mathrm{W}^\Delta)$ defined in Definition \ref{def:2_functor_to_weighted_2_category} is functorial and induces an isometry of interleaving distances. We first introduce an appropriate notion of a morphism between (1 or 2)-weighted (1 or 2)-categories.

\begin{defn}[Lipschitz Functors]
    A \emph{Lipschitz 2-functor} from one 2-weighted 2-category $(\cat{C},\mathrm{W})$ to another  $(\cat{C}',\mathrm{W}')$ is a 2-functor $\Theta:\cat{C} \to \cat{C}'$  such that
    \begin{equation}\label{eqn:lipschitz_property}
    \mathrm{W}_1'(\Theta(g)) \leq \mathrm{W}_1(g) \quad \mbox{and} \quad \mathrm{W}_2'(\Theta(\alpha)) \leq \mathrm{W}_2(\alpha)
    \end{equation}
    for each 1-morphism $g$ and 2-morphism $\alpha$ of $\cat{C}$. A \emph{Lipschitz functor} of 1-categories endowed with Lawvere weights is defined similarly.
\end{defn}

Let $\cat{W2Fun}(\cdot, \cat{Cat})$ denote the category whose objects are 2-functors $\Delta:\cat{C} \to \cat{Cat}$ from a weighted 2-category $(\cat{C},\mathrm{W})$ to $\cat{Cat}$, and whose morphisms are pairs $(F,\eta)$ as in Definition \ref{defn:2-nat_transf}, such that $F$ is Lipschitz. Let $\cat{W2Cat}$ denote the category whose objects are 2-weighted 2-categories and whose morphisms are Lipschitz 2-functors. 

\begin{lem}\label{lem:faithful_functor_2_functor_to_weighted_2_category}
    Let $(\cat{C},\mathrm{W})$ be a weighted category. The map taking a 2-functor $\Delta:\cat{C}\to \cat{Cat}$ to the 2-weighted 2-category $(\cat{C}^\Delta,\weight^\Delta)$ defines a functor $\cat{W2Fun}(\cdot,\cat{Cat}) \to \cat{W2Cat}$.
\end{lem}

\begin{proof} Let $\Delta: \cat C \to \cat{Cat}$ and $\nabla: \cat D \to \cat{Cat}$ be 2-functors, $F: \cat C \to \cat D$ a 2-functor, and $\eta: \Delta \Rightarrow  \nabla \circ F$ a 2-natural transformation, satisfying the conditions of Definition \ref{defn:2-nat_transf}.
From the data $(F, \eta)$, we produce a 2-functor $F_*: \cat C ^{\Delta} \to  \cat D ^{\nabla}$.

Beginning with objects, we have $\cat{C}_0 ^{\Delta} = \text{Im} (\Delta)$ and $\cat D ^{\nabla}_0= \text{Im}(\nabla)$. Let $X \in \Delta(A)_0$ for $A \in \cat C _0$. We have a functor $\eta_A: \Delta(A) \to \nabla(F(A))$, so define $F_*(X) = \eta_A(X)$.

To extend $F_\ast$ to 1-morphisms, let $X, Y \in \text{Im}(\Delta)$, so that $X \in \Delta(A)_0$ and $Y \in \Delta(B)_0$. Let $(g, \phi): X \to Y$ be a morphism in $\cat C^{\Delta}$, with $g: B \to A$ and $\phi: X \to \Delta_g(Y)$. Then $F(g): F(B) \to F(A)$ and
\[
\eta_A(\phi): \eta_A(X) \to \eta_A(\Delta_g(Y)) = (\nabla \circ F)_g ( \eta_B(Y)) = \nabla_{F(g)}(\eta_B(Y)).
\]
Using the definition of $\cat D ^{\nabla}$, the corresponding map is $(F(g), \eta_A(\phi)): \eta_A(X) \to \eta_B(Y)$, and we define 
\[
F_*(g, \phi) = (F(g), \eta_A(\phi)).
\]
Using the commutativity properties given in the diagrams above, one checks that $F_*$ preserves compositions of 1-morphisms. If $(g, \phi): X \to Y$ and $(h, \psi): Y \to Z$ are morphisms in $\cat C ^{\Delta}$, their composition is given by $(gh, \Delta_g(\psi) \circ \phi)$, and thus 
\[
\begin{array}{rl}
F_*(h, \psi) \circ F_*(g, \phi) & = (F(h), \eta_B(\psi)) \circ (F(g), \eta_A(\phi)) \\ 
& = \left(F(g)F(h), \nabla_{F(g)}(\eta_B(\psi)) \circ \eta_A(\phi)\right) \\
& = \left(F(gh), \nabla_{F(g)}(\eta_B(\psi)) \circ \eta_A(\phi)\right) \\
& = \left(F(gh), (\nabla \circ F)_g ( \eta_B(\psi) )\circ \eta_A(\phi) \right)\\ 
& = \left(F(gh), \eta_A(\Delta_g(\psi) )\circ \eta_A(\phi)\right) \hspace{1cm} \mbox{by }(\ref{diagm:1-comm}) \\ 
& = \left(F(gh), \eta_A(\Delta_g(\psi)\circ \phi)\right) \\ 
& = F_*(gh, \Delta_g(\psi) \circ \phi)
\end{array}
\]
We treat the definition of $F_*$ on 2-morphisms by cases, following , part ~\eqref{defn:CDelta_2morph} of Definition~\ref{def:2_functor_to_weighted_2_category}. In the first case, we have a single abstract 2-morphism $\widehat{\psi} \in \cat C_2^{\Delta}\left((1_A, 1_X), (h, \psi)\right)$. Mapping to $\cat D^{\nabla}$, we have $F_*(1_A, 1_X) = (F(1_A), \eta_A(1_X)) = (1_{F(A)}, 1_{\eta_A(X)})$ and $F_*(h, \psi) = (F(h), \eta_A(\psi)).$ By definition, there is a single abstract 2-morphism $\widehat{\eta_A(\psi)} \in \cat D^{\nabla}_2((1_{F(A)}, 1_{\eta_A(X)}),(F(h), \eta_A(\psi))).$ We thus define 
\[
F_*(\widehat{\psi}) = \widehat{\eta_A(\psi)}.
\]
In the second case, we have a single abstract element $1_{g, \phi} \in \cat C_2^{\Delta}((g, \phi), (g, \phi))$, and a single abstract element $1_{F(g), \eta_A(\phi)} \in \cat D ^{\nabla}_2((F(g), \eta_A(\phi), (F(g), \eta_A(\phi))$. Define 
\[
F_*(1_{g, \phi}) = 1_{F(g), \eta_A(\phi)}.
\]
Composition of 2-morphisms is preserved by $F_*$. As stated in Definition~\ref{def:2_functor_to_weighted_2_category}, it suffices to treat horizontal compositions. Let $\widehat{\phi}: (1_A, 1_X) \Rightarrow (g, \phi)$ and $\widehat{\psi}: (1_A, 1_X) \Rightarrow (k, \psi)$ be 2-morphisms. We wish to show that $F_*(\widehat{\psi} \bullet \widehat{\phi}) = F_*(\widehat{\psi}) \bullet F_*(\widehat{\phi})$. The horizontal composition law for 2-morphisms given in Definition~\ref{def:2_functor_to_weighted_2_category}, determined by Equation~\ref{eqn:2morph_comp}, yields
\begin{equation}\label{eqn:2morph_comp_formula}
\widehat{\Delta(\beta)_X} \bullet \widehat{\Delta(\alpha)_X} = (\Delta_g(\Delta(\beta))_X) \circ \Delta(\alpha)_X)^{\wedge},
\end{equation}
\begin{equation}\label{eqn:2morph_comp_nabF}
((\nabla \circ F)(\beta)_X)^{\wedge} \bullet ((\nabla \circ F)(\alpha)_X)^{\wedge} = ((\nabla\circ F)_g((\nabla\circ F)(\beta)_{\eta_A(X)}) \circ (\nabla \circ F)(\alpha)_{\eta_A(X)})^{\wedge}.
\end{equation}
Using these facts along with Equation~\ref{eqn:2morcomp}, we have
\[
\begin{array}{rl}
    F_*(\widehat{\psi} \bullet \widehat{\phi}) & = F_*(\widehat{\omega})\\
    & = \widehat{\eta_A(\omega)}\\
    & = \eta_A(\Delta_g(\Delta(\beta)_X) \Delta(\alpha)_X)^{\wedge} \hspace{4.1cm} \mbox{by }(\ref{eqn:2morph_comp})\\
    & = (\eta_A(\Delta_g(\Delta(\beta)_X)) \circ \eta_A(\Delta(\alpha)_X))^{\wedge}\\
    & = ((\nabla\circ F)_g(\eta_A(\Delta(\beta)_X)) \circ (\nabla \circ F)_g (\alpha)_{\eta_A(X)})^{\wedge} \hspace{1cm}  \mbox{by }(\ref{eqn:2morcomp})\\
    & = ((\nabla\circ F)_g((\nabla\circ F)(\beta)_{\eta_A(X)}) \circ (\nabla \circ F)(\alpha)_{\eta_A(X)})^{\wedge} \hspace{0.38cm} 
 \mbox{by }(\ref{eqn:2morcomp})\\
    & = ((\nabla \circ F)(\beta)_{\eta_A(X)})^{\wedge} \bullet ((\nabla \circ F)(\alpha) _{\eta_A(X)})^{\wedge} \hspace{1.52cm} \mbox{by }(\ref{eqn:2morph_comp_nabF})\\
    & = (\eta_A(\Delta(\beta)_X))^{\wedge} \bullet (\eta_A(\Delta(\alpha)))^{\wedge} \hspace{3.31cm} \mbox{by }(\ref{eqn:2morcomp})\\
    & = \widehat{\eta_A(\psi)} \bullet \widehat{\eta_A(\phi)} \\
    & = F_*(\widehat{\psi}) \bullet F_*(\widehat{\phi})
\end{array}
\]
The association $(F, \eta) \mapsto F_*$ is functorial. Let $\Delta_1: \cat C \to \cat{Cat}$, $\Delta_2: \cat D \to \cat{Cat}$ and $\Delta_3: \cat E \to \cat{Cat}$ be 2-functors, and let $(F, \eta): \Delta_1 \to \Delta_2$ and $(G,\nu):\Delta_2 \to \Delta_3$ be morphisms. The composition $(G, \nu) \circ (F, \eta) =(G \circ F, (\nu \bullet F) \circ \eta)$ determines the 2-functor $(G \circ F)_*: \cat C^{\Delta_1} \to \cat E ^{\Delta_3}$, and we have $(G \circ F)_* = G_* \circ F_*$.
Indeed, we verify this on objects 
\[
\begin{array}{rl}
(G\circ F)_*(X) & = ((\nu \bullet F) \circ \eta)_A(X) = \nu_{F(A)}(\eta_A(X)) = G_*(\eta_A(X)) = G_*(F_*(X)),
\end{array}
\]
on 1-morphisms
\[
\begin{array}{rl}
(G\circ F)_*(g, \phi) & = ((G\circ F)(g), ((\nu \bullet F) \circ \eta)_A(\phi))\\
& = (G(F(g)), \nu_{F(A)}(\eta_A(\phi))) \\
& = G_*(F(g), \eta_A(\phi)) = G_*(F_*(g, \phi)),
\end{array}
\]
and on 2-morphisms
\[
\begin{array}{rl}
(G \circ F)_*(\widehat{\psi}) & = (((\nu \bullet F) \circ \eta)_A(\psi))^{\wedge}\\
& = ((\nu \bullet F)_A \circ \eta_A (\psi))^{\wedge}\\
& = ((\nu_{F(A)} \circ \eta_A) (\psi))^{\wedge} = G_*(\widehat{\eta_A(\psi)}) = G_*(F_*(\widehat{\psi})).
\end{array}
\]

We have established that $(F, \eta)$ naturally defines a 2-functor $F_*$, so it only remains to show that this 2-functor is Lipschitz. Since $F$ is (1-)Lipschitz, we have $\weight_{\cat D}(F(g)) \leq \weight_{\cat C}(g)$ for all $g \in \cat C _1$. Thus, $$\begin{array}{rl} 
    \left(\weight_{\cat D}^{\nabla}\right)_1(F_*(g, \phi)) & = \left(\weight_{\cat D}^{\nabla}\right)_1((F(g), \eta_A(\phi)))\\ & = \weight_{\cat D}(F(g)) \leq \weight_{\cat C}(g) = \left(\weight_{\cat C}^{\Delta}\right)_1((g, \phi)).
        \end{array}$$ 
        By definition, $\left(\weight_{\cat C} ^{\Delta}\right)_2=\left(\weight_{\cat D}^{\nabla}\right)_2 =0,$ so there is nothing to check.
\end{proof}

\begin{lem}\label{lem:2_functor_interleaving_equals_weighted_category_interleaving}
    Let $\Delta:\cat{C} \to \cat{Cat}$ be a 2-functor, let $\weight$ be a Lawvere weight on $\cat{C}$ and let $\cat{C}^\Delta$ and $\weight^\Delta$ be defined as above. Then the 2-functor interleaving distance on $\mathrm{Im}(\Delta) = \cat{C}_0^\Delta$ agrees with the 2-weighted 2-category interleaving distance, $d_{\Delta,\weight} = d_{\cat{C}^\Delta,\weight^\Delta}$.
\end{lem}

\begin{proof}
    Let $X,Y \in \mathrm{Im}(\Delta) = \cat{C}_0^\Delta$. Unpacking the definitions, there is a $(g,h)$-interleaving (in the sense of 2-functor interleaving) of $X$ and $Y$ if and only if there is a $(g,h,\alpha,\beta)$-interleaving (in the sense of weighted 2-categories). Moreover,
    \[
    \max\{\weight(g),\weight(h)\} = \max \{\weight^\Delta_1(g),\weight^\Delta_1(h),\weight^\Delta_2(\alpha),\weight^\Delta_2(\beta)\},
    \]
    so this completes the proof.
\end{proof}

From Lemmas \ref{lem:faithful_functor_2_functor_to_weighted_2_category} and \ref{lem:2_functor_interleaving_equals_weighted_category_interleaving}, we immediately deduce the following result.

\begin{thm}\label{thm:2-Functors_to_2-categories}
    The map from Definition \ref{def:2_functor_to_weighted_2_category} induces a functor $\cat{W2Fun}(\cdot,\cat{Cat}) \to \cat{W2Cat}$ and an isometry at the level of interleaving distances.
\end{thm}

\subsection{Locally Persistent Categories}\label{sec:locally_persistent_categories} In~\cite{scoccola2020locally}, Scoccola introduces an alternate perspective on interleaving distances through the lens of enriched categories. We will show that these interleaving distances can be realized in our 2-category setting. In what follows, we use the font $\LPC{D}$ to denote an enriched 1-category.

\begin{defn}[Locally Persistent Category \cite{scoccola2020locally}]\label{defn:LPC}
    A \emph{locally persistent category (LPC)} is a 1-category $\LPC{D}$ enriched in $\cat{Fun}(\cat{R},\cat{Set})$, where $\cat{R}$ is the poset category associated to the nonnegative reals $\R_{\geq 0}$ and $\cat{Set}$ is the category of sets. Explicitly, this entails the following data:
    \begin{enumerate}
        \item A class of objects $\LPC{D}_0$.
        \item For each $A,B \in \LPC{D}_0$ and each $s \geq 0$, a set of morphisms $\LPC{D}_1(A,B)_s$. 
        \begin{enumerate}
            \item For $A = B$ and $s = 0$, there is a distinguished identity morphism $1_A \in \LPC{D}_1(A,A)_0$. 
            \item Compositions of morphisms respect the $\R_{\geq 0}$-grading in the sense that $g \in \LPC{D}_1(A,B)_s$ and $h \in \LPC{D}_1(B,C)_t$ compose to give $hg \in \LPC{D}_1(A,C)_{s + t}$. 
        \end{enumerate}
        \item For $s \leq t$, there is a \emph{shift operation} $\mathrm{S}_{s,t}:\LPC{D}_1(A,B)_s \rightarrow \LPC{D}_1(A,B)_t$.
    \end{enumerate}
    This data is required to satisfy various natural axioms; we refer to \cite[Definition 3.11]{scoccola2020locally} for details.
\end{defn}

Interleaving distance in the LPC setting is defined as follows.

\begin{defn}[LPC Interleaving Distance \cite{scoccola2020locally}]
    Let $\LPC{D}$ be an LPC and let $A,B \in \LPC{D}_0$. We say that $A$ and $B$ are \emph{$(s,t)$-interleaved} if there exist $f \in \LPC{D}_1(A,B)_s$ and $g \in \LPC{D}_1(B,A)_t$ such that $gf = \mathrm{S}_{0,s+t}(1_A)$ and $fg = \mathrm{S}_{0,s+t}(1_B)$. 

    The \emph{LPC interleaving distance} between $A$ and $B$ is 
    \[
    d_\LPC{D}(A,B) = \inf \{\max\{s,t\} \mid \mbox{$A$ and $B$ are $(s,t)$-interleaved}\}.
    \]
\end{defn}

Next we will show that $d_\LPC{D}$ is equivalent to a 2-weighted 2-category interleaving distance. We will show that the connection between these distances is natural, in a precise sense.

\subsubsection{From LPCs to 2-Categories}

We now fit the notion of LPCs into the framework of 2-categories, using a construction similar to that of $\cat{GH}$ above. 

\begin{defn}[2-Category Associated to an LPC]\label{defn:LPC_to_2_category} Let $\LPC{D}$ be an LPC. We define a  2-category $\cat{C}^\LPC{D}$ consisting of the following data:
    \begin{enumerate}
        \item $\cat{C}^\LPC{D}_0 = \LPC{D}_0$.
        \item For $A,B \in \LPC{D}_0$, 
        \[
        \cat{C}^\LPC{D}_1(A,B) = \{(g,s) \mid g \in \LPC{D}_1(A,B)_s\}.
        \]
        Compositions are defined by
        \[
        (h,t) \circ (g,s) = (hg,s+t)
        \]
        and the identity morphism is $(1_A,0)$.
        \item For $(g,s),(h,t) \in \cat{C}^\LPC{D}_1(A,B)$, the class of 2-morphisms $\cat{C}^\LPC{D}_2(g,h)$ contains a unique 2-morphism, denoted  $\alpha_{g,h}$, if and only if $s \leq t$ and $\mathrm{S}_{s,t}(g) = h$. Vertical and horizontal compositions are defined, respectively, by the formulas 
        \[
        \alpha_{g,h} \alpha_{f,g} = \alpha_{f,h} \quad \mbox{and} \quad \alpha_{h,k} \bullet \alpha_{f,g} = \alpha_{hf,kg}.
        \]
        One can use the axioms of an LPC~\cite[Definition 3.11]{scoccola2020locally} to deduce that these compositions are well-defined. 
    \end{enumerate}
\end{defn}

The 2-category defined above inherits a Lawvere 2-weight.

\begin{lem}\label{lem:LPC_lawvere}
    Let $\LPC{D}$ be an LPC with associated 2-category $\cat{C}^\LPC{D}$. Let $\weight^\LPC{D}_1:\cat{C}^\LPC{D}_1 \to \R$ be the function $\weight^\LPC{D}_1((g,s)) = s$ and let $\weight^\LPC{D}_2:\cat{C}^\LPC{D}_2 \to \R$ be constantly zero. Then $\weight^\LPC{D} = (\weight^\LPC{D}_1,\weight^\LPC{D}_2)$ is a Lawvere 2-weight.
\end{lem}

\begin{proof}
    The properties of $\weight^\LPC{D}_2$ are trivial, so we only need to check the properties of $\weight^\LPC{D}_1$.
    By definition, $\weight^\LPC{D}_1((1_A,0)) = 0$. For $(g,s) \in \cat{C}_1^\LPC{D}(A,B)$ and $(h,t) \in \cat{C}_1^\LPC{D}(B,C)$, we have that 
    \[
    \weight^\LPC{D}((h,t) \circ (g,s)) = \weight^\LPC{D}_1((gh,s+t) = s + t = \weight^\LPC{D}_1((g,s)) + \weight^\LPC{D}_1((h,t)).
    \]
\end{proof}

Similar to the treatement we gave in the 2-functor setting, we now wish to show that the construction of Definition \ref{defn:LPC_to_2_category} is functorial. Let $\LPC{D}$ and $\LPC{E}$ be locally persistent categories. Following \cite[Definition 3.16]{scoccola2020locally}, we define the appropriate notion of morphism between LPCs as follows. A \emph{locally persistent functor} $F: \LPC D \to \LPC E$ is a functor which is compatible with the $\mathbb{R}$-grading on morphism sets and which commutes with the shift functor (Defintion~\ref{defn:LPC}). More precisely, for any $s \in \mathbb R$, the functor $F$ induces a map on morphism sets $$\LPC D_1(A, B)_s \to \LPC E_1(F(A), F(B))_{s},$$ and for any $s \leq t$ and $g \in \LPC D_1$, we have $\mathrm S_{s, t}(F(g)) = F(\mathrm S_{s,t}(g))$. Let $\cat{LPC}$ denote the category whose objects are LPC's and whose morphisms are locally persistent functors. 

\begin{lem}\label{lem:LPC_to_2_category_functorial}
 The map taking an LPC $\LPC{D}$ to the 2-weighted 2-category $(\cat{C}^\LPC{D},\mathrm{W}^\cat{D})$ defines a faithful functor $\cat{LPC} \to \cat{W2Cat}$.
\end{lem}

\begin{proof}
Given LPC's $\LPC D$, $\LPC E$, and a LP-functor $F: \LPC D \to \LPC E$, we define $F_*:\cat{C}^{\LPC D} \to \cat C ^{\LPC E}$ as follows. On objects, $F_*$ is given by $F$ since $\cat C ^{\LPC D}_0 = \LPC D _0$ and $\cat C ^{\LPC E}_0 = \LPC E _0$. Given a morphism $(g, s): A \to B \in \cat C _1^{\LPC D}$, so that $g \in \LPC D_1(A, B)_s$, we have $F(g) \in \LPC E _1(F(A), F(B))_s$. We thus define
\[F_*((g, s)) = (F(g), s): F(A) \to F(B) \in \cat C ^{\LPC E}_1.
\]
Composition of 1-morphisms is preserved. Indeed, we have
\[
\begin{array}{rl}
F_*((h, t) \circ (g, s)) & = F_*(hg, s+t)  = (F(hg), s+t) \\ & = (F(h) \circ F(g), t+s) = (F(h), t) \circ (F(g), s) = F_*((h, t)) \circ F_*((g, s)).
\end{array}
\]
There is a 2-morphism $\alpha_{g,h} \in \cat C _2 ^{\LPC D}((g, s), (h, t))$ if and only if $s \leq t$ and $\mathrm S _{s,t}(g) = h$. Since $F$ commutes with shifts, this yields $\mathrm S_{s,t}(F(g)) = F(\mathrm S_{s,t}(g)) = F(h),$ and so we have a 2-morphism $\alpha_{F(g), F(h)} \in \cat C_2 ^{\LPC E}((F(g), s), (F(h), t))$. We define 
\[
F_*(\alpha_{g, h}) = \alpha_{F(g), F(h)}.
\]
Then vertical composition and horizontal compositions are preserved:
\[
\begin{array}{rl}
F_*(\alpha_{g,h} \alpha_{f,g}) & = F_*(\alpha_{f,h})  = \alpha_{F(f), F(h)} = \alpha_{F(g), F(h)} \alpha_{F(f), F(g)} = F_*(\alpha_{g,h}) F_*(\alpha_{f, g})
\end{array}
\]
and
\[
\begin{array}{rl}
F_*(\alpha_{h,k} \bullet \alpha_{f,g}) & = F_*(\alpha_{hf, kg} ) = \alpha_{F(hf), F(kg)} \\ & = \alpha_{F(h)F(f), F(k)F(g)}  = \alpha_{F(h), F(k)} \bullet \alpha_{F(f), F(g)}  = F_*(\alpha_{h, k}) \bullet F_*(\alpha_{f, g}).
\end{array}
\]
This association is functorial; that is, given $F: \LPC D \to \LPC E$ and $G: \LPC E \to \LPC F$, we have $(G \circ F)_* = G_* \circ F_*$. Indeed, this is clear on objects, on 1-morphisms, we have
\[\begin{array}{rl}
(G\circ F)_*(g,s) & = (G(F(g)), s)  = G_*(F(g), s)  = G_*(F_*(g, s))  = (G_* \circ F_*)(g,s)
\end{array}
\]
and on 2-morphisms, we have 
\[\begin{array}{rl}
(G\circ F)_*(\alpha_{g,h}) & = \alpha_{G(F(g)), G(F(h))}  = G_*(\alpha_{F(g), F(h)})  = (G_* \circ F_*)(\alpha_{g, h}).
\end{array}
\]
Moreover, this functor is faithful. To see this, let $F, G: \LPC D \to \LPC E$ be LP-functors with $F_* = G_* : \cat C ^{\LPC D} \to \cat C ^{\LPC E}$. Then $F(A) = G(A)$ for all $A \in \LPC D_0 = \cat C ^{\LPC D}_0,$ and $F_*((g,s)) = G_*((g, s))$ for all $(g,s) \in \cat C ^{\LPC D}_1.$ This second statement implies $(F(g), s) = (G(g), s)$, so that $F(g) = G(g)$. Since $F$ and $G$ agree on objects and morphisms, we must have $F = G$.

It only remains to check that $F_\ast$ is Lipschitz. Indeed, a stronger condition holds, as we have
\[\begin{array}{rl}
\weight ^{\LPC E}_1(F_*(g,s)) & =\weight ^{\LPC E}_1((F(g), s)) = s = \weight ^{\LPC D}_1( (g,s)).
\end{array}
\]
\end{proof}

\begin{lem}\label{lem:equivalence_to_LPC}
    Let $\LPC{D}$ be an LPC and let $\cat{C}^\LPC{D}$ and $\weight^\LPC{D}$ be defined as above. Then the LPC interleaving distance on $\LPC{D}_0 = \cat{C}_0^\LPC{D}$ agrees with the 2-weighted 2-category interleaving distance, $d_{\LPC{D}} = d_{\cat{C}^\LPC{D},\weight^\LPC{D}}$.
\end{lem}

\begin{proof}
    This follows immediately from the construction, as $A,B \in \LPC{D}_0 = \cat{C}_0^\LPC{D}$ are $(s,t)$ -interleaved via $f \in \LPC{D}_1(A,B)_s$ and $g \in \LPC{D}_1(B,A)_t$, in the sense of LPC interleaving distance, if and only if they are $\big((f,s),(g,t),\alpha_{f,g},\alpha_{g,f}\big)$-interleaved, in the sense of 2-weighted 2-category interleaving distance.
\end{proof}

From Lemmas \ref{lem:LPC_to_2_category_functorial} and \ref{lem:equivalence_to_LPC}, we deduce the following result.

\begin{thm}\label{thm:LPC_to_2_functor}
    The map from Definition \ref{defn:LPC_to_2_category} induces a faithful functor $\cat{LPC} \to \cat{W2Cat}$ and an isometry at the level of interleaving distances.
\end{thm}

\section{Stability}\label{sec:stability}

An important property of the topological invariants of datasets studied in TDA is that they are stable under perturbations of the data. The prototypical quantification of this statement is the main theorem of \cite{cohen2005stability}, stated here informally: for a topological space $\set{X}$, continuous maps $f,g: \set{X} \to \R$, and persistence diagrams $D(f)$ and $D(g)$ obtained from degree-$k$ sublevel set persistent homology, we have (under mild tameness assumptions)
\begin{equation}\label{eqn:bottleneck_stability}
d_\mathrm{B}(D(f),D(g)) \leq \|f - g\|_{L^\infty(\set{X})},
\end{equation}
where $d_\mathrm{B}$ denotes bottleneck distance (see Example \ref{ex:multiplication_group}).

A common theme in the literature on interleaving distances is the idea that certain stability results of TDA can be viewed naturally from the categorical perspective; for example, see \cite[Theorem 4.3]{de2018theory}, \cite[Theorem 4.2.2]{scoccola2020locally}, \cite[Theorem 3.16]{bubenik2015metrics} and \cite[Theorem 4.6]{bubenik2017interleaving}. Stability in the setting of 2-weighted 2-categories is very natural, and more general than other notions in the literature, in light of the results of the previous section. We give a brief overview of stability and some applications in this section.

\subsection{Stability for 2-Weighted 2-Categories} Recall that $\cat{W2Cat}$ denotes the 1-category whose objects are small 2-weighted 2-categories and whose morphisms are Lipschitz 2-functors. Let $\cat{EPMet}$ denote the 1-category whose objects are extended pseudometrics and whose morphisms are Lipschitz functions.

\begin{thm}[Stability]\label{thm:stability}
    The map taking a 2-weighted 2-category $(\cat{C},\mathrm{W})$ to the extended pseudometic space $(\cat{C}_0,d_{\cat{C},\mathrm{W}})$ defines a functor from $\cat{W2Cat}$ to $\cat{EPMet}$. In particular, for any Lipschitz 2-functor of 2-weighted 2-categories $\Theta:(\cat{C},\mathrm{W}) \to (\cat{C}',\mathrm{W}')$, we have
    \[
    d_{\cat{C}',\mathrm{W}'}(\Theta(A),\Theta(B)) \leq d_{\cat{C},\mathrm{W}}(A,B)
    \]
    for all $A,B \in \cat{C}_0$.
\end{thm}

\begin{proof}
    We only need to check the Lipschitz inequality. Given a diagram
    \[\begin{tikzcd}
	A && B \\
	\\
	A && B
	\arrow["g", from=1-1, to=1-3]
	\arrow["h", from=3-3, to=3-1]
	\arrow[""{name=0, anchor=center, inner sep=0}, "{1_A}"', curve={height=12pt}, from=1-1, to=3-1]
	\arrow[""{name=1, anchor=center, inner sep=0}, "hg", curve={height=-12pt}, from=1-1, to=3-1]
	\arrow[""{name=2, anchor=center, inner sep=0}, "gh"', curve={height=12pt}, from=1-3, to=3-3]
	\arrow[""{name=3, anchor=center, inner sep=0}, "{1_B}", curve={height=-12pt}, from=1-3, to=3-3]
	\arrow["\alpha"', shorten <=5pt, shorten >=5pt, Rightarrow, from=0, to=1]
	\arrow["\beta", shorten <=5pt, shorten >=5pt, Rightarrow, from=3, to=2]
\end{tikzcd}\]
    in $\cat{C}$, we obtain a diagram
    \[\begin{tikzcd}
	\Theta(A) && \Theta(B) \\
	\\
	\Theta(A) && \Theta(B)
	\arrow["\Theta(g)", from=1-1, to=1-3]
	\arrow["\Theta(h)", from=3-3, to=3-1]
	\arrow[""{name=0, anchor=center, inner sep=0}, "{\Theta(1_A)}"', curve={height=12pt}, from=1-1, to=3-1]
	\arrow[""{name=1, anchor=center, inner sep=0}, "\Theta(hg)", curve={height=-12pt}, from=1-1, to=3-1]
	\arrow[""{name=2, anchor=center, inner sep=0}, "\Theta(gh)"', curve={height=12pt}, from=1-3, to=3-3]
	\arrow[""{name=3, anchor=center, inner sep=0}, "{\Theta(1_B)}", curve={height=-12pt}, from=1-3, to=3-3]
	\arrow["\Theta(\alpha)"', shorten <=5pt, shorten >=5pt, Rightarrow, from=0, to=1]
	\arrow["\Theta(\beta)", shorten <=5pt, shorten >=5pt, Rightarrow, from=3, to=2]
\end{tikzcd}
\quad \mbox{or, equivalently, } \quad
\begin{tikzcd}
	{\Theta(A)} &&& {\Theta(B)} \\
	\\
	{\Theta(A)} &&& {\Theta(B)}
	\arrow[""{name=0, anchor=center, inner sep=0}, "{1_{\Theta(A)}}"', curve={height=12pt}, from=1-1, to=3-1]
	\arrow[""{name=1, anchor=center, inner sep=0}, "{\Theta(h)\Theta(g)}", curve={height=-12pt}, from=1-1, to=3-1]
	\arrow[""{name=2, anchor=center, inner sep=0}, "{\Theta(g)\Theta(h)}"', curve={height=12pt}, from=1-4, to=3-4]
	\arrow[""{name=3, anchor=center, inner sep=0}, "{1_{\Theta(B)}}", curve={height=-12pt}, from=1-4, to=3-4]
	\arrow["{\Theta(g)}", from=1-1, to=1-4]
	\arrow["{\Theta(h)}", from=3-4, to=3-1]
	\arrow["{\Theta(\alpha)}"', shorten <=5pt, shorten >=5pt, Rightarrow, from=0, to=1]
	\arrow["{\Theta(\beta)}", shorten <=5pt, shorten >=5pt, Rightarrow, from=3, to=2]
\end{tikzcd}
\]
    in $\cat{C}'$ via functoriality of $\Theta$. By the Lipschitz property \eqref{eqn:lipschitz_property} of $\Theta$, we have
    \[
    \max\{\mathrm{W}'_1(\Theta(g)),\mathrm{W}'_1(\Theta(h)),\mathrm{W}'_2(\Theta(\alpha)),\mathrm{W}'_2(\Theta(\beta))\} \leq \max\{\mathrm{W}_1(g),\mathrm{W}_1(h),\mathrm{W}_2(\alpha),\mathrm{W}_2(\beta)\}.
    \]
    The claim follows.
\end{proof}

\subsection{Stability for Monoidal Functors} Since most examples of interest of interleaving distances in this paper arise from monoidal functors, let us show how to apply Theorem \ref{thm:stability} in this context. Let $T:\cat{G} \to \cat{End}(\cat{X})$ and $T':\cat{G}' \to \cat{End}(\cat{X}')$ be monoidal functors and suppose that $\cat{G}$ and $\cat{G}'$ are endowed with monoidal weights $\mathrm{W}$ and $\mathrm{W}'$, respectively. 

\begin{defn}  
    An \emph{equivariant Lipschitz functor pair} is a pair $(H,K)$ consisting of a Lipschitz monoidal functor $H: \cat{G} \to \cat{G}'$ and a functor $K:\cat{X} \to \cat{X}'$ such that, for all $g \in  \cat{G}_0$, $K \circ T_g = T'_{H(g)} \circ K$. 
\end{defn}

Recall from Proposition \ref{prop:monoidal_functors_and_2_functors} that a monoidal functor $T: \cat{G} \to \cat{End}(\cat{X})$ can be promoted to a 2-functor into $\cat{Cat}$ via the delooping construction. Suppose that $\cat{G}$ is endowed with a monoidal weight $\mathrm{W}$; this can then be considered as a weight on the 1-category $\cat{B}\cat{G}$. We slightly abuse notation and denote the 2-weighted 2-category associated to this 2-functor data (Definition \ref{def:2_functor_to_weighted_2_category}) as $(\cat{C}^T,\mathrm{W}^T)$. 

\begin{prop}\label{prop:stability_actegory}
    An equivariant Lipschitz functor pair $(H,K)$ from $T:\cat{G} \to \cat{End}(\cat{X})$ to $T':\cat{G}' \to \cat{End}(\cat{X}')$ induces a Lipschitz 2-functor from $(\cat{C}^T,\mathrm{W}^T)$ to $(\cat{C}^{T'},\mathrm{W}^{T'})$. It follows that, for all $X,Y \in \cat{X}_0$, we have the Lipschitz bound on interleaving distances
    \[
    d_{T',\mathrm{W}'}(K(X),K(Y)) \leq d_{T,\mathrm{W}}(X,Y).
    \]
\end{prop}

\begin{proof}
    We define a 2-functor $\Theta:\cat{C}^T \to \cat{C}^{T'}$ by first defining it on objects via the observation that $\cat{C}^T_0 = \cat{X}_0$, hence we may define $\Theta(X) = K(X)$. The definition of $\Theta$ on 1 and 2-morphisms requires more work, and follows below.
   
    Let $(g,\phi):X \to Y$ be a morphism of $\cat{C}^T$; that is, we have $\phi:X \to T_g(Y)$ in $\cat{X}$. Define $\Theta((g,\phi)):\Theta(X) \to \Theta(Y)$ by $\Theta((g,\phi)) = (H(g),K(\phi))$. This is well-defined: indeed, it is the case that $K(\phi)$ is a morphism from $\Theta(X) = K(X)$ to $T_{H(g)}(\Theta(Y)) = T_{H(g)}(K(Y))$, since, definitionally, $K(\phi):K(X) \to K(T_g(Y))$, and $K(T_g(Y)) = T_{H(g)}'(K(Y))$, by equivariance. Finally, we need to show that this preserves identities and compositions. Identity preservation follows immediately by the assumption that $H$ is a monoidal functor and $K$ is a functor, since $\Theta((e,1_X)) = (H(e),K(1_X)) = (e',1_{K(X)})$, where $e'$ denotes the identity object in $\cat{G}'$. To establish preservation of compositions, consider $(g,\phi):X \to Y$ and $(h,\psi):Y \to Z$. Then, using (monoidal) functoriality and equivariance, we have
    \begin{align*}
    \Theta((h,\psi) \circ (g,\phi)) &= \Theta((gh,T_g(\psi) \circ \phi)) = (H(gh),K(T_g(\psi) \circ \phi)) \\
    &= (H(g)H(h),K(T_g(\psi)) \circ K(\phi)) = (H(g)H(h),T_{H(g)}'(K(\psi)) \circ K(\phi)) \\
    &= (H(h),K(\psi)) \circ (H(g),K(\phi)).
    \end{align*}

    Now consider a 2-morphism in $\cat{C}^T_2((g,\phi),(h,\psi))$. There are two cases to consider: $(g,\phi) = (h,\psi)$ and the 2-morphism is of the form $1_{g,\phi}$, or $g = e$, $\phi = 1_X$, and the 2-morphism is of the form $\hat{\psi}$. In the former case, define $\Theta(1_{g,\phi}) = 1_{H(g),K(\psi)}$, and in the latter, define $\Theta(\hat{\psi}) = \widehat{K(\psi)}$. Verification that this is well-defined, and that it preserves identities and compositions is similar to the work done in the previous paragraph.
   
    It is obvious, by definition, that $\Theta$ satisfies the Lipschitz property \eqref{eqn:lipschitz_property}, so this completes the proof. The last statement of the proposition then follows from Lemma \ref{lem:2_functor_interleaving_equals_weighted_category_interleaving} and Theorem \ref{thm:stability}.
\end{proof}

\subsection{Sublevel Set Persistent Homology} A common application of category-theoretic metrics in TDA is to illustrate the stability under perturbations of various flavors of persistent homology; see, for example, the statement at the beginning of this section, from \cite{cohen2005stability}, or related results in~\cite{chazal2009proximity,bubenik2015metrics,de2018theory}. We extend these results to the setting of prosets with monoid actions that we considered in Section \ref{sec:interleaving_group_actions}. Our general approach is similar to those previously explored in the literature, so we give a somewhat terse exposition here; please see the above references for more details.

Let $(\set{P},\leq)$ be a proset with associated category $\cat{P}$. Suppose that $\set{P}$ is endowed with an action by a group $\set{G}$ with associated proset monoidal category $\cat{G}^\set{P}$ and proset action monoidal functor $T:\cat{G}^\set{P} \to \cat{End}(\cat{P})$ (Definitions \ref{def:proset_monoidal_category} and \ref{def:preordered_set_action_functor}). Finally, suppose that $\set{G}$ is endowed with a monoidal weight $\mathrm{W}$. 

Consider the (1-)category of functions from some fixed topological space $\set{X}$ into $\set{P}$, denoted $\cat{Map}(\set{X},\set{P})$---the objects of this category are functions $\phi:\set{X} \to \set{P}$, and there is a unique morphism from $\phi$ to $\psi$, denoted $\phi \leq \phi$, if $\phi(x) \leq \psi(x)$ for all $x \in \set{X}$. 

Let $\cat{Top}$ denote the category of topological spaces with continuous maps and let $\cat{Vec} = \cat{Vec}_\mathbb{k}$ denote the category of vector fields over some fixed category $\mathbb{k}$ with linear maps. We now define several additional monoidal functors; we leave it to the reader to check that these are well-defined.
\begin{enumerate}
    \item Let $T':\cat{G}^\set{P} \to \cat{End}(\cat{Map}(\set{X},\set{P}))$ be defined as follows:
    \begin{enumerate}
        \item for $g \in \set{G}$, $\phi:\set{X} \to \set{P}$, $T'_g(\phi):\set{X} \to \set{P}$ is the map $T'_g(\phi)(x) = g(\phi(x))$---in other words, $T'_g(\phi) = g \circ \phi$;
        \item for a morphism $\alpha_{g,h}$ in $\cat{G}^\set{P}$ (i.e., $g(p) \leq h(p)$ for all $p \in \set{P}$), $T'(\alpha_{g,h}):T'_g \Rightarrow T'_h$ has component at $\phi$ given by $T'(\alpha_{g,h})_\phi = g \circ \phi  \leq h \circ \phi $.
    \end{enumerate}
    \item Let $\big(\cat{G}^\set{P}\big)^{\mathrm{op}}$ denote the opposite category of $\cat{G}^\set{P}$; there is a unique morphism denoted $\overline{\alpha}_{g,h}:g \to h$ if and only if $g(p) \geq h(p)$ for all $p \in \set{P}$. Let $T'':\big(\cat{G}^\set{P}\big)^{\mathrm{op}} \to \cat{End}(\cat{Fun}(\cat{P},\cat{Top}))$ be defined as follows:
    \begin{enumerate}
        \item for $g \in \set{G}$ and a functor $F:\cat{P} \to \cat{Top}$, $T''_g(F)$ is the functor defined by $T''_g(F)(p) = F(g^{-1}(p))$ and $T''_g(F)(p \leq q) = F(g^{-1}(p) \leq g^{-1}(q))$; for a natural transformation $\beta:F \Rightarrow G$ of functors $F,G:\cat{P} \to \cat{Top}$, $T''_g(\beta)$ has component at $p$ given by $T''_g(\beta)_p = \beta_{g^{-1}(p)}$, the component of $\beta$ at $g^{-1}(p)$;
        \item for a morphism $\overline{\alpha}_{g,h}$ in $\big(\cat{G}^\set{P}\big)^{\mathrm{op}}$, $T''(\overline{\alpha}_{g,h}):T''_g \Rightarrow T''_h$ has component at $F:\cat{P} \to \cat{Top}$ given by the natural transformation
        \[
        T''(\overline{\alpha}_{g,h})_F:T_g''(F) \Rightarrow T_h''(F)
        \]
        with component at $p \in \set{P}$ given by
        \[
        \big(T''(\overline{\alpha}_{g,h})_F\big)_p = F(g^{-1}(p) \leq h^{-1}(p))
        \]
        (note that $g(p) \geq h(p)$ for all $p \in \set{P}$ implies $g^{-1}(p) \leq h^{-1}(p)$).
    \end{enumerate}
    \item Let $T''':\big(\cat{G}^\set{P}\big)^{\mathrm{op}} \to \cat{End}(\cat{Fun}(\cat{P},\cat{Vec}))$ be defined similarly to $T''$ (observe that the definition of $T''$ didn't use the assumption that the target in the functor category was $\cat{Top}$ in any meaningful way, so that the definition extends in general to a monoidal functor $\big(\cat{G}^\set{P}\big)^{\mathrm{op}} \to \cat{End}(\cat{Fun}(\cat{P},\cat{C}))$ for any category $\cat{C}$). 
\end{enumerate}

We then define two functors which link the constructions above.

\begin{enumerate}
    \item The \emph{sublevel set filtration functor} $S:\cat{Map}(\set{X},\set{P}) \to \cat{Fun}(\cat{P},\cat{Top})$ is defined as follows.
    \begin{enumerate}
        \item For $\phi \in \cat{Map}(\set{X},\set{P})_0$, $S(\phi):\cat{P} \to \cat{Top}$ is the functor with 
        \begin{enumerate}
            \item for $p \in \set{P} = \cat{P}_0$, $S(\phi)(p) = \phi^{-1}(\{r \in \set{P} \mid r \leq p\}) \subset \set{X}$, endowed with the subspace topology, so that $S(\set{X},f)(p) \in \cat{Top}_0$;
            \item For the morphism $p \leq q$  in $\cat{P}_1(p,q)$, 
            \[
            S(\phi)(p \leq q): \phi^{-1}(\{r \in \set{P} \mid q \leq p\}) \hookrightarrow \phi^{-1}(\{q \in \set{P} \mid r \leq q\})
            \]
            is the inclusion map.
        \end{enumerate}
        \item For a morphism $\phi \leq \psi$ in $\cat{Map}(\set{X},\set{P})$, $S(\phi \leq \psi):S(\phi) \Rightarrow S(\psi)$ is the natural transformation with component at $p \in \set{P}$ given by inclusion
        \[
        S(\phi \leq \psi)_p:\phi^{-1}(\{r \in \set{P} \mid r \leq p\}) \to \psi^{-1}(\{r \in \set{P} \mid r \leq p\});
        \]
        this is well-defined due to the assumption that $\phi \leq \psi$ and it is easy to see that the requisite diagram commutes in order for this to define a natural transformation.
    \end{enumerate}
    \item For any choice of degree $n \in \mathbb{Z}_{\geq 0}$, let $H_n:\cat{Top} \to \cat{Vec}$ denote the \emph{degree-$n$} homology functor over the field $\mathbb{k}$. This induces the \emph{degree-$n$ persistent homology functor} $PH_n:\cat{Fun}(\cat{P},\cat{Top}) \to \cat{Fun}(\cat{P},\cat{Vec})$:
    \begin{enumerate}
        \item for a functor $F:\cat{P} \to \cat{Top}$, we define $PH_n(F) = H_n \circ F$
        \item for a natural transformation $\beta:F \Rightarrow G$, $H_n(\beta):H_n(F) \Rightarrow H_n(G)$ has component at $p \in P$ given by $H_n(\beta)_p = H_n(\beta_p)$.
    \end{enumerate}
\end{enumerate}

We now apply these constructions to generalize the classical TDA stability result \eqref{eqn:bottleneck_stability}. Let $1_{\big(\cat{G}^\set{P}\big)^{\mathrm{op}}}$ denote the identity functor on $\big(\cat{G}^\set{P}\big)^{\mathrm{op}}$ and let $O:\cat{G}^\set{P} \to \big(\cat{G}^\set{P}\big)^{\mathrm{op}}$ be the \emph{opposite functor} which is the identity on objects and which takes  $\alpha_{g,h}$ to $\overline{\alpha}_{h,g}$. 

\begin{prop}\label{prop:sublevel_stability}
    With the notation above, we have the following statements.
    \begin{enumerate}
        \item $(O,S)$ is an equivariant Lipschitz functor pair for $T'$ and $T''$.
        \item $(1_{\big(\cat{G}^\set{P}\big)^{\mathrm{op}}},PH_n)$ is an equivariant Lipschitz functor pair for $T''$ and $T'''$, for any degree $n$. 
        \item Interleaving distance for sublevel set persistent homology is stable, in the sense that for any functions $\phi,\psi:\set{X} \to \set{P}$, 
        \[
        d_{T'''}(PH_n \circ S(\phi),PH_n \circ S(\psi)) \leq d_{T''}(S(\phi),S(\psi)) \leq d_{T'}(\phi,\psi).
        \]
        \item The interleaving distance on $\cat{Map}(\set{X},\set{P})$ can be expressed as
        \[
        d_{T'}(\phi,\psi) = \inf \{\max\{\weight(g),\weight(h)\} \mid \phi(x) \leq g(\psi(x)), \; \psi(x) \leq h(\phi(x)) \; \forall \, x \in \set{X}\}.
        \]
        \item When $\set{P} = \R$ with $\mathrm{W}(p) = |p|$, the above results recover classical stability of sublevel set persistent homology  \eqref{eqn:bottleneck_stability}.
    \end{enumerate}
\end{prop}

\begin{proof}
    To prove point (1), we need to show that for all $g \in \set{G}$, $S \circ T'_g = T''_{g} \circ S$. First, let $\phi \in \cat{Map}(\set{X},\set{P})_0$. For any $p \in \set{P}$, 
    \[
    S(T'_g(\phi))(p) = S(g \circ \phi)(p) = (g \circ \phi)^{-1} \big(\{r \in \set{P} \mid r \leq p\}\big) = \{x \in \set{X} \mid g(\phi(x)) \leq p\}
    \]
    and
    \begin{align*}
    T''_{g} (S(\phi))(p) = S(\phi)(g^{-1}(p)) &= \phi^{-1}\big(\{r \in \set{P} \mid r \leq g^{-1}(p)\}\big) \\
    &= \{x \in \set{X} \mid \phi(x) \leq g^{-1}(p)\} = \{x \in \set{X} \mid g(\phi(x)) \leq p\},
    \end{align*}
    so that $S \circ T'_g = T''_{g} \circ S$ at the object level. A similar computation shows that the equivariance property holds at the morphism level. Since $\cat{G}^\set{P}_0 = (\big(\cat{G}^\set{P}\big)^{\mathrm{op}})_0 = \set{G}$ share the same monoidal weight, the Lipschitz property is obvious, so this proves point (1). 

    Point (2) also only requires us to show the equivariance property. For a functor $F:\cat{P} \to \cat{Top}$ and point $p \in \set{P}$, we have
    \[
    PH_n(T''_g(F))(p) = H_n\big(T''_g(F)(p)\big) = H_n\big(F(g^{-1}(p))\big)
    \]
    and
    \[
    T'''_g(PH_n(F))(p) = PH_n(F)(g^{-1}(p)) = H_n\big(F(g^{-1}(p))\big).
    \]
    Equivariance at the level of morphisms follows similarly.

    Point (3) is a direct corollary of points (1) and (2) and Proposition \ref{prop:stability_actegory}. Point (4) follows by unpacking the definitions.

    Finally, we prove point (5). By point (4), the interleaving distance on $\cat{Map}(\set{X},\R)$ is given by 
    \[
    d_{T'}(\phi,\psi) = \inf \{ \max \{|t|,|s|\} \mid \phi(x) \leq \psi(x) + t, \, \psi(x) \leq \phi(x) + s \, \forall \, x \in \set{X}\},
    \]
    and this is equal to $\|\phi - \psi\|_{L^\infty(\set{X})}$ by an elementary argument. The objects $PH_n \circ S(\phi)$ and $PH_n \circ S(\psi)$ are standard (one parameter) persistent homology modules, and we claim that $d_{T'''}$ is the standard interleaving distance on persistence modules in $\cat{Fun}(\cat{R},\cat{Vec})$ (i.e., the one described in the introduction in equations \eqref{eqn:basic_interleaving_distance_1} and \eqref{eqn:basic_interleaving_distance_2}), where $\cat{R}$ is the poset category associated to $\R$. Indeed, an argument similar to the one used in the proof of Proposition \ref{prop:equivalence_to_flow} shows that it suffices to consider $(u,u)$-interleavings (rather than general $(u,v)$-interleavings). Such an interleaving of persistence modules $M,N:\cat{R} \to \cat{Vec}$ consists of natural transformations $\phi:M \to T'''_u N$ and $\psi:N \to T'''_u M$ such that the diagrams 
\begin{equation}\label{eqn:new_interleaving_distance}\begin{tikzcd}
	M && {T'''_{2u}M} \\
	& {T'''_u N}
	\arrow[""{name=0, anchor=center, inner sep=0}, "\phi"',Rightarrow, from=1-1, to=2-2]
	\arrow[""{name=1, anchor=center, inner sep=0}, "{T'''(\overline{\alpha}_{0,2u})_M}", Rightarrow,from=1-1, to=1-3]
	\arrow["{T'''_u(\psi)}"',Rightarrow, from=2-2, to=1-3]
\end{tikzcd}
\mbox{ and }
\begin{tikzcd}
	N && {T'''_{2u}N} \\
	& {T'''_u M}
	\arrow[""{name=0, anchor=center, inner sep=0}, "\psi"',Rightarrow, from=1-1, to=2-2]
	\arrow[""{name=1, anchor=center, inner sep=0}, "{T'''(\overline{\alpha}_{0,2u})_N}", Rightarrow,from=1-1, to=1-3]
	\arrow["{T'''_u(\phi)}"',Rightarrow, from=2-2, to=1-3]
\end{tikzcd}\end{equation}
commute. The existence of $\overline{\alpha}_{0,2u}$ means that $u \leq 0$; let us make the change of variables $t = -u \geq 0$. The fact that $\phi$ and $\psi$ are natural transformations then imply the existence of commutative diagrams \eqref{eqn:basic_interleaving_distance_1} for all $r \leq s$. Moreover, after unraveling the definitions, we see that the diagrams \eqref{eqn:new_interleaving_distance} imply the existence of diagrams \eqref{eqn:basic_interleaving_distance_2} for all $s$. Putting this together, we get that standard interleaving distance between the sublevel set persistent homology modules is bounded above by $\|f-g\|_{L^\infty(\set{X})}$, which yields an upper bound on bottleneck distance under reasonable tameness conditions, by standard results of TDA~\cite{chazal2009proximity}.
\end{proof}

\section{Discussion}\label{sec:discussion}

This paper establishes the fundamental theory for some novel generalizations of interleaving distance, suggesting several directions for future research; we briefly discuss some of them here. First, we note that this work concerns the theory of interleaving distance, but was inspired by ideas coming from applied fields such as TDA and statistical shape analysis. A major question is whether any of the metrics considered here lead to computationally feasible methods. For example, can interleaving distances based on diffeomorphism group actions on 1-parameter persistence modules be computed using barcode decompositions and numerical methods from shape analysis? One of our initial motivations for developing this framework was to relate these distances to $p$-Wasserstein distances between barcodes, with a view toward a categorical proof of $p$-Wasserstein stability (a version of which was recently proved via combinatorial and algebraic methods in \cite{skraba2020wasserstein}). At this point, however, this remains an open question. There are also many possibilities for theoretical extensions of these ideas. One obvious direction is to generalize our constructions and results to lax or bicategory settings. Another natural idea is to extend the 2-weighted 2-category interleaving distances to higher categories. These generalizations should be motivated by natural applications, which we believe could be useful for studying (bounded, coherent) derived categories of schemes and their higher-categorical counterparts. Work of Neeman \cite{neeman2020metrics} highlights the utility of metrics in relation to categorical constructions in the triangulated setting.

\section*{Acknowledgements}

This research was partially supported by NSF grants  DMS 2107808 and DMS 2324962.

\bibliographystyle{plain}
  \bibliography{biblio}

\begin{thebibliography}{10}

\bibitem{abraham2012manifolds}
Ralph Abraham, Jerrold~E Marsden, and Tudor Ratiu.
\newblock {\em Manifolds, tensor analysis, and applications}, volume~75.
\newblock Springer Science \& Business Media, 2012.

\bibitem{bauer2017numerical}
Martin Bauer, Martins Bruveris, Philipp Harms, and Jakob M{\o}ller-Andersen.
\newblock A numerical framework for {S}obolev metrics on the space of curves.
\newblock {\em SIAM Journal on Imaging Sciences}, 10(1):47--73, 2017.

\bibitem{bauer2014overview}
Martin Bauer, Martins Bruveris, and Peter~W Michor.
\newblock Overview of the geometries of shape spaces and diffeomorphism groups.
\newblock {\em Journal of Mathematical Imaging and Vision}, 50:60--97, 2014.

\bibitem{bauer2022elastic}
Martin Bauer, Nicolas Charon, Eric Klassen, Sebastian Kurtek, Tom Needham, and
  Thomas Pierron.
\newblock Elastic metrics on spaces of euclidean curves: Theory and algorithms.
\newblock {\em arXiv preprint arXiv:2209.09862}, 2022.

\bibitem{bauer2015induced}
Ulrich Bauer and Michael Lesnick.
\newblock Induced matchings and the algebraic stability of persistence
  barcodes.
\newblock {\em Journal of Computational Geometry}, 6(2):162--191, 2015.

\bibitem{beg2005computing}
M~Faisal Beg, Michael~I Miller, Alain Trouv{\'e}, and Laurent Younes.
\newblock Computing large deformation metric mappings via geodesic flows of
  diffeomorphisms.
\newblock {\em International journal of computer vision}, 61:139--157, 2005.

\bibitem{benabou1967introduction}
Jean B{\'e}nabou, R~Davis, A~Dold, J~Isbell, S~MacLane, U~Oberst, J~E Roos, and
  Jean B{\'e}nabou.
\newblock Introduction to bicategories.
\newblock In {\em Reports of the midwest category seminar}, pages 1--77.
  Springer, 1967.

\bibitem{bobrowski2023universal}
Omer Bobrowski and Primoz Skraba.
\newblock A universal null-distribution for topological data analysis.
\newblock {\em Scientific Reports}, 13(1):12274, 2023.

\bibitem{botnan2020rectangle}
Magnus~Bakke Botnan, Vadim Lebovici, and Steve Oudot.
\newblock On rectangle-decomposable 2-parameter persistence modules.
\newblock In {\em 36th International Symposium on Computational Geometry, SoCG
  2020}, pages 1--16. Schloss Dagstuhl-Leibniz-Zentrum fur Informatik GmbH,
  Dagstuhl Publishing, 2020.

\bibitem{botnan2022introduction}
Magnus~Bakke Botnan and Michael Lesnick.
\newblock An introduction to multiparameter persistence.
\newblock {\em arXiv preprint arXiv:2203.14289}, 2022.

\bibitem{brown1987groups}
Ronald Brown.
\newblock From groups to groupoids: a brief survey.
\newblock {\em Bull. London Math. Soc}, 19(2):113--134, 1987.

\bibitem{bruveris2014geodesic}
Martins Bruveris, Peter~W Michor, and David Mumford.
\newblock Geodesic completeness for {S}obolev metrics on the space of immersed
  plane curves.
\newblock In {\em Forum of Mathematics, Sigma}, volume~2, page e19. Cambridge
  University Press, 2014.

\bibitem{bubenik2015metrics}
Peter Bubenik, Vin De~Silva, and Jonathan Scott.
\newblock Metrics for generalized persistence modules.
\newblock {\em Foundations of Computational Mathematics}, 15(6):1501--1531,
  2015.

\bibitem{bubenik2017interleaving}
Peter Bubenik, Vin De~Silva, and Jonathan Scott.
\newblock Interleaving and {G}romov-{H}ausdorff distance.
\newblock {\em arXiv preprint arXiv:1707.06288}, 2017.

\bibitem{bubenik2014categorification}
Peter Bubenik and Jonathan~A Scott.
\newblock Categorification of persistent homology.
\newblock {\em Discrete \& Computational Geometry}, 51(3):600--627, 2014.

\bibitem{burago2022course}
Dmitri Burago, Yuri Burago, and Sergei Ivanov.
\newblock {\em A course in metric geometry}, volume~33.
\newblock American Mathematical Society, 2022.

\bibitem{cagliari2001size}
Francesca Cagliari, Massimo Ferri, and Paola Pozzi.
\newblock Size functions from a categorical viewpoint.
\newblock {\em Acta Applicandae Mathematica}, 67:225--235, 2001.

\bibitem{carlsson2014topological}
Gunnar Carlsson.
\newblock Topological pattern recognition for point cloud data.
\newblock {\em Acta Numerica}, 23:289--368, 2014.

\bibitem{chacholski2020metrics}
Wojciech Chach{\'o}lski and Henri Riihimaki.
\newblock Metrics and stabilization in one parameter persistence.
\newblock {\em SIAM Journal on Applied Algebra and Geometry}, 4(1):69--98,
  2020.

\bibitem{chazal2009proximity}
Fr{\'e}d{\'e}ric Chazal, David Cohen-Steiner, Marc Glisse, Leonidas~J Guibas,
  and Steve~Y Oudot.
\newblock Proximity of persistence modules and their diagrams.
\newblock In {\em Proceedings of the twenty-fifth annual symposium on
  Computational geometry}, pages 237--246, 2009.

\bibitem{cohen2005stability}
David Cohen-Steiner, Herbert Edelsbrunner, and John Harer.
\newblock Stability of persistence diagrams.
\newblock In {\em Proceedings of the twenty-first annual symposium on
  Computational geometry}, pages 263--271, 2005.

\bibitem{de2018theory}
V~De~Silva, E~Munch, and A~Stefanou.
\newblock Theory of interleavings on categories with a flow.
\newblock {\em Theory and Applications of Categories}, 33(21):583--607, 2018.

\bibitem{dey2018computing}
Tamal~K Dey and Cheng Xin.
\newblock Computing bottleneck distance for 2-d interval decomposable modules.
\newblock In {\em 34th International Symposium on Computational Geometry (SoCG
  2018)}. Schloss Dagstuhl-Leibniz-Zentrum fuer Informatik, 2018.

\bibitem{edelsbrunner2002topological}
Edelsbrunner, Letscher, and Zomorodian.
\newblock Topological persistence and simplification.
\newblock {\em Discrete \& Computational Geometry}, 28:511--533, 2002.

\bibitem{grenander1996elements}
Ulf Grenander.
\newblock {\em Elements of pattern theory}.
\newblock JHU Press, 1996.

\bibitem{hamilton1982inverse}
Richard~S Hamilton.
\newblock The inverse function theorem of {N}ash and {M}oser.
\newblock {\em Bulletin of the American Mathematical Society}, 7(1):65--222,
  1982.

\bibitem{3236071}
Moishe~Kohan (https://math.stackexchange.com/users/84907/moishe kohan).
\newblock Relationship between distances on homogeneous spaces and their {L}ie
  groups.
\newblock Mathematics Stack Exchange.
\newblock URL:https://math.stackexchange.com/q/3236071 (version: 2019-05-22).

\bibitem{janelidze2001note}
George Janelidze and Gregory~M Kelly.
\newblock A note on actions of a monoidal category.
\newblock {\em Theory Appl. Categ}, 9(61-91):02, 2001.

\bibitem{johnson20212}
Niles Johnson and Donald Yau.
\newblock {\em 2-dimensional categories}.
\newblock Oxford University Press, USA, 2021.

\bibitem{kelly2006doctrinal}
G~Max Kelly.
\newblock Doctrinal adjunction.
\newblock In {\em Category Seminar: Proceedings Sydney Category Theory Seminar
  1972/1973}, pages 257--280. Springer, 2006.

\bibitem{kriegl1997convenient}
Andreas Kriegl and Peter~W Michor.
\newblock {\em The convenient setting of global analysis}, volume~53.
\newblock American Mathematical Soc., 1997.

\bibitem{lawvere1973metric}
F~William Lawvere.
\newblock Metric spaces, generalized logic, and closed categories.
\newblock {\em Rendiconti del seminario mat{\'e}matico e fisico di Milano},
  43:135--166, 1973.

\bibitem{lesnick2015theory}
Michael Lesnick.
\newblock The theory of the interleaving distance on multidimensional
  persistence modules.
\newblock {\em Foundations of Computational Mathematics}, 15(3):613--650, 2015.

\bibitem{Lurie2009}
Jacob Lurie.
\newblock {\em Higher topos theory}.
\newblock Princeton University Press, 2009.

\bibitem{mccrudden2000categories}
Paddy McCrudden.
\newblock Categories of representations of coalgebroids.
\newblock {\em Advances in Mathematics}, 154(2):299--332, 2000.

\bibitem{memoli2012some}
Facundo M{\'e}moli.
\newblock Some properties of {G}romov--{H}ausdorff distances.
\newblock {\em Discrete \& Computational Geometry}, 48:416--440, 2012.

\bibitem{michor2007overview}
Peter~W Michor and David Mumford.
\newblock An overview of the {R}iemannian metrics on spaces of curves using the
  hamiltonian approach.
\newblock {\em Applied and Computational Harmonic Analysis}, 23(1):74--113,
  2007.

\bibitem{needham2020simplifying}
Tom Needham and Sebastian Kurtek.
\newblock Simplifying transforms for general elastic metrics on the space of
  plane curves.
\newblock {\em SIAM journal on imaging sciences}, 13(1):445--473, 2020.

\bibitem{neeman2020metrics}
Amnon Neeman.
\newblock Metrics on triangulated categories.
\newblock {\em Journal of Pure and Applied Algebra}, 224(4):106206, 2020.

\bibitem{pareigis1977non}
Bodo Pareigis.
\newblock Non-additive ring and module theory. ii. c-categories, c-functors and
  c-morphisms.
\newblock {\em Publ. Math. Debrecen}, 24(3-4):351--361, 1977.

\bibitem{scoccola2020locally}
Luis~N Scoccola.
\newblock {\em Locally persistent categories and metric properties of
  interleaving distances}.
\newblock PhD thesis, The University of Western Ontario (Canada), 2020.

\bibitem{skraba2020wasserstein}
Primoz Skraba and Katharine Turner.
\newblock Wasserstein stability for persistence diagrams.
\newblock {\em arXiv preprint arXiv:2006.16824}, 2020.

\bibitem{srivastava2010shape}
Anuj Srivastava, Eric Klassen, Shantanu~H Joshi, and Ian~H Jermyn.
\newblock Shape analysis of elastic curves in euclidean spaces.
\newblock {\em IEEE transactions on pattern analysis and machine intelligence},
  33(7):1415--1428, 2010.

\bibitem{srivastava2016functional}
Anuj Srivastava and Eric~P Klassen.
\newblock {\em Functional and shape data analysis}, volume~1.
\newblock Springer, 2016.

\bibitem{stefanou2018dynamics}
Anastasios Stefanou.
\newblock {\em Dynamics on categories and applications}.
\newblock State University of New York at Albany, 2018.

\end{thebibliography}

\end{document}